\documentclass[12pt]{amsart}
\usepackage{amssymb, amsmath}
\usepackage{amsmath,amssymb,amsthm,mathrsfs,graphicx,url,mathtools}
\DeclareMathAlphabet{\mathpzc}{OT1}{pzc}{m}{it}
\usepackage{graphics}
\usepackage{epsfig}
\usepackage{amsxtra}
\usepackage{color}
\usepackage{amsxtra}

\usepackage[letterpaper,hmargin=1in,vmargin=1.2in]{geometry}

\numberwithin{equation}{section}

\newtheorem{thm}{Theorem}[section]

\newtheorem{lemma}[thm]{Lemma}

\newtheorem{prop}[thm]{Proposition}
\newtheorem{cor}[thm]{Corollary}

\theoremstyle{definition}

\newcommand{\defeq}{\stackrel{\rm{def}}{=}}
\newcommand{\zsharp}{\zeta^{\flat}}

\renewcommand{\Re}{\operatorname{\rm Re}\nolimits}
\renewcommand{\Im}{\operatorname{\rm Im}\nolimits}

\def \rank {\operatorname{rank}}

\def \dist {\operatorname{dist}}

\def \reg {\operatorname{reg}}
\def \tr {\operatorname{tr}}

\def \mcd {{\mathcal D}}

\def \restrict {\upharpoonright}

\def \mcd {{\mathcal D}}
\def \mch {{\mathcal H}}
\def \Real {{\mathbb R}}

\def \Sphere {\mathbb{S}}
\def \Complex {\mathbb{C}}
\def \Natural {{\mathbb N}}
\def \Sphere {{\mathbb S}}

\def \U+ {U_+}
\def \Integers {{\mathbb Z}}

\def \loc {\operatorname{loc}}
\def \Rzz {R_{00}}
\def \RVzz {R_{V_00}}
\def \RVz {R_{V_0}}
\def \RV {R_V}

\def \Rz {R_0}
\def \Vd {V^\#}
\def \pr{\mathcal{P}}
\def \przhat {\mathpzc{p}}

\def \zhat{\hat{Z}}

\def \mvzz {m_{V_0 0}}
\def \mV {m_V}

\def \Dl {D_l}

\def \spec{\sigma}
\def \chf{\chi_{I_{0}}}

\def \zconj {\zeta^\dagger}
\def \Lz {\Delta_0}
\def \Rd { {\mathbb R}^d}
\def \grz {\nabla_0}
\def \Ran {\operatorname{Ran}}
\newcounter{np} 
\newcommand*{\np}{\refstepcounter{np}\par\arabic{np}. }

\title
 [Schr\"odinger operators on cylinders]
{
Resonances for Schr\"odinger operators on infinite cylinders
and other products
}
   \author { T.J. Christiansen}
\address{Department of Mathematics,
University of Missouri,
Columbia, Missouri 65211, USA} 
\email{christiansent@missouri.edu}
\begin{document}

\begin{abstract}

 We study the resonances of Schr\"odinger operators
on the infinite product $X=\Real^d\times \Sphere^1$, where $d$ is odd,
$\Sphere^1$ is the unit circle, and the 
potential $V\in L^\infty_c(X)$.  This paper shows that at high energy,
resonances of the Schr\"odinger operator $-\Delta +V$ on $X=\Real^d\times \Sphere^1$
which are 
near the continuous spectrum 
are approximated by the resonances of $-\Delta +V_0$ on $X$, where 
the potential $V_0$ given by averaging $V$ over the unit circle.  These resonances
are, in turn, given in terms of the resonances  of a Schr\"odinger
operator on $\Real^d$  which lie in a bounded set.
If the potential is smooth, we obtain improved localization of the resonances,
particularly in the case of simple, rank one poles of the corresponding
 scattering 
resolvent on $\Real^d$.  In that case, we obtain the leading order correction for the 
location of the corresponding high energy resonances.
  In addition to direct results about the location of resonances, we show that at high energies away from the resonances, the resolvent of the model operator 
$-\Delta+V_0$ on $X$
approximates that of 
$-\Delta+V$ on $X$.  If
$d=1$, in certain cases this implies the existence of an asymptotic
expansion of solutions of the wave equation.
Again for  the special case of $d=1$, we obtain a resonant rigidity type result
for the zero 
potential among all real-valued potentials.  
\end{abstract}

\maketitle
\section{Introduction}\label{s:introduction}

We study the Schr\"odinger operator $-\Delta +V$ on the 
manifold $X=\Rd \times \Sphere^1$ with the product metric, where $d$ is odd,
$\Sphere^1$ is the unit circle, 
and $V\in L^\infty_c(X)$. In  the special case $d=1$ $X$ is the infinite 
cylinder $\Real\times \Sphere^1$.   We show that in the large 
energy limit, resonances near the continuous spectrum are well-approximated by
those of $-\Delta +V_0$, where $V_0$ is the average of $V$ over $\Sphere^1$:
$V_0(x)=\frac{1}{2\pi}\int_0^{2\pi}V(x,\theta)d\theta.$  By 
a separation of variables argument, these, in turn, are determined by the 
low energy resonances of 
the Schr\"odinger operator
$-\sum_{j=1}^d \frac{\partial^2}{\partial x^2_j}+V_0$  on $\Real^d$.
In the case of smooth potentials $V$, for simple 
rank one 
poles of the (scattering) resolvent of $-\sum_{j=0}^d\frac{\partial^2}{\partial x_j^2}+V_0$,
we find the leading-order corrections to the location of the corresponding
poles of the resolvent of $-\Delta+V$ on $X$.  Among other things, this
allows us to prove that no other smooth real-valued potential on $\Real \times \Sphere^1$ has 
the same resonances as the zero potential.  For potentials with 
$V_0\equiv 0$, we show the existence of 
large resonance-free regions. 
When $d=1$ and $V\in C_c^\infty(X;\Real)$ under certain hypotheses on the 
potential $V_0$ we are able to give an asymptotic expansion
of solutions of the wave equation.  For the case of $d=1$
we study a simple
example of a nontrivial potential $V$ with $V_0\equiv 0$ and locate 
some of the corresponding resonances.  Some of these results are 
reminiscent of Drouot's results for rapidly oscillating potentials on 
$\Real^d$, \cite{dro}.

  Let $\Delta \leq 0$ denote the 
Laplacian on $X=\Rd \times \Sphere^1$. For  
$V\in L^\infty_c(X)$ the Schr\"odinger operator 
$-\Delta +V$ has continuous spectrum $[0,\infty)$, with 
multiplicity which increases at each {\em threshold} $j^2$, $j\in \Natural_0$.
For $\Im \zeta>0$, set 
$\RV(\zeta)=(-\Delta+V-\zeta^2)^{-1}$.
This (scattering) resolvent has
a meromorphic continuation to $\hat{Z}$, the minimal Riemann surface for which 
$\tau_l(\zeta)\defeq ( \zeta^2 -l^2)^{1/2}$ is a single-valued analytic function for each
 $l\in \Natural_0$.  The resonances are poles of the resolvent $\RV(\zeta)$.
  We refer to the portion of $\zhat$ 
for which $\Im \tau_l(\zeta)>0$ for all $l\in \Natural_0$
as the {\em physical space}.  In this set $\RV$ is a bounded 
operator on $L^2(X)$, away from 
a discrete set of points which correspond to (square roots of) eigenvalues.
 For $l\in \Natural_0,$ $\rho>0$, denote by 
$B_{l}(\rho)$ the connected component of 
$\{ \zeta \in \zhat: |\tau_l(\zeta)|<\rho\}$ 
which nontrivially  intersects both the physical space and the set 
$\{\zeta\in \zhat:\Re \tau_0(\zeta)>0\}$. 
Using as 
the coordinate $\tau_l(\zeta)$, 
$B_{l}(\rho)$ is identified with the disk
of radius $\rho$  in the complex plane, centered at the 
origin, and this 
identification is
compatible with the complex structure of $\zhat\restrict_{B_l(\rho)}$ if 
$\rho<\sqrt{2l-1}$.  The point $\tau_l(\zeta)=0$ in $B_l(\rho)$ corresponds to the $l$th threshold.  We study the resonances of $-\Delta +V$ in $B_l(\rho)$, 
or $B_l(\alpha \log l)$, as $l\rightarrow \infty$.  Results of Section \ref{ss:corollary} show that these are the high energy resonances ``near" the continuous spectrum which have $ \Re \tau_0>0$.

For a function $V\in L^\infty_c(X)$
and $m\in \Integers$ define 
$$V_m(x) = \frac{1}{2\pi } \int_0^{2\pi} V(x,\theta)e^{-i m\theta}d\theta,$$
so that 
$V(x,\theta)=\sum_{m=-\infty}^\infty V_m(x)e^{i m \theta}$.  We use the 
notation $\Lz=\sum_{j=1}^d\frac{\partial ^2}{\partial x_j^2}$ for the Laplacian
on $\Rd$.

\begin{thm}\label{thm:poleexist}
Let $X=\Rd \times \Sphere^1$, $d$ odd, and let $V\in L^\infty_c(X)$ satisfy 
$\|V_m\|_{L^\infty}=O(m^{-\delta})$ for some $\delta$ with
$0<\delta\leq 1/2$.  Suppose
$\lambda_0\in \Complex  $, $\lambda_0\not = 0$ is a resonance of 
$-\Lz+V_0$ on $\Real^d$,
of multiplicity $\mvzz(\lambda_0)$.  Let $\rho\in \Real, $ $\rho>|\lambda_0|$.  
Then there are  $C_0>0$, $L>0$
so that for $l>L$, $l\in \Natural$
 there are exactly $2\mvzz(\lambda_0)$ resonances, when 
counted with multiplicity, of 
$-\Delta +V$ in the set
$$\{ \zeta \in B_{l}(\rho): |\tau_l(\zeta)-\lambda_0|<C_0 l^{-\delta/(\mvzz(\lambda_0))}\}.$$
\end{thm}
In this paper we refer to any pole of the resolvent as a resonance, including those which correspond to eigenvalues.
We remark that the second part of Theorem \ref{thm:polefree}, for which
$V$ is assumed to be smooth, implies an improved localization of the resonances for smooth
potentials.

The minimal assumption on a potential $V$ in most of this paper will be
\begin{equation}\label{eq:Vhyp}
V\in L^\infty_c(X)\; \text{and}\; \|V_m\|=O(m^{-\delta})\; \text{for some $\delta$ 
with $0<\delta \leq 1/2$}.
\end{equation}
Note that this imposes an assumption on $\delta$ as well, which we shall 
include we invoke the hypothesis (\ref{eq:Vhyp}).

The next theorem shows that, other than possible poles near the 
threshold, the poles as described above are all the poles in $B_l(\rho)$ for
sufficiently large $l$.  Here $\RVzz$ is the (scattering) resolvent of
$-\Lz+V_0$ on $\Real^d$, see Section \ref{ss:Es}, so that the poles of $\RVzz$ are the resonances of $-\Lz+V_0$.
\begin{thm}\label{thm:polefree}
Let $X=\Rd \times \Sphere^1$, $d$ odd, and suppose $V$ satisfies the hypothesis
(\ref{eq:Vhyp}).
Choose $\rho>0$ so that if
$\lambda_j$ is a pole of $\RVzz(\lambda)$, then $|\lambda_j|\not = \rho$.  Set 
\begin{equation*}
\Lambda_\rho= \{ \lambda_j \in \Complex:\; |\lambda_j|<\rho\; \text{and}\; \lambda_j \; \text{is a pole of
$\RVzz(\lambda)$ }\}.
\end{equation*}
Let $\epsilon'>0$ be so that $\epsilon'<\min\{|\lambda_j|:\; 
\lambda_j\in \Lambda_\rho,\; \lambda_j\not=0\}$.
Then there are $\tilde{C}$, $L>0$ so that for $l>L$, $l\in \Natural$, there are no resonances of $-\Delta +V$
in 
$$\{ \zeta \in B_{l}(\rho): \; |\tau_l(\zeta)|>\epsilon' \;\text{and}\;
|\tau_l(\zeta)-\lambda_j|>\tilde{C}l^{-\delta/\mvzz(\lambda_j)}\; 
\text{for all $\lambda_j\in \Lambda_\rho$} \}.$$
Moreover, if $V$ is smooth for perhaps larger $L$ and $\tilde{C}$ if
$l>L$ there are 
no resonances in 
$$\{ \zeta \in B_{l}(\rho): \; |\tau_l(\zeta)|>\epsilon'\; \text{and}\; |\tau_l(\zeta)-\lambda_j|>\tilde{C}l^{-2/(\mvzz(\lambda_j))}\; 
\text{for all $\lambda_j\in \Lambda_\rho$} \}.$$
In addition, if $\RVzz(\lambda)$ is analytic in a neighborhood of the 
origin, then there
 are no poles in $B_l(\epsilon')$ for $l$ sufficiently large.
\end{thm}

For Schr\"odinger operators on $\Real^d$, the behavior of the 
singularities of the resolvent at the origin is  delicate.
For example, notions of multiplicity of 
a resonance which agree at points away from the 
origin may differ at the origin.  These same sorts of issues arise
at thresholds in the case under study here, and 
accounts for the fact that this next theorem, which 
concerns resonances very near the thresholds,
is weaker than the previous ones.  

\begin{thm}\label{thm:0}
Let $V$ satisfy (\ref{eq:Vhyp})
and suppose the resolvent of $-\Lz+V_0$ on $\Real^d$
has a pole at $0$ of order $r>0$, and multiplicity $\mvzz(0)$.  
Then there are $C,\;L>0$ so
that  $-\Delta +V$ on $X$ has at least $2\mvzz(0)$ 
resonances, when counted with multiplicity, in $B_l(C l^{-\delta/r})$ when $l>L$,
$l\in \Natural$.  Moreover, there is an $\epsilon>0$ so
that $-\Delta+V$ has no poles in $B_l(\epsilon)\setminus B_l(C l^{-\delta/r})$
when $l>L$.  If $V\in C_c^\infty(X)$, then this 
can be improved to show that there is a $C_1>0$ so that 
$-\Delta+V$ has no poles in $B_l(\epsilon)\setminus B_l(C_1 l^{-2/r})$
when $l>L$.  Moreover, under
the hypothesis (\ref{eq:Vhyp}), if $r=1$ there are
exactly $2\mvzz(0)$ resonances of $-\Delta +V$ in $B_l(C l^{-\delta})$ for $l>L$.
\end{thm}
Suppose for the moment that $V_0$ is real-valued.  In this
case, it is well-known  that if $d=1$ the order of the 
pole of the resolvent of $-\frac{d^2}{dx^2}+V_0$ at $0$ cannot exceed $1$,
and if it is $1$, then $\mvzz(0)=1$.   If $d\geq 3$ is odd, then 
the order of the pole of the resolvent of $-\Lz+V_0$ at $0$ cannot exceed $2$.
For general $V$,  $r$,  the order of the pole at $0$, can be bounded
from above in terms of 
$\mvzz(0)$, and in case $d=1$, $\mvzz(0)$ can be bounded above by $r$.

It is of particular interest to understand poles of the resolvent 
$\RV$ near the 
physical region.  In Section \ref{ss:corollary} we show that there are 
large regions near the physical region that contain 
no resonances.  A consequence of those results is that large energy
resonances ``near'' the continuous spectrum and having $\Re \tau_0(\zeta)>0$ are contained in regions
of the form $B_l(\rho)$, where $\rho$ depends
on how near the continuous spectrum we wish to look.   In Section \ref{ss:corollary} we further justify our 
focus on the resonances in sets $B_l(\rho)$.

 Theorems \ref{thm:poleexist}, \ref{thm:polefree}, and \ref {thm:0},
combined with
results of Section \ref{ss:corollary}
yield the following corollary.  Here $d_{\zhat}$ is a distance
on $\zhat$, defined in Section \ref{ss:corollary}.  The boundary of the 
physical region corresponds to the continuous spectrum.  In the 
corollary, we use $\{\zsharp_j\}$ to denote a sequence of points in 
$\zhat$, to
distinguish from $\zeta_l$ which is used elsewhere to denote a 
particular mapping from
an open  subset of the complex plane into $\zhat$.
\begin{cor}\label{c:approachsequence}
Let $V\in L^\infty_c(X;\Real)$ satisfy (\ref{eq:Vhyp}).
Then
$\RV(\zeta)$ has a sequence $\{\zsharp_j\}_{j=1}^\infty$ of 
poles satisfying both $|\tau_0(\zsharp_j)|\rightarrow \infty$ as $j\rightarrow \infty$ and 
{\em } $d_{\zhat}(\zsharp_j,\text{physical region})\rightarrow 0$
as $j\rightarrow \infty$ if and only if $\RVzz(\lambda)$ has at least one
pole in $i[0,\infty)$. 
\end{cor}
 In particular, if $d=1$ by \cite[Theorem XIII.110]{ReSi} if $\int_X V\leq 0$ 
 then $\RV(\zeta)$ has such a sequence of poles approaching
the physical space.  In contrast,
if $V_0(x)\geq 0$ for all $x$ and $V_0$ is nontrivial, $\RV(\zeta)$ does
not have such a sequence of poles.  Note that
for any fixed $k_0\in \Natural$, $|\tau_0(\zsharp_j)|\rightarrow
\infty$ as $j\rightarrow \infty$ if and only if $|\tau_{k_0}(\zsharp_j)|\rightarrow \infty$ as $j\rightarrow \infty$.  We remark that we could prove an 
analog of Corollary \ref{c:approachsequence} for complex-valued potentials
as well.

If we enlarge the region centered at the threshold $l^2$ with increasing
 $l$, we have less 
fine localization of the resonances, see Theorem \ref{thm:biggerdisk}.
 However, when $V_0$, the average of the potential, 
is identically $0$, we can get 
a larger resonance-free region.  The difference in
the next result for $d=1$ and $d\geq 3$ is due to the fact that 
the resolvent of $-\frac{d^2}{dx^2}$ on $\Real$ has a 
pole at the origin, but that of $-\Lz$ on $\Real^d$ for $d\geq 3$ odd 
does not.
\begin{thm}\label{thm:V0iszero}
Let $V\in L^\infty_c(X)$ satisfy (\ref{eq:Vhyp}),
and suppose $V_0\equiv 0$.  If $d=1$ there 
are  $\alpha,c_0>0$ so that for $l\in \Natural$ sufficiently large there are no
resonances of $-\Delta+V$ in the set 
$\{\zeta \in B_l(\alpha \log l): |\tau_l(\zeta)|>c_0/l^{\delta}\}$.
If $d\geq 3$ is odd, there
is an $\alpha>0$ so that for $l$ sufficiently large
there are no resonances of  $-\Delta+V$ in the set
$B_l(\alpha \log l)$. 
\end{thm}
There is a sense in which this theorem is sharp; see Proposition
\ref{p:exlogl} for 
a computation for the case $d=1$ with
the potential $V(x,\theta)=2\chf(x)\cos \theta $,
where $\chf$ is the characteristic function of the interval $[-1,1]$.

We can find the leading correction term 
for high energy resonances of $-\Delta +V$ which correspond to simple 
resonances of $-\Lz+V_0$.  In the next theorem, $\grz$ is the 
gradient on $\Rd$, so that $\grz f=\left( \frac{\partial}{\partial x_1}f,
\frac{\partial}{\partial x_2}f, ...,\frac{\partial}{\partial x_d}f\right)$.
\begin{thm}\label{thm:leadcorrection}
Let $X=\Rd \times \Sphere^1$, $d$ odd, $V\in C_c^\infty(X)$, and suppose 
$\lambda_0\in \Complex $ is a simple pole of the scattering
resolvent $\RVzz$ of $-\Lz+V_0$ on $\Real^d$, and that
the residue of $\RVzz$ at $\lambda_0$ has 
rank $1$.  Suppose
\begin{equation}\label{eq:singform}
\RVzz(\lambda) -\frac{i}{\lambda- \lambda_0} u\otimes u
\end{equation}is analytic 
near $\lambda =\lambda_0$.  Let $\rho>|\lambda_0|$.  Then there
are $\epsilon,\;L>0$ so that for $l>L$ there are
exactly two poles of $R_V(\zeta)$, when counted with multiplicity,
in $\{\zeta \in B_{l}(|\lambda_0|+1): |\tau_l(\zeta)-\lambda_0|<\epsilon\}$,
and each pole of $R_V(\zeta)$ in this set satisfies
$$\tau_l(\zeta)=\lambda_0- \frac{i}{4l^2} \sum_{k\not = 0}\frac{1}{k^2} \int_\Real
  \left( k^2V_{-k}V_{k}u^2+ \left(\grz V_{-k}\cdot \grz V_k\right) u^2 \right)(x)dx  
   +O(l^{-3}).$$
\end{thm}
We note that the normalization of the singularity in (\ref{eq:singform}) is
chosen so that if $V$ is real-valued and $\lambda_0\in i[0,\infty)$, 
then $u$ is real-valued.  Proposition \ref{p:exnearthreshold} shows
that the leading correction may be rather different for a non-smooth potential
by considering the special case of the potential on $\Real \times \Sphere^1$
 given by
$V(x,\theta)=2\cos \theta \chf(x)$, where $\chf$ is the characteristic
function of the interval $[-1,1]$.

If $V_0\in L^\infty_c(\Real^d;\Real)$
 and the operator  $-\Lz+V_0$ on $L^2(\Rd)$
 has a simple negative eigenvalue
$-\beta^2$, then this
negative eigenvalue corresponds to a simple pole of $\RVzz$  on the positive
imaginary
axis at $i|\beta|$, and the residue has rank $1$.  
By Theorem \ref{thm:poleexist} (or Corollary \ref{c:approachsequence}), in this 
case $\RV$ has a sequence of poles approaching the physical space.
If $V\in C_c^\infty(X;\Real)$, the poles approach the physical space
very rapidly.
\begin{thm}\label{thm:polesequence} Suppose $V\in C^\infty_c(X;\Real)$ and $\lambda_0\in \Complex$
is a simple pole of $\RVz(\lambda)$ with $\Re \lambda_0=0$,
with residue of $\RVz$ at $\lambda_0$ having 
rank one.  Then there
is an $\epsilon>0$ so that if $\{ \zsharp_l\}_{l=L}^\infty \subset \zhat$
is a sequence of poles of $\RV$ with $\zsharp_l\in B_l(|\lambda_0|+1)$,
$|\tau_l(\zsharp_l)-\lambda_0|<\epsilon$, then $\Re \tau_l(\zsharp_l)=O(l^{-N})$
for any $N$.  In particular this implies, if $\Im \lambda_0> 0$,
$d_{\zhat}(\zsharp_l,\text{physical region})=O(l^{-N})$.
\end{thm}
Proposition \ref{p:exnearthreshold} demonstrates the necessity of the smoothness hypothesis, at least for $d=1$.

This paper was initially motivated by the case $d=1$, as $\Real \times \Sphere^1$ provides a particularly simple example of a manifold with infinite cylindrical
ends and as such provides a testing ground for studying resonances for Schr\"odinger
operators on such manifolds.
  Most of the proofs of the preceding theorems are essentially
identical for any odd dimension of the factor $\Real^d$, so we have included the more 
general results.  However, Theorems \ref{thm:0isrigid} and 
\ref{thm:waveexp} are particular
to the $d=1$ case. 

As a corollary of Theorems \ref{thm:poleexist}, \ref{thm:0},
and \ref{thm:leadcorrection},
we get in the case $d=1$ a uniqueness-type result for the zero potential
among smooth real-valued potentials.
\begin{thm}\label{thm:0isrigid}
Let $V\in C_c^\infty(\Real \times \Sphere^1;\Real)$.  
Suppose for each $\rho>0$ there is a sequence
$\{l_j\}_{j=1}^\infty = \{l_j(\rho) \}_{j=1}^\infty\subset \Natural$ with
$l_j\rightarrow \infty$ when $j\rightarrow \infty$ and so that 
in $B_{l_j}(\rho)$
the resonances of $-\Delta +V$ and $-\Delta$ on $X=\Real\times \Sphere^1$
 are the same.  Then
$V\equiv 0$.
\end{thm}
This result is not true if we omit the hypothesis that $V$ is real-valued.
For example, for $V_1\in C_c^\infty(\Real)$ set $V(x,\theta)=V_1(x)e^{i\theta}$.
 Then the operators
$-\Delta +V$ and $-\Delta$ have the same resonances; see \cite[Section 4]{chr04} or \cite{autin}.  
This example can be easily generalized.

As part of our study of the distribution of resonances, we prove that,
in a suitable sense, near the 
physical region of $\zhat$, $R_V$ is well-approximated by $\RVz$ away from the poles of $\RVz$; see Proposition \ref{p:Rest} and  Lemma \ref{l:bdryphys}.  
In the case $d=1$, this
 and results of \cite{cd2} give a wave expansion, Theorem \ref{thm:waveexp}.

Let $X=\Real\times \Sphere^1$,
$V\in C_c^\infty(X;\Real)$, and suppose $-\Delta+V$
has finitely many eigenvalues $\mu_1,\; \mu_2,\dots,\; \mu_J$, repeated with multiplicity, with 
associated orthonormal eigenfunctions
$\{ \eta_j\}$, so that $(-\Delta+V)\eta_j=\mu_j\eta_j$.  Let $u$ satisfy
\begin{equation}\label{eq:waveq}
\frac{\partial^2}{\partial t^2}u -\Delta u +Vu=0,\; (u,u_t)\restrict_{t=0}=(f_1,f_2)\in C_c^\infty(X)\times C_c^\infty(X).\end{equation}
\begin{thm}\label{thm:waveexp}
Let $X=\Real \times \Sphere^1$,  $V,\;f_1,f_2\in C_c^\infty(X)$, $V$ be real-valued, and suppose $-\frac{d^2}{dx^2}+V_0$ on $\Real$ has no negative eigenvalues and no resonance at $0$. Let $u$ be
the solution of (\ref{eq:waveq}) on $[0,\infty)\times X$.  Then for each $k_0\in \Natural$ we can write
$u(t) = u_e(t) + u_{thr,k_0}(t) + u_{r,k_0}(t)$, 
where
\begin{multline}\label{eq:ue1}
 u_e(t,x,\theta) = \sum_{\substack{\mu_j\in \spec_{p}(-\Delta+V) \\
\mu_j\not =0}} \eta_j(x,\theta) \left( \cos((\mu_j)^{1/2} t) \langle f_1,\eta_j\rangle + 
 \frac{\sin( (\mu_j)^{1/2} t)}{(\mu_j)^{1/2} } \langle f_2,\eta_j\rangle \right)
\\ + \sum_{\substack{\mu_j\in \spec_{p}(-\Delta+V )\\\mu_j =0}} \eta_j(x,\theta) \left( \langle f_1,\eta_j\rangle + t \langle f_2,\eta_j\rangle \right)
\end{multline}
and
\begin{equation*}
u_{thr,k_0} (t,x,\theta)= b_{0,0,+}(x,\theta) +
 \sum_{k=0}^{k_0-1}t^{-1/2-k} 
\sum_{j=1}^\infty (e^{it j}b_{j,k,+}(x,\theta)+ e^{-it j}b_{j,k,-}(x,\theta))  
\end{equation*}
for some  $b_{j,k,\pm}\in \langle r \rangle ^{1/2+2k+\epsilon}L^2(X)$.
For any $\chi \in C_c^\infty(X)$ there is a constant $C$ so that
$$\sum_{j=1}^\infty \|\chi  b_{j,k,\pm}\|_{L^2(X)}<C,\; k=0,1,2,...,k_0-1$$
and
\[
 \|\chi u_{r,k_0}(t)\|_{L^2(X)} \le C t^{-k_0}\; \text{for $t$ sufficiently large}.
\]
\end{thm} 
The assumption that  $-\frac{d^2}{dx^2}+V_0$ on $\Real$ has no negative eigenvalues and no resonance at $0$ means, by Theorem \ref{thm:polefree}, that 
$R_V$ has at most finitely many poles on the boundary of the 
physical space.  In particular, this means at most finitely many
eigenvalues of $-\Delta +V$, so that the 
sum in $u_e$ is finite.  Further, there are at most finitely many 
poles at thresholds, and this implies
via results of \cite{cd2} that at most finitely many of the $b_{j,0,\pm}$ are
nonzero.  

If  $-\frac{d^2}{dx^2}+V_0$ on $\Real$ has one
or more negative eigenvalues, it seems plausible that 
there is an asymptotic expansion of
solutions of the wave equation on compact sets. 
 Since in this case by Theorem \ref{thm:polesequence} 
the resolvent $\RV$ may have a sequence of poles rapidly approaching,
but not lying in, the continuous 
spectrum, such an expansion would need to take these into account and is more
complicated--see e.g. \cite{TaZw} for an expansion in a Euclidean
scattering  setting 
with resonances approaching the continuous spectrum.  In our setting proving the existence of such an expansion may use
techniques similar to those 
of \cite{cd2}  but does not follow directly from the results of \cite{cd2}.
Proving this is outside the scope of this paper.

In this paper we have, for simplicity, limited ourselves to the case of 
Schr\"odinger operators on $\Rd \times \Sphere$.  However, many of our
results for $L^\infty$ potentials hold as well for Schr\"odinger 
operators with Dirichlet or Neumann boundary 
conditions on $\Real^{d-1}\times(0,\infty)\times \Sphere$ or on $\Rd \times (0,\pi)$.

\subsection{Relation to previous work}
This paper was inspired in part by two different sets of papers.  
 The first is the paper 
\cite{dro}
of Drouot, which studies the distribution of resonances of Schr\"odinger
operators $-\Lz +V_\epsilon$ on $\Real^d$, $d$ odd.  Here 
$V_\epsilon(x)=V_0(x) +\sum_{k\in \Integers^d,k\not = 0} V_k(x)e^{ik\cdot x/\epsilon}$, 
$x\in \Real^d$.  Drouot shows in quantitative ways 
that in the limit $\epsilon \downarrow 0$, 
resonances of $-\Lz +V_\epsilon$ near the 
continuous spectrum are well-approximated by those of $-\Lz +V_0$.
In addition, he proves some refinements related, for example,
 to the leading order 
correction of the positions of the resonances.
Theorems \ref{thm:poleexist}, \ref{thm:polefree}, \ref{thm:0}, \ref{thm:V0iszero},
and \ref{thm:leadcorrection}, as well as some computations in 
Section \ref{s:example},
  are inspired by results in \cite{dro}.  However,
the proofs are quite different.  In part, this is because the different
setting require different techniques.  Additionally, Drouot's results 
come mainly from studying regularized determinants.  While this
has the potential of producing in some instances  more refined results than we obtain 
here, it requires a 
substantial amount of technical work.  We have chosen instead to mostly
avoid determinants, or to work only with determinants of operators of 
the type $I+F$, where $F$ is finite rank.  Instead, we use 
an operator Rouch\'e Theorem of \cite{go-si}.   In some 
places this may allow for sharper results than 
could be obtained by using a regularized determinant.  We note in addition that
in the setting of \cite{dro}, the resonances lie on the complex 
plane, while for us, the resonances lie on Riemann surface which is a countable
but infinite cover of the complex 
plane, with infinitely many branch points.  This 
means that some of the techniques used in \cite{dro} cannot be applied here.

A less direct source of inspiration is  work done on the 
distribution of eigenvalues of the Schr\"odinger operator $-\Delta_{\Sphere^{n}} +W$ 
 on the sphere 
$\Sphere^{n}$ (and certain other compact manifolds), $n\geq 2$, see e.g. 
\cite{Wei,Wid}.
In this setting, eigenvalues of the Schr\"odinger operator occur in bands.
   Roughly speaking, these
authors show that a suitable ``average'' of the potential $W$ can be used to 
obtain information about the distribution of high-energy eigenvalues of the Schr\"odinger
operator within these bands. This 
average is over closed geodesics, rather than over all of $\Sphere^n$.
 Of course, our function $V_0(x)$ is
the average of the potential $V$ over the cross section $\Sphere^1$,
the unique closed geodesic on $\Sphere^1$.

This paper was originally motivated by the $d=1$ case, which gives $X=\Real
\times \Sphere^1$, a manifold with an infinite cylindrical end.
The spectral and scattering theory of manifolds with infinite cylindrical
ends has been studied in, for example, \cite{gol,gui,tapsit}.  
There is a large literature studying the existence of eigenvalues and,
in certain settings, the locations of resonances for such manifolds and 
the related problems of waveguides which have a ``one-dimensional
infinity'' as our $d=1$ case does; see e.g. the paper \cite{LeMa}
or the recent monograph \cite{ExKo}
and references therein.  This monograph also includes some results for manifolds with 
``higher-dimensional infinity.''
Many of these results focus on low-energy eigenvalues or resonances.
We mention the
 papers \cite{ChZw1,Par,Sub,chr04,Edw,cd} which are more 
directly connected with high-energy 
behavior.

\subsection{Ideas from the proofs}
Our starting point for the study of resonances of $-\Delta +V$
is an identification of the
resonances with the points $\zeta$ for which the operator
$I+(V-V_0)\RVz(\zeta)\chi$ has nontrivial null space.
Here
$R_W(\zeta)$ is the meromorphic continuation of the resolvent of
$-\Delta+W$, and $\chi\in L^\infty_c(X)$ satisfies $\chi V=V$ and is, for
convenience, chosen independent of $\theta$. By separating
variables, we can
understand $\RVz$ in terms of the resolvent of
$-\sum_{j=1}^d\frac{\partial^2}{\partial x_j^2}+V_0(x)$ on $\Real^d$.

We use two well-known and related properties of the resolvent of $-\sum_{j=1}^d \frac{\partial^2}{\partial x_j^2}+V_0(x)$ on $\Real^d$.   One is the
 estimate $\| \tilde{\chi} (-\sum_{j=1}^d\frac{\partial^2}{\partial x_j^2}+V_0 -(\lambda+i0)^2)^{-1}\tilde{\chi} \|=O(|\lambda |^{-1})$ as $\lambda \rightarrow \infty$  for 
 $\lambda \in \Real$ and $\tilde{\chi} \in L^{\infty}_c(\Real^d)$.  The second is
the existence of a logarithmic resonance-free neighborhood of the real axis.

An immediate consequence of this second fact and the fact that the distance between
thresholds of
our operator $-\Delta +V$ on $X$ increases at high energy is that if
$V=V_0$, at high energy near the thresholds the resonances of $-\Delta+V_0$
are determined by low energy resonances of $-\sum_{j=1}^d\frac{\partial^2}{\partial x_j^2}+V_0$
on $\Real^d$.  Moreover, using these facts and an operator Rouch\'e theorem of \cite{go-si}, we are
able to show that at high energy near
the thresholds the zeros of
$I+(V-V_0)\RVz \chi$ are approximated by the poles of $\chi \RVz \chi$.
These ideas underly the proofs of  the $L^\infty$ results of Theorems
\ref{thm:poleexist}-\ref{thm:0} and \ref{thm:V0iszero}.  They
are also central to the  proofs of the smooth versions of these results and of Theorem
\ref{thm:leadcorrection}, although these proofs
require some additional study of the resolvent of $-\Delta +V_0$
when $V_0$ is smooth.

\subsection{Organization}
In Section \ref{ss:Es} we recall some results from Euclidean scattering
and show that the resolvent of $-\Delta+V$ on $X$ has a meromorphic continuation to
$\zhat$.  (We note that this latter is known; see Section 
\ref{ss:Es} for references.)
 We define the multiplicity of a pole of the resolvent,
 and give several
useful indentities  involving it in Section \ref{ss:mpandGS}.  In addition,
this section  introduces some notation
and results related to the operator Rouch\'e Theorem of \cite{go-si}.  
With these preliminaries we prove Theorems \ref{thm:poleexist}
and
 \ref{thm:polefree} in the case of an $L^\infty$ potential $V$, using 
results from \cite{go-si}.
Section \ref{ss:corollary} contains more discussion of the Riemann surface
$\zhat$ and shows the existence of 
  resonance-free regions which are, at high 
energy, near the physical region and away from thresholds.  This provides
the missing pieces of the proof of Corollary \ref{c:approachsequence}.
Combining these with the resolvent estimates of Section \ref{s:thms1and2}
and results of \cite{cd2}
proves Theorem \ref{thm:waveexp}.

Section \ref{s:smoothprelims} contains
 preliminary computations which are needed to refine our results
for smooth potentials.  The smooth case of Theorem \ref{thm:polefree} is 
proved with techniques similar to that of the $L^\infty$ result, but using
in addition results of Section \ref {s:smoothprelims}.  

In Section \ref{s:smoothresults} we prove Theorems \ref{thm:leadcorrection} and \ref{thm:polesequence}.  We do this using Fredholm
determinants, but determinants
of the form $\det(I+F)$, where $F$ is a finite rank operator. 
Theorem \ref{thm:0isrigid} follows rather directly from the earlier results.
 Finally,
in  Section \ref{s:example}, in
the case $d=1$ we give approximations of some of the high
energy resonances for a particularly simple potential which has $V_0\equiv 0$
and which is not smooth.

\section{Notation and conventions}

On $X=\Rd\times \Sphere^1$ we use the coordinates $(x,\theta)$ or $(x',\theta')$, with $x,x'\in \Rd$ and  $\theta,\theta'\in [0,2\pi)$.  

Throughout the paper, $V\in L^\infty_c(X)$, $l\in \Natural_0$, and the 
dimension $d$ of $\Rd$ is odd.  We use $C$ to stand for a positive
constant the value of which may change without comment.

Suppose $A$ and $B$ are linear operators on domains in $L^2(\Rd)$ and $L^2(\Sphere^1)$, respectively, and 
are given by $(Af)(x)=\int_{\Rd} A(x,x')f(x')dx'$ and 
$(Bg)(\theta)=\int_0^{2\pi} B(\theta,\theta')g(\theta')d\theta$.
Then $A$ and $B$ give rise to linear operators on 
domains in $L^2(X)$, which
we again denote by $A$ and $B$, and which are given 
by 
$(Ah)(x,\theta)=\int_{\Rd}A(x,x')h(x',\theta)dx'$ and 
$(Bh)(x,\theta)=\int_{\Rd}B(\theta,\theta')h(x,\theta')d\theta'$.

For $f,g\in L^2(\Rd)$, the operator $f\otimes g:L^2(\Rd)\rightarrow L^2(\Rd)$ is
defined via $((f\otimes g )h)(x)=f(x)\int_{\Rd}g(x')h(x')dx'$. 
 If $f,g\in L^2(X)$, the operator $f\otimes g$ on $L^2(X)$ is defined analogously.

We list some repeatedly used notation for the convenience of the reader.
\begin{itemize}
\item $\Lz=\sum_{j=1}^d\frac{\partial^2}{\partial x_j^2}$ is the Laplacian on 
$\Rd$, and $\Delta=\sum_{j=1}^d\frac{\partial^2}{\partial x_j^2}+\frac{\partial^2}{\partial \theta^2}$ is the
Laplacian on $X$.
\item $V_m(x)=\frac{1}{2\pi}\int_0^{2\pi} V(x,\theta)e^{-im\theta}d\theta$ for $m\in \Integers$
\item $\Vd=\Vd(x,\theta)=V(x,\theta)-V_0(x)$
\item $B_l(\rho)$ and $\Dl(\lambda_0,\rho)$ are an open sets in $\zhat$, defined in Sections \ref{s:introduction} and \ref{s:thms1and2}, respectively.
\item $\RV$ is the (scattering) resolvent of $-\Delta +V$ on $X$; see
Section \ref{s:resolventV}.
\item $\RVzz$ is the (scattering) resolvent of $-\Lz+V_0$ on $\Rd$; see
Section \ref{ss:Es}.
\item $\mV(\zeta_0)$ is the multiplicity of $\zeta_0\in \zhat$ as a pole
of $\RV$; see (\ref{eq:mV}).
\item $\mvzz(\lambda_0)$ is the multiplicity of $\lambda_0\in \Complex$
as a pole of $\RVzz$; see (\ref{eq:mvzz}).
\item $\zeta_l:\{z\in \Complex: |z|<\sqrt{2l-1}\}\rightarrow B_l(\sqrt{2l-1})\subset \zhat$ is  the (local) inverse of $B_l(\sqrt{2l-1})\ni\zeta\mapsto \tau_l(\zeta)\in\{z\in \Complex: |z|<\sqrt{2l-1}\}\subset \Complex. $
\end{itemize}

\section{Odd-dimensional Euclidean scattering and continuation of the resolvent}
\label{ss:Es}
We begin by fixing some notation and recalling some well-known facts from
Euclidean  scattering theory.  We then use these to give a self-contained proof that the resolvent of $-\Delta+V$ on $X$ has a meromophic continuation 
to $\zhat$.

\subsection{The Euclidean resolvent}
Let $V_0\in L^\infty_c(\Rd)$, $d$ odd,
 and set $\RVzz(\lambda)=\left( -\Lz+V_0-\lambda^2\right)^{-1}:L^2(\Real^d)\rightarrow L^2(\Real^d)$ when $\Im \lambda>0$.  The $0$ in the second place in
the subscript is to help us think of this as a model operator, as we 
shall see.  We shall later use the
explicit expression for the resolvent as an integral when $d=1$, 
$f\in L^2(\Real)$ and $\Im \lambda>0$:
\begin{equation}\label{eq:RzzExplicit}
(\Rzz(\lambda)f)(x)=\frac{i}{2\lambda} \int e^{i\lambda|x-x'|}f(x')dx'\; 
 \; \text{for $d=1$}.
\end{equation}
From this we can see immediately that if $\chi \in L^\infty_c(\Real)$,  $\chi \Rzz(\lambda)\chi$
has a meromorphic continuation to $\Complex \setminus \{0\}$. 
The same is true for $d\geq 3$ is odd:
if $\chi\in L^\infty_c(\Rd)$, $\chi \Rzz(\lambda)\chi$
has an analytic continuation to the complex plane, see 
\cite[Theorem 3.3]{DyZw}.  In higher dimensions, the
 Schwartz kernel is given in terms of a Hankel function.
 It is well-known, see \cite[Theorem 3.8]{DyZw}
that if $V_0,\chi\in L^\infty_c(\Rd)$, then
$\chi \RVzz(\lambda)\chi$ has a meromorphic continuation to the complex plane.
Alternatively, restricting the domain and enlarging the range,
 $\RVzz(\lambda):L^2_c(\Rd)\rightarrow H^2_{\loc}(\Rd)$ has a meromorphic
continuation to $\Complex$.

The following lemma is well-known, but we include it for completeness, as it is crucial for our arguments.
\begin{lemma}\label{l:Rvzzresfree} Let $V_0,\chi\in L^\infty_c(\Rd)$.  Then there are
constants $C_0, C_1>0$ so that $\chi \RVzz(\lambda)\chi $ is analytic in 
$\{ \lambda\in \Complex: |\Re \lambda|>C_0,\; \Im \lambda > - C_1 \log (1+|\Re \lambda|)\}$.  
Moreover, in this region $\| \chi \RVzz(\lambda)\chi \|
= O(|\lambda|^{-1})$.
\end{lemma}
\begin{proof}
Without loss of generality, we may assume $\chi V_0=V_0$. 
Then
$$\chi \RVzz(\lambda)\chi = \chi \Rzz(\lambda)\chi (I+V_0\Rzz(\lambda)\chi)^{-1}.$$
Since from (\ref{eq:RzzExplicit}) when $d=1$ or
\cite[Theorem 3.1]{DyZw} when $d\geq 3$, 
there is a $C>0$ so that $\|V\Rzz(\lambda)\chi\|\leq C e^{C (\Im \lambda)_-}/|\lambda|$, where $(\Im \lambda)_-=\max(0, -\Im \lambda)$, the lemma follows immediately.
\end{proof}


\subsection{The resolvent of $-\Delta +V$ on $X$ and the Riemann surface $\zhat$}
\label{s:resolventV}
Recall that when $d=1$ $X$ is a manifold with infinite cylindrical ends.
For a manifold with infinite cylindrical ends, the space to which the 
resolvent of a Schr\"odinger operator continues is determined by the distinct
eigenvalues of the Laplacian on the cross-section of the end(s).  
Here that means $\{j^2\}_{j\in \Natural_0}$, since this is the set of (distinct)
eigenvalues of $-\frac{d^2}{d\theta^2}$ on $\Sphere^1$.  As we show
below, the resolvent for $-\Delta+V$ on $\Real^d\times\Sphere^1$
has a meromorphic continuation to the same space as that of the resolvent
of $-\Delta +V$ on $\Real\times \Sphere^1$, provided $d$ is odd.

For $j\in \Natural_0$ and $\zeta \in \Complex,\ \Im \zeta >0$,
set $$\tau_j(\zeta)\defeq(\zeta^2-j^2)^{1/2}$$
with $\Im \tau_j(\zeta)>0$. Set $\tau_{-j}(\zeta)=\tau_j(\zeta)$ if 
$j\in \Natural$.

The Riemann surface $\zhat$
is defined to be the minimal Riemann surface on which for 
each $j\in \Natural_0$
 $\tau_j$ is a single-valued analytic function on $\zhat$. We briefly 
describe its construction and some of 
 its properties.
Note that $\tau_0(\zeta)=\zeta $ for $\zeta$ in the upper half plane,
and this has, of course, an analytic continuation to $\Complex$.
Now $\tau_1(\zeta)=\tau_{-1}(\zeta)$ is an analytic function of 
$\zeta \in \Complex\setminus\left( (-\infty,1]\cup [1,\infty)\right)$, and there is a minimal Riemann surface $\hat{Z}_1$ so that $\tau_1$ extends
analytically to  $\hat{Z}_1$.  This is a double cover of $\Complex$, ramified at the points $\pm 1$.
This process can be repeated for each $j\in \Natural$, resulting in a minimal Riemann surface
$\zhat$ on which $\tau_j$ is analytic for each $j\in \Natural_0$.
We define a projection $\przhat:\zhat\rightarrow \Complex$ as
follows. For $\zeta$ in the
physical space, identified with the upper half plane, $\przhat(\zeta)=\zeta$,
and $\przhat$ is in general the analytic continuation of this function.
  Then $\zhat$ has infinitely many ramification points 
which project under $\przhat$ to $j \in\Integers \setminus\{0\}$.
  We call the
set $\{\zeta \in \zhat:\; \Im \tau_j(\zeta)>0\; \text{for all }\; j\in \Natural_0\}$
the {\em physical space,}
or {\em physical region.}  
For further discussion
of this Riemann surface, see \cite[Section 6.6]{tapsit}. 

We shall say that a point $\zeta_0\in \zhat$ {\em corresponds to a 
threshold} if $\tau_0(\zeta_0)\in \Integers$.  Note that with this
definition, all the ramification points of $\zhat$ correspond to 
thresholds.  In addition, the set of points corresponding to thresholds 
includes those points projecting to $0$.  These might naturally also 
be considered ramification points of $\zhat$, as in some sense by choosing
$\zeta^2$ to originally be our spectral parameter we have already made
the cuts corresponding to the $0$ threshold.

In order to separate variables below, we introduce the orthogonal projections 
$\pr_k:L^2(X)\rightarrow L^2(X)$ defined for $k\in \Integers$ by:
\begin{align*}(\pr_k f)(x,\theta)& =
\frac{1}{2\pi}\int_0^{2\pi}  f(x,\theta')\left(e^{ik(\theta-\theta')}+ 
e^{-ik(\theta-\theta')}\right)d\theta'\; \text{if }\;k\in \Natural\\
(\pr_0 f)(x,\theta)& =\frac{1}{2\pi}\int_0^{2\pi}f(x,\theta')\;d\theta'.
\end{align*}
We shall use these throughout the paper.

Let $V\in L^\infty_c(X)$.  
For  $\zeta \in \Complex$ with $\Im \zeta >0$, set 
$R_V(\zeta)=(-\Delta +V-\zeta^2)^{-1}$.  Consider first the special case where $V\in L_c^\infty(X)$ is
independent of $\theta$.  Then $V=V_0$, and we can think of $V_0$ as an element
of $L^\infty_c(X)$ or of  $L^\infty_c(\Rd)$.
In this special case we can separate
variables to obtain
\begin{equation}\label{eq:RVzsep}
\RVz(\zeta)=\sum_{k=0}^\infty 
\RVzz(\tau_k(\zeta))\pr_k.
\end{equation}
The explicit expression (\ref{eq:RVzsep})
for $\RVz$ using separation of variables
shows that if $\chi \in L^\infty_c(X)$, 
$\chi \RVz \chi$  and $\RVz:L^2_c(X)\rightarrow H^2_{\loc}(X)$ have  meromorphic
continuations to $\zhat$.  In fact,
the same is true for $\chi \RV \chi$ and $\RV$ for
general $V\in L^\infty_c(X)$.
This is well known, at least 
when $d=1$, see \cite{gol,gui,tapsit}, though we sketch a proof below, 
valid for all odd $d$.

If $\zeta\in \Complex$, $\Im \zeta>0$, then
$$(-\Delta+V-\zeta^2)R_0(\zeta)= I+V R_0(\zeta).$$
Multiplying by a function $\chi \in L_c^\infty(X)$ with $\chi V=V$,
$$(-\Delta+V- \zeta^2)R_0(\zeta)\chi = \chi( I+V R_0(\zeta)\chi),$$
implying that
\begin{equation}\label{eq:req1}
\chi R_0(\zeta)\chi= \chi  R_V(\zeta)\chi( I+V R_0(\zeta)\chi)
\end{equation}
or 
\begin{equation}\label{eq:req2}
\chi R_V(\zeta)\chi = \chi R_0(\zeta) \chi (I+V R_0(\zeta)\chi)^{-1}\;
\text
{and}\; (I+VR_0(\zeta)\chi)^{-1}= I-VR_V(\zeta)\chi;
\end{equation}
compare \cite[(2.2.15)-(2.2.16)]{DyZw}.  
Likewise, writing 
\begin{equation}
\Vd\defeq V-V_0
\end{equation}
we find, making the additional hypothesis that $\chi \Vd=\Vd$,
\begin{equation}\label{eq:V0asmodel}
\chi\RVz(\zeta)\chi=\chi\RV(\zeta)\chi(I+\Vd \RVz(\zeta)\chi)\; \text{and}\;
(I+\Vd \RVz(\zeta)\chi)^{-1}=I-\Vd \RV(\zeta)\chi.
\end{equation}
Each of these is helpful.
Since $V R_0(\zeta)\chi:L^2(X)\rightarrow L^2(X)$ is compact and 
has a meromorphic extension to $\zhat$, and $I+VR_0(\zeta)\chi$ is 
invertible for $\zeta$ in the 
physical space with $\Im \zeta$ sufficiently large,  meromorphic Fredholm
theory ensures that $(I+V R_0(\zeta)\chi)^{-1}$ is a meromorphic 
operator-valued function on $\zhat$, and each of (\ref{eq:req1}),
(\ref{eq:req2}), and (\ref{eq:V0asmodel}) hold on all of $\zhat$.  Moreover,
writing $I+V\Rz= (I+V\Rz (I-\chi))(I+V\Rz \chi)$ and noting that 
$(I+V\Rz(I-\chi))^{-1}=I-V\Rz(I-\chi)$, this shows that 
$$\RV(\zeta)=\Rz(\zeta)(I+V\Rz(\zeta)\chi))^{-1}(I-V\Rz(\zeta)(I-\chi)):L^2_c(X)\rightarrow H^2_{\loc}(X)$$
 has a meromorphic continuation to $\zhat$.

We note from (\ref{eq:RVzsep}) that 
 $\RVz$ is bounded on $L^2(X)$ when $\zeta$ is in the
physical space and is away from a discrete set of poles (corresponding to eigenvalues).
The same is true of $\RV$.

Throughout this paper we shall mainly work with subsets of $B_l(\sqrt{2l-1})
\subset \zhat$,
$l\in \Natural$.
We recall 
$B_{l}(\rho)$ is defined to be the
the connected component of 
$\{ \zeta \in \zhat: |\tau_l(\zeta)|<\rho \}$ 
which has nonempty intersection
 with
 both the physical space and the portion of $\zhat$  with
$\Re \tau_0(\zeta)>0$.  
  The choice of $\sqrt{2l-1}$ in $B_l(\sqrt{2l-1})$ 
is made because then (for $l\geq 1$),
$B_l(\sqrt{2l-1})$ contains only a single 
point of $\zhat$ corresponding to a threshold, the one associated with the eigenvalue $l^2$ of 
$-\frac{d^2}{d\theta^2}$ on $\Sphere^1$.  
 If $\epsilon>0$, then
$z=\tau_l(\zeta)$ gives the complex structure of 
$\zhat_{\restrict B_l(\sqrt{2l-1}-\epsilon)}$, and $B_l(\sqrt{2l-1}-\epsilon)$
is naturally identified with a disk $B_\Complex(\sqrt{2l-1}-\epsilon)$ of radius   $\sqrt{2l-1}-\epsilon$ in
$\Complex$,
centered at the origin.
In this coordinate $z$,  $z=0$ corresponds to  the threshold $l^2$ and  the intersection of $B_\Complex(\sqrt{2l-1}-\epsilon)$ with the first quadrant corresponds to a region in physical space, and so has $\Im \tau_k>0$ for all $k\in \Natural_0$. 
If $z$ lies in the intersection of $B_\Complex(\sqrt{2l-1}-\epsilon)$ with the fourth quadrant, then
$\Im \tau_k(\zeta(z))<0$ for $0\leq k\leq l$ and $\Im \tau_k(\zeta(z))>0$ for $k>l$.
On the other hand, if $z$ lies in the intersection of $B_\Complex(\sqrt{2l-1}-\epsilon)$ with the second quadrant, then
$\Im \tau_k(\zeta(z))<0$ for $0\leq k\leq l-1$ and $\Im \tau_k(\zeta(z))>0$ for 
$k\geq l$.

On the open set $B_l(\sqrt{2l-1}-\epsilon)$, $z=\tau_l(\zeta)$ is a coordinate
compatible with 
the complex structure of $\zhat$.  Thus it is natural to use $\tau_l$ 
as a local coordinate.  We write 
$$\zeta_l:\{z\in \Complex:\;|z|<\sqrt{2l-1}-\epsilon\}\rightarrow B_l(\sqrt{2l-1}-\epsilon)\subset\zhat $$
as the function satisfying
$$ \zeta_l(\tau_l(\zeta))=\zeta\; \text{ for
all}\; \zeta\in B_l(\sqrt{2l-1}-\epsilon).$$

We note that if $\zeta \in B_l(\sqrt{2l-1}-\epsilon)$, then $\Re \tau_j(\zeta)>0$ if $0\leq j<l$ and $\Im \tau_j(\zeta)>0$ if $j>l$.

The next lemma follows easily from (\ref{eq:RVzsep}) and Lemma \ref{l:Rvzzresfree}, but is fundamental to many of the results of this paper.
\begin{lemma} \label{l:lognbhd}
Let $V_0\in L^\infty_c(\Real)$,
 $\alpha>0$ and $\chi \in L^\infty_c(X)$.  Then for $l$ sufficiently large,
uniformly for $\zeta \in B_{l}(\alpha \log l )$,
$\| \chi(I-\pr_l)\RVz(\zeta)\chi\|=O(l^{-1/2})$.
\end{lemma}
\begin{proof}  Set $\tau_l=z$, and $|z|<\alpha \log l$.  
Then using the identity
$$\tau_k^2=\tau_l^2+l^2-k^2,$$  for $l$ sufficiently large 
$|\tau_k(\zeta_l(z))|>\sqrt{l}$ for $k\in \Natural_0$, $k\not = l$. Moreover, $\Im \tau_k(\zeta_l(z))>0$ if $k>l$,
and $|\Im \tau_k(\zeta_l(z))|=O(1)$ if $k<l$.
 Then the lemma follows from Lemma \ref{l:Rvzzresfree} and the 
 representation of $\RVzz$ given by (\ref{eq:RVzsep}).
\end{proof}

\section{Multiplicities of poles and results of \cite{go-si}}
\label{ss:mpandGS}

For an operator $A$ depending meromorphically on $\zeta\in \Complex$ 
or $\zeta \in \zhat$, let $\Xi(A,\zeta_0)$ denote the principal part of the Laurent expansion of $A$ at $\zeta_0$.  
For $V\in L^\infty_c(X)$ and $\zeta_0\in \zhat$, define 
\begin{equation}\label{eq:mV}
\mV(\zeta_0)\defeq
\rank \Xi (R_V,\zeta_0)(L^2_c(X)).
\end{equation}
Suppose $\chi \in L^\infty_c(X)$ satisfies $\chi V=V$ (and, if $V\equiv 0$,
$\chi$ is nontrivial).  Then it follows from an expansion of $\RV$ at its
singularities as in \cite[Theorems 2.5, 2.7, 3.9, 3.17]{DyZw} and a unique continuation 
result, e.g. \cite[Remark 6.7]{JeKe}, that
$m_V(\zeta_0)= \rank \Xi (\chi R_V \chi,\zeta_0)$.  Note in case $R_V$ is 
analytic at $\zeta_0$, then $m_V(\zeta_0)=0$.

If $V_0\in L^\infty_c(\Rd)$ and $\lambda_0\in \Complex$ we define 
\begin{equation}\label{eq:mvzz}
\mvzz(\lambda_0)\defeq \rank \Xi(\RVzz,\lambda_0)(L_c^2(\Real^d)).
\end{equation}
Again, the second $0$ in the subscript is meant to help us think of this as
corresponding to a model.
 As for $\mV$, if $\chi\in L^\infty_c(\Real)$ satisfies
$\chi V=V$ (and $\chi$ is nontrivial if $V_0\equiv 0$), then
$\mvzz(\lambda_0)=  \rank \Xi(\chi \RVzz \chi,\lambda_0)$.

We recall some definitions and 
results of \cite{go-si}, adapted to our setting.

Let  $A$ be a bounded linear operator on a complex Hilbert space $\mch$,
depending meromorphically on $z\in \Omega \subset \Complex$, where $\Omega$
is a domain.  Near a point $z_0\in \Omega$, $A(z)=\sum_{j=-n}^\infty (z-z_0)^jA_j$.  If the operators $A_{-1},\dots,A_{-n}$ are finite rank, then
we say
$A$ is {\em finitely-meromorphic} at $z_0$.  If $A$ is finitely-meromorphic
at each $z_0\in \Omega$, then $A$ is {\em finitely-meromorphic on $\Omega$.}  
Suppose that $A$ is a compact operator on $\mch$, $A$ is finitely-meromorphic
on $\Omega$, and $I+A(z_1)$ is invertible for some $z_1\in \Omega$.
Then by the meromorphic Fredholm theorem, $(I+A(z))^{-1}$ is
finitely-meromorphic on $\Omega$.


Suppose $A$ is a  finitely-meromorphic operator
on a domain $\Omega$, with $(I+A)^{-1}$ also finitely-meromorphic on $\Omega$.
Below we denote the derivative of $A$ with respect to 
$z$ by $\dot{A}$.  Then for $z_0\in \Omega$, define
$$M(I+ A,z_0)\defeq\frac{1}{2\pi i}\tr \int_{\gamma_{z_0}}\dot{A}(z)(I+A(z))^{-1}dz$$
where $\gamma_{z_0}$ is a  positively oriented circle, centered at 
$z_0$ with radius $\epsilon$.  Here we choose $\epsilon$ 
small enough that $\{|z-z_0|\leq \epsilon\}\subset \Omega$
 and neither $A$ nor $(I+A)^{-1}$ 
has poles in the set $\{z: 0<|z-z_0|\leq \epsilon\}$.

Our definition of finitely-meromorphic is local, so makes sense on domains
in $\zhat$ as well, using a local coordinate compatible with the complex
structure of $\zhat$.  
Likewise, we can define $M(I+A, \zeta_0)$ for such operators.  (This requires 
the choice of a circle small enough that it has in its interior at most 
one ramification point of $\zhat$.)

We will say the linear operator $A$ on the Hilbert space $\mch$ 
satisfies hypotheses (H1) on a domain $\Omega 
\subset \Complex$ if $A$ is a finitely-meromorphic, compact operator defined
on $\Omega$, and $I+A$ is invertible for at least one
point in $\Omega$, and hence has a finitely-meromorphic inverse in $\Omega$.

The following lemma is a direct consequence of \cite[Proposition 5]{go-si}.
\begin{lemma}\label{l:Mproduct}( \cite[Proposition 5]{go-si})
Suppose $A,\;B:\mch \rightarrow \mch$ satisfy hypotheses (H1),
and  suppose $B$ and $(I+B)^{-1}$ 
are analytic 
on $\Omega$.  Then for $z_0\in \Omega$, 
$M(I+A, z_0)=M((I+A)(I+B),z_0).$
\end{lemma}

Let $T:L^2(X)\rightarrow L^2(X)$ be a bounded linear operator. 
 We shall repeatedly 
make use of the straightforward identities
\begin{equation}\label{eq:prlidentity}
I+T\pr_l= (I+\pr_lT\pr_l)(I+(I-\pr_l) T\pr_l)\; \text{and}\; 
(I+(I-\pr_l)T\pr_l)^{-1}=I-(I-\pr_l)T\pr_l.
\end{equation}

\begin{lemma}\label{l:Mpr} Let $A:L^2(X) \rightarrow L^2(X)$ satisfy hypotheses
(H1)  on a domain $\Omega$.
Then for $z_0\in \Omega$
$$M(I+A\pr_l,z_0)= M(I+\pr_l A \pr_l, z_0).$$
\end{lemma}
\begin{proof}

Using (\ref{eq:prlidentity}) implies that
\begin{multline}\label{eq:MAprl}
M(I+ A\pr_l,z_0)=\frac{1}{2\pi i}\tr \int_{\gamma_{z_0}}\dot{A}(z)\pr_l(I+A(z)\pr_l)^{-1}dz \\= \frac{1}{2\pi i}\tr \int_{\gamma_{z_0}}\dot{A}(z)\pr_l(I+\pr_lA(z)\pr_l)^{-1}dz
\end{multline}
where $\gamma_{z_0}$ is a small circle centered at $z_0$ as in the definition
of $M(I+A,z_0)$.

Because $\pr_l$ is a projection, using the 
cyclicity of the trace $\tr (B\pr_l)= \tr (\pr_l B \pr_l)$
for a trace class operator $B:L^2(X)\rightarrow L^2(X)$.
Using this in (\ref{eq:MAprl}) gives
$$M(I+ A\pr_l,z_0)=\frac{1}{2\pi i}\tr \int_{\gamma_{z_0}}\pr_l\dot{A}(z)\pr_l(I+\pr_l A(z)\pr_l)^{-1}dz = M(I+\pr_lA\pr_l,z_0).$$
\end{proof}

The following proposition is a variant of a well-known result in the study of 
resonances of Schr\"odinger operators on $\Real^d$; compare e.g. \cite[Theorem 3.15]{DyZw}.  
\begin{prop} \label{p:mvsM} Suppose $V\in L^\infty_c(X)$ is 
nontrivial, and let $\chi \in L_c^\infty(X)$  satisfy
$\chi V=V$. 
Then the operator $R_V(\zeta)$ has a pole at $\zeta_0\in \zhat$  
if and only if the operator $I+V\Rz(\zeta)\chi$ has nontrivial
null space at $\zeta_0$. 
 Moreover, if 
$\zeta_0$ does not correspond to a threshold, then 
$$\mV(\zeta_0)=M(I+VR_0\chi, \zeta_0).$$
\end{prop}
\begin{proof} 
This proposition can be proved by 
essentially the same method as \cite[Theorem 3.5]{DyZw}. 
\end{proof}

We recall the notation
$\Vd=V-V_0.$

Another useful identity is the following.
\begin{lemma}\label{l:Mvdrv0}
Let $\chi \in L^\infty_c(X)$ satisfy $\chi V=V$ and $\chi V_0=V_0$.
Then for $\zeta_0\in \zhat$ so that $\zeta_0$ does not correspond
to a threshold,
$$\mV(\zeta_0)= M(I+\Vd \RVz \chi,\zeta_0)+m_{V_0}(\zeta_0).$$
\end{lemma}
\begin{proof}
We first note that 
\begin{align*}
I+V\Rz \chi & =\left (I+\Vd \Rz \chi(I+V_0\Rz \chi)^{-1}\right)(I+V_0\Rz \chi)\\
 & = (I+\Vd \RVz \chi)(I+V_0\Rz \chi).
\end{align*}
Thus using Proposition \ref{p:mvsM} and  \cite[Theorem 5.2]{go-si}
gives
\begin{multline}
\mV(\zeta_0)=M(I+V \Rz \chi,\zeta_0)=
M(I+\Vd \RVz\chi,\zeta_0)+M(I+V_0\Rz \chi,\zeta_0)\\
= M(I+\Vd \RVz\chi,\zeta_0)+m_{V_0}(\zeta_0),
\end{multline}
proving the lemma.
\end{proof}

\begin{lemma}\label{l:multandprl}
Suppose $V,\chi\in L^\infty_c(X)$,
$\chi V=V$, and $\chi$ is independent of $\theta.$  Let
$\alpha>0$. Then there is an $L>0$ so that for $l>L$ 
$$M(I+ V\Rz \chi,\zeta_0)= 
M(I+\pr_l (I+V\Rz  (I-\pr_l)\chi)^{-1}V \Rz \pr_l \chi ,\zeta_0)$$
for any $\zeta_0\in B_l(\alpha \log l).$ 
\end{lemma}
\begin{proof}
We begin by writing
$$I+V\Rz \chi = 
(I+V\Rz (I-\pr_l)\chi)(I+ (I+V\Rz (I-\pr_l)\chi)^{-1}V\Rz \pr_l\chi)$$
and noting that since by Lemma \ref{l:lognbhd}
$\| V\Rz (I-\pr_l)\chi\|=O(l^{-1/2})$ uniformly
on $B_l(\alpha \log l)$ 
 there is an $L>0$ so that for $l>L$, $ (I+V\Rz (I-\pr_l)\chi)^{-1}$
is analytic on $B_l(\alpha \log l)$. 
Thus for these $l$ by Lemma \ref{l:Mproduct}
$M(  I+V\Rz \chi, \zeta_0)= M(I+ (I+V\Rz (I-\pr_l)\chi)^{-1}V\Rz \pr_l\chi,
\zeta_0)$ for any $\zeta_0\in B_l(\alpha \log l)$.  An application of 
Lemma \ref{l:Mpr} completes the proof. 
\end{proof}

\begin{lemma} \label{l:RTcheck}
Let $V,\chi\in L^\infty_c(X)$, with $V$ satisfying (\ref{eq:Vhyp}),
$\chi V=V$, and $\chi$ independent of $\theta$.
Set $A_{l,V}= (I+V\Rz  (I-\pr_l)\chi)^{-1}$, 
$B_{l,V}=V \Rz \pr_l \chi$.
Let $K\subset \Complex $ be a compact set such that
$\RVzz$ is analytic on $K$, and suppose $0\not \in K$ if $d=1$.  Choose $\rho>0$
so that $K\subset \{ \lambda \in \Complex: |\lambda|<\rho\}$, and set
 $K_l=\{ \zeta \in B_l(\rho): \tau_l(\zeta)\in K\}$.
Then for sufficiently large $l$ 
\begin{equation}\label{eq:diffest1}
\|\pr_l(A_{l,V}B_{l,V}-A_{l,V_0}B_{l,V_0})\|=O(l^{-\delta})
\end{equation}
and \begin{equation}\label{eq:diffest2}
\| (I+\pr_l A_{l,V_0}B_{l,V_0})^{-1}\pr_l(A_{l,V}B_{l,V}-A_{l,V_0}B_{l,V_0})\|=
O(l^{-\delta})
\end{equation}
uniformly for $\zeta \in K_l$.
\end{lemma}
\begin{proof}
We write
\begin{equation}\label{eq:rewrite}
\pr_l(A_{l,V}B_{l,V}-A_{l,V_0}B_{l,V_0})=
\pr_l(A_{l,V}-A_{l,V_0})B_{l,V}+\pr_lA_{l,V_0}(B_{l,V}-B_{l,V_0}).
\end{equation}
By Lemma \ref{l:lognbhd}, $\|A_{l,V}-I\|=O(l^{-1/2})$, $\|A_{l,V_0}-I\|=O(l^{-1/2})$ uniformly on $B_l(\rho)$, so that the first 
term on the left hand side is $O(l^{-1/2})$.  
Moreover, $\pr_lA_{l,V_0}(B_{l,V}-B_{l,V_0})=A_{l,V_0}\pr_l(B_{l,V}-B_{l,V_0})
= A_{l,V_0}\pr_l\Vd \Rz \pr_l$, and  $\|\pr_l\Vd \pr_l\|=O(l^{-\delta})$ by 
our assumption on $V$.  Hence the norm of the second term on the right hand side of 
(\ref{eq:rewrite}) is $O(l^{-\delta})$.  This proves (\ref{eq:diffest1}).

On $K_l$, 
\begin{equation}\label{eq:exp}
I+\pr_lA_{l,V_0}B_{l,V_0}= I+\pr_l B_{l,V_0}+O(l^{-1/2})=I +\pr_l V_0\Rz \chi+O(l^{-1/2}).
\end{equation}  But 
$(I+\pr_l V_0\Rz \chi)^{-1}=I-\pr_l+(I-V_0\RVzz(\tau_l)\chi)\pr_l
=I-\pr_l+T\pr_l$, where 
$T$ 
is given by $T=(I+V_0\Rzz(\tau_l)\chi)^{-1}
= I-V_0\RVzz (\tau_l)\chi.$  By our choice of $K$, $T$ is
uniformly bounded for $\tau_l\in K$,
or for $\zeta \in K_l$, and 
hence $(I+\pr_l V_0\Rz \chi)^{-1}$ is bounded on $K_l$.  Using (\ref{eq:exp}),
this shows $(I+\pr_l A_{l,V_0}B_{l,V_0})^{-1}$ is bounded 
on $K_l$, and thus, by (\ref{eq:diffest1}), we 
get (\ref{eq:diffest2}).
\end{proof}


\section{A resolvent estimate and localizing the resonances in the $L^\infty$ case: Proofs of Theorems \ref{thm:poleexist},  
 \ref{thm:polefree}, and \ref{thm:0}\label{s:thms1and2}}
In this section we  
prove Theorems \ref{thm:poleexist},  
 \ref{thm:polefree}, and \ref{thm:0} in the case of an $L^\infty$ potential $V$, providing a high-energy localization of the resonances 
 in sets $B_l(\rho)$.
We also prove Proposition \ref{p:Rest} and Lemma \ref{l:bdryphys}, which show
 that the resolvent for the potential $V_0$
is, at high energies, a good approximation of the resolvent for the 
potential $V$ away from poles.

\vspace{2mm}
We shall use notation for a disk in the $\tau_l$ coordinate in $B_l(\rho)$.
For $\lambda_0\in \Complex$, $r_0>0$, set $\rho= |\lambda_0|+r_0+1$, and
define, for $2l> \rho^2+1$, $\Dl(\lambda_0,r_0)\subset B_l(\rho)\subset\zhat$ by 
$$\Dl(\lambda_0,r_0)\defeq \{\zeta\in B_l(\rho): \; |\tau_l(\zeta)-\lambda_0|<
r_0\}.$$

A preliminary step in the proof of Theorem \ref{thm:poleexist} is the 
following proposition, which provides an initial localization of the 
resonances.

\begin{prop}\label{p:prelimThm1}
Let $V\in L^\infty_c(X)$ satisfy (\ref{eq:Vhyp}).
Suppose
$\lambda_0\in \Complex  $, $\lambda_0\not = 0$ is a resonance of 
$-\Lz+V_0$ on $\Rd$,
of multiplicity $\mvzz(\lambda_0)$.  Then there are $L, \epsilon>0$
so that
$$\sum_{\substack{\zeta\in \Dl (\lambda_0,\epsilon)\\
m_V(\zeta)> 0}}m_V(\zeta) 
= 2\mvzz(\lambda_0)$$
when $l>L$.
\end{prop}
\begin{proof}  Choose $\epsilon>0$ so that $\RVzz(\lambda)$ is analytic on 
$0<|\lambda -\lambda_0|\leq \epsilon$ and $\epsilon<|\lambda_0|$.
By our expression (\ref{eq:RVzsep}) for $\RVz$ using separation of variables and
Lemmas \ref{l:Rvzzresfree} and \ref{l:lognbhd},
$m_{V_0}(\zeta_l(\lambda_0))= 2\mvzz(\lambda_0)$ for $l$ sufficiently large.  
Choose $\chi \in L^\infty_c(X)$ independent of $\theta$ so that 
$\chi V=V$.
From Proposition \ref{p:mvsM} and our choice of $\epsilon$,  
for $l$ sufficiently
large 
$$m_{V_0}(\zeta_l(\lambda_0))= M(I+V_0\Rz\chi,\zeta_l(\lambda_0))
= \sum_{\substack{\zeta\in \Dl (\lambda_0,\epsilon)\\
M(I+V_0\Rz \chi, \zeta)\not = 0}}M(I+V_0\Rz\chi,\zeta ).$$

Lemma \ref{l:multandprl}
implies that if $W=V_0$ or $W=V$,
\begin{equation}\label{eq:Wintermed}
M(I+W\Rz\chi,\zeta')=
M(I+\pr_l (I+W\Rz  (I-\pr_l)\chi)^{-1}W \Rz \pr_l \chi ,\zeta' )
\;\text{for $\zeta'\in \Dl(\lambda_0,\epsilon)$ }
\end{equation}
if  $l$ is sufficiently large.  

By Lemma \ref{l:RTcheck} and 
an operator Rouch\'e Theorem  (\cite[Theorem 2.2]{go-si}), for
$l$ sufficiently large
\begin{multline}
\label{eq:goal}\sum_{\substack{\zeta\in \Dl (\lambda_0,\epsilon)\\
M(I+\pr_l (I+V\Rz  (I-\pr_l)\chi)^{-1}V \Rz \pr_l \chi , \zeta)\not =0}}M(I+\pr_l (I+V\Rz  (I-\pr_l)\chi)^{-1}V \Rz \pr_l \chi ,\zeta )\\
=\sum_{\substack{\zeta\in \Dl (\lambda_0,\epsilon)\\
M(I+\pr_l (I+V_0\Rz  (I-\pr_l)\chi)^{-1}V \Rz \pr_l \chi , \zeta)\not =
0}} M(I+\pr_l (I+V_0\Rz  (I-\pr_l)\chi)^{-1}V_0 \Rz \pr_l \chi ,\zeta).
\end{multline}
Combining (\ref{eq:Wintermed}) (with $W=V$ and with $W=V_0$),
(\ref{eq:goal}),  and
 another application of Proposition \ref{p:mvsM}, this time with $V$,
 proves the proposition.
\end{proof}

\subsection{Proofs of Theorems \ref{thm:poleexist}
and \ref{thm:polefree} for $V\in L^\infty_c(X)$}
Theorem \ref{thm:poleexist} follows from combining
Proposition \ref{p:prelimThm1} and the result of Theorem
\ref{thm:polefree} for $L^\infty$ potentials.  In this section we prove Theorem 
\ref{thm:polefree} for $L^\infty$ potentials $V$.

Recall by the definition of $\Xi(\RVzz,\lambda_0)$,
if $\lambda_0\in \Complex$ is a pole of $\RVzz$, then 
$\RVzz(\lambda)-\Xi(\RVzz(\lambda),\lambda_0)$ is analytic at $\lambda_0$.  
Define
\begin{equation}\label{eq:RVzreg}
\RVz^{\reg}(\zeta;\lambda_0,l)\defeq \RVz(\zeta)-\Xi(\RVz,\zeta_l(\lambda_0)).
\end{equation}
For $l$ sufficiently large, by (\ref{eq:RVzsep}) and Lemma \ref{l:lognbhd}
\begin{equation}\label{eq:RVzreglsuff}\RVz^{\reg}(\zeta;\lambda_0,l)=
\Xi(\RVzz(\lambda),\lambda_0)|_{\lambda=\tau_l(\zeta)}\pr_l.
\end{equation}
Note that if $\RVz$ is analytic at $\zeta_l(\lambda_0)$, then 
$\RVz^{\reg}(\zeta;\lambda_0,l)= \RVz(\zeta)$.

\begin{lemma}\label{l:Rreg} Suppose
$V,\;\chi \in L^\infty_c(X)$ and $V$ satisfies (\ref{eq:Vhyp}).
Let $\lambda_0\in \Complex$ and 
$\RVz^{\reg}=\RVz^{\reg}(\zeta;\lambda_0,l)$ be the operator defined in (\ref{eq:RVzreg}).
If $\RVzz(\lambda)$ is analytic for $0<|\lambda-\lambda_0|\leq \epsilon$,
then for
$l$ sufficiently large $\Vd\RVz^{\reg}(\zeta)\chi= \Vd \RVz^{\reg}(\zeta;\lambda_0,l)\chi$ is analytic on $
\overline{\Dl}(\lambda_0,\epsilon)$ and 
 as $l\rightarrow \infty$ the estimate
$\left\| \chi \RVz^{\reg}(\zeta)\Vd\RVz^{\reg}(\zeta)\chi\right\|= O(l^{-\delta})$ holds uniformly for
$ \zeta\in \overline{\Dl}(\lambda_0,\epsilon)$.
\end{lemma}
\begin{proof}
Without loss of generality we can assume $\chi$ is 
independent of $\theta$ and $\chi V=V$.  
By (\ref{eq:RVzsep}) and Lemma \ref{l:lognbhd}, for $l$ sufficiently large $\RVz^{\reg}(\zeta)$ is analytic and bounded
in $\overline{\Dl}(\lambda_0,\epsilon)$.  
We write
\begin{multline}
\chi \RVz^{\reg}(\zeta)\Vd\RVz^{\reg}(\zeta)\chi = \chi\RVz^{\reg} \chi(I-\pr_l)\Vd\RVz^{\reg} \chi 
\\+ \chi\RVz^{\reg}\chi \pr_l \Vd \RVz^{\reg} \chi (I-\pr_l)
+ \chi \RVz^{\reg} \chi \pr_l \Vd \RVz^{\reg} \chi \pr_l.
\end{multline}
Now for $l$ sufficiently large and $\zeta \in \overline{\Dl}(\lambda_0,\epsilon)$,
$\| \chi \RVz^{\reg} \chi(I-\pr_l)\|=O(l^{-1/2})$ uniformly in $\overline{\Dl}(\lambda_0,\epsilon)$.
Since $\|V_m\|=O(m^{-\delta}), $ $\|\pr_l\Vd \pr_l\|=O(l^{-\delta})$, and 
so $$\|\pr_l \Vd \RVz^{\reg} \chi \pr_l\|= \|\pr_l \Vd \pr_l \RVz^{\reg} \chi \pr_l\|=O(l^{-\delta}).$$
\end{proof}

A related lemma which we also need is the following.
\begin{lemma}\label{l:VdRVzs}Let $V,\chi \in L^\infty_c(X)$ with 
$V$ satisfying (\ref{eq:Vhyp}).
Let $K\subset \Complex$ be a compact set on which $\RVzz$ is analytic
and suppose 
$K\subset \{\lambda\in \Complex: |\lambda|<\rho\}$. 
Set $K_l\defeq \{ \zeta\in B_l(\rho): \tau_l(\zeta)\in K\}\subset \zhat$.
Then for $l$ sufficiently large, $\|\chi R_{V_0}\Vd \RVz \chi\|=O(l^{-\delta})$
uniformly on $K_l$.
\end{lemma}
\begin{proof}
This lemma can be proved  by mimicking the proof of Lemma \ref{l:Rreg}.
Alternatively, it can be proved by covering $K_l$ with a finite number of 
neighborhoods on which Lemma \ref{l:Rreg} holds.
\end{proof}


\vspace{2mm}
\noindent
{\em Proof of Theorem \ref{thm:polefree} for $V\in L^\infty_c(X)$.} 
We shall use the identities (\ref{eq:V0asmodel}).
Thus poles of $\RV$ in 
$B_l(\rho)$  are the values of  $\zeta\in B_l(\rho)$
such that $I+\Vd\RVz(\zeta)\chi$ is not invertible.  Here $\chi \in C_c^\infty(X)$
satisfies $\chi V=V$ and is independent of $\theta$.

\np For each $\lambda_j\in \Lambda_\rho$, $\lambda_j\not = 0$,
let $\epsilon_j>0$ be as guaranteed
by Proposition \ref{p:prelimThm1},
so that there are exactly $2\mvzz(\lambda_0)$ resonances
(counted with multiplicity) of $-\Delta +V$ in $\Dl(\lambda_j,\epsilon_j)$ for 
$l$ sufficiently large.  Set 
$$K=\{\lambda \in \Complex: \; \epsilon'\leq |\lambda|\leq \rho\; \text{and}\; |\lambda-\lambda_j|\geq \epsilon_j\; \text{for all $\lambda_j\in \Lambda_\rho$}\},$$
and $K_l=\{\zeta\in B_l(\rho+1):\tau_l(\zeta)\in K\}= \overline{B}_l(\rho)
\setminus\left(D_l(0,\epsilon') \cup_{\lambda_j \in \Lambda_\rho }\Dl(\lambda_j,\epsilon_j)\right)$.
By an application of Lemma \ref{l:VdRVzs}, 
for $l$ sufficiently large $I+\Vd \RVz(\zeta)\chi$ is
invertible by its Neumann series on $K_l$.
Thus by (\ref{eq:V0asmodel})
$R_V$ has no poles on $K_l$ for $l$ sufficiently large.

\np Now we work on $\Dl(\lambda_j,\epsilon_j)$, and
set $\RVz^{\reg}(\zeta)=\RVz^{\reg}(\zeta;l,\lambda_j)$,
so that 
$\RVz^{\reg}(\zeta)=\RVz(\zeta)-\Xi(\RVzz(\lambda),\lambda_j)|_{\lambda=\tau_l(\zeta)}\pr_l$
for $l$ sufficiently large.
By our choice of $\epsilon_j$ this
is analytic on $\overline{\Dl}(\lambda_j,\epsilon_j)$ for large enough $l$.
  Then by 
Lemma \ref{l:Rreg} $I+\Vd\RVz^{\reg}(\zeta)\chi$ is invertible in 
$\Dl(\lambda_j,\epsilon_j)$, with  $(I+\Vd \RVz^{\reg}(\zeta)\chi)^{-1}=
 I-\Vd\RVz^{\reg}(\zeta)\chi+O_{L^2(X)\rightarrow L^2(X)}(l^{-\delta})$
for
$\zeta \in \overline{\Dl}(\lambda_j,\epsilon_j)$.
Thus on $\overline{\Dl}(\lambda_j,\epsilon_j)$
\begin{align}\label{eq:factor1}
I+\Vd \RVz\chi & = (I+\Vd \RVz^{\reg}(\zeta)\chi)\Big(I+(I+\Vd\RVz^{\reg}(\zeta)\chi)^{-1}
\Vd \Xi(\RVzz(\lambda),\lambda_0)|_{\lambda=\tau_l(\zeta)} \pr_l\chi\Big).
\end{align}

By (\ref{eq:factor1}) and (\ref{eq:prlidentity}), $I+\Vd\RVz\chi$ is invertible at a point $\zeta \in \Dl(\lambda_j,\epsilon_j) $ if and only if 
$I+\pr_l (I+\Vd\RVz^{\reg}(\zeta)\chi)^{-1}
\Vd \Xi(\RVzz(\lambda),\lambda_0)|_{\lambda=\tau_l(\zeta)} \pr_l\chi$ is invertible at $\zeta$.
There is a $C_j$ so that $\|\chi \Xi(\RVzz,\lambda_j)\chi\|
\leq C_j|\lambda-\lambda_j|^{-\mvzz(\lambda_j)}$ on 
$\{\lambda\in \Complex: \;|\lambda-\lambda_j|\leq \epsilon_j\}$, 
\cite[Theorems 2.5, 3.9]{DyZw}.
 Thus on $\Dl(\lambda_j,\epsilon_j)$ using Lemma \ref{l:Rreg}
\begin{align*}
& \left\|  \pr_l (I+\Vd \RVz^{\reg}(\zeta)\chi)^{-1}
\Vd \Xi(\RVzz(\lambda,\lambda_0))|_{\lambda=\tau_l(\zeta)}\pr_l\chi
\right\| \\
& = \left\|  \sum_{j=0}^\infty
\pr_l (-\Vd \RVz^{\reg}(\zeta))^j \Vd 
\Xi(\RVzz(\lambda,\lambda_0))|_{\lambda=\tau_l(\zeta)}\pr_l\chi
\right\|
\\
 & \leq \left\| \pr_l(I- \Vd \RVz^{\reg}(\zeta)) \Vd \Xi(\RVzz(\lambda,\lambda_0))|_{\lambda=\tau_l(\zeta)}\pr_l\chi
\right\| + C_j'l^{-\delta}|\tau_l(\zeta)-\lambda_j|^{-\mvzz(\lambda_j)}.
\end{align*}
Now we use 
Lemma \ref{l:lognbhd}, $\|V_m\|_{L^{\infty}}=O(m^{-\delta})$ and
that $\pr_l$ commutes with $\RVzz$ 
 so that $\|\pr_l(I- \Vd \RVz^{\reg}(\zeta)) \Vd\pr_l\|=O(l^{-\delta})$ on $\overline{\Dl}(\lambda_j,\epsilon_j)$.  
Thus there is a  (new)$C_j'$ so that  $$\|  \pr_l (I+\Vd \RVz^{\reg}(\zeta))^{-1}
\Vd \Xi(\RVzz(\lambda,\lambda_0))|_{\lambda=\tau_l(\zeta)}\pr_l\chi\| \leq 
C_j'l^{-\delta}|\tau_l(\zeta)-\lambda_j|^{-\mvzz(\lambda_j)}$$
on $\overline{\Dl}(\lambda_j,\epsilon_j)$. Therefore
$I+ \pr_l (I+\RVz^{\reg}(\zeta))^{-1}
\Xi(\RVzz(\lambda,\lambda_0))|_{\lambda=\tau_l(\zeta)}\pr_l\chi$ is 
invertible in
this region if $|\tau_l(\zeta)-\lambda_j|\geq C_j l^{-\delta/\mvzz(\lambda_j)}$,
where we can take  $C_j= (2C_j')^{1/\mvzz(\lambda_j)}$.  Taking $\tilde{C}=\max_{\lambda_j\in \Lambda_\rho}C_j$ 
finishes the proof of Theorem \ref{thm:polefree} away from $\tau_l=0$.

\np If $\RVzz(\lambda)$ does not have a pole at the origin, then there
is a $\delta>0$ so that
for $l$ sufficiently large $\RVz(\zeta)$ is analytic in $\overline{B}_l(\delta)$.
Thus by
 Lemma \ref{l:VdRVzs}, for $l$ sufficiently large $\RV(\zeta)$ is analytic
in $\overline{B}_l(\delta)$.
\qed

\subsection{Approximating the resolvent $\RV$}
In a sense made precise below in Proposition 
\ref{p:Rest} and Lemma \ref{l:bdryphys}, at high energies $R_{V_0}$ approximates $\RV$ well away from resonances.
The first result is useful for neighborhoods of thresholds.
\begin{prop}\label{p:Rest} 
Let $V,\chi \in L^\infty_c(X)$, with $V$ satisfying (\ref{eq:Vhyp}).
Let $K\subset \Complex$ be a compact set on which $\RVzz$ is analytic
and suppose 
$K\subset \{\lambda\in \Complex: |\lambda|<\rho\}$. 
Define $K_l\defeq \{ \zeta\in B_l(\rho): \tau_l(\zeta)\in K\}\subset \zhat$.
 Then 
for $l$ sufficiently large, $R_V$ is analytic on $K_l$.  Moreover,
if $\chi \in L^\infty_c(X)$, then 
$\|\chi (R_V(\zeta)-\RVz(\zeta))\chi\|=O(l^{-\delta})$ 
uniformly for $\zeta \in K_l$.
\end{prop}
\begin{proof}
Without loss of generality we may assume
$\chi$ is independent of $\theta$ and 
satisfies $\chi V=V$.
Then $\chi \RVz \chi = \chi R_V \chi(I+\Vd R_{V_0}\chi)$.
Since by Lemma \ref{l:VdRVzs}  $\| (\Vd R_{V_0}\chi)^2\|\leq 1/2$ on $K_l$  for $l$ sufficiently
large,  $I+\Vd R_{V_0}\chi$ is invertible 
as 
$(I+\Vd R_{V_0}\chi)^{-1}=\sum_{j=0}^\infty(-\Vd R_{V_0}\chi)^j$, and thus
$R_V$ is analytic on $K_l$.  Moreover,
$$\chi(R_V(\zeta)-\RVz(\zeta))\chi
= \chi \sum_{j=1}^\infty  \RVz(\zeta) (-\Vd R_{V_0}(\zeta)\chi)^j .
$$
By applying Lemma \ref{l:VdRVzs} twice,
 this is
 $$\chi(R_V(\zeta)-\RVz(\zeta))\chi= -\chi \RVz(\zeta)\Vd\RVz(\zeta)\chi+O_{L^2\rightarrow L^2}(l^{-\delta})=O_{L^2\rightarrow L^2}(l^{-\delta}).$$
\end{proof}

A similar result with a similar proof is the following lemma.
The points $\zeta\in \zhat$ considered in this lemma lie on the boundary
of the physical space, but are away from the thresholds.  
\begin{lemma}\label{l:bdryphys}
Let $V,\chi\in L^\infty_c(X)$, with 
$V$ satisfying (\ref{eq:Vhyp}).
 Then there are 
constants $M,\; L>0$ so that 
\begin{multline}\text{if}\; l>L, \; 
\zeta \in B_l(\sqrt{2l-1}), \; \
\tau_l(\zeta)\in i[0,\infty),\; \text{and}\;
M<\frac{\tau_l(\zeta)}{i}<\sqrt{2l-1}-\frac{M}{\sqrt{l}},\;\\ \text{then}\; 
\|\chi (\RV(\zeta)-\RVz(\zeta))\chi\|=O(l^{-\delta}).
\end{multline}
Likewise, there are constants $M_1,\; L_1>0$ so that
\begin{multline}\text{if}\;
 l>L_1,  \; \zeta \in B_l(\sqrt{2l-1}), \; 
\tau_l(\zeta)\in [0,\infty),\;\text{and}\;
M_1<\tau_l(\zeta)<\sqrt{2l-1}-\frac{M_1}{\sqrt{l}}, \; \\ \text{then}\;
\|\chi (\RV(\zeta)-\RVz(\zeta))\chi\|=O(l^{-\delta}).
\end{multline}
\end{lemma}
\begin{proof}
This proof is very similar to the proof of Proposition \ref{p:Rest}.  We
outline the proof of the first statement only, as the proof of the second
is analogous.

Without loss of generality, we may assume $\chi$ is independent of $\theta$
and satisfies $\chi V=V$.  

We next note that if $\zeta \in B_l(\sqrt{2l-1}),$ then for $l>3$ either
$|\tau_{l}(\zeta)|>\sqrt{2l-1}/4$ or $|\tau_{l-1}(\zeta)| >\sqrt{2l-1}/4$
or both are true.  In either case, if $\tau_l(\zeta)\in i[0,\infty)$,
then there is a $c_0>0$ so that $|\tau_j(\zeta)|>c_0l^{1/2}$ for 
$j\not = l, \; l-1$.  Moreover, again with $\tau_l(\zeta)\in i[0,\infty)$,
 $\Im \tau_j(\zeta) >0$ if $j>l$ 
and $\Im \tau_j(\zeta)=0$ if $0\leq j<l$.

Suppose $\zeta \in B_l(\sqrt{2l-1})$, $\tau_l(\zeta)\in i[0,\infty)$, 
and $|\tau_{l}(\zeta)|>\sqrt{2l-1}/4$.  Then using Lemma \ref{l:Rvzzresfree}
and (\ref{eq:RVzsep}) we see that 
$\|\chi \RVz(\zeta)\chi (I-\pr_{l-1})\|=O(l^{-1/2}).$  By Lemma \ref{l:Rvzzresfree} there is a $C>0$ so that if $\lambda\in \Real$, $|\lambda|>C$ then
$\| \Vd\|_{L^\infty}\|\chi \RVzz(\lambda)\chi\|\leq 1/2$.  Choose $M> C+1$; then 
if $\tau_l(\zeta)\in i[0,\infty)$ with $\tau_l(\zeta)/i <\sqrt{2l-1}-M/\sqrt{l}$, for $l$ sufficiently
large $|\tau_{l-1}(\zeta)|> C$.  
Now we restrict ourselves to $\tau_l(\zeta)\in i[0,\infty)$, $\sqrt{2l-1}/4<\tau_l(\zeta)/i< \sqrt{2l-1}-M/\sqrt{l}$. 
Since $\|\pr_{l-1}\Vd \pr_{l-1}\|=O(l^{-\delta})$ by our assumption on 
$\|V_m\|_{L^\infty}$, $\| \chi \RVz(\zeta)\pr_{l-1} \Vd \RVz(\zeta)\pr_{l-1}\chi\|=O(l^{-\delta})$,
and we can follow the proof of Lemma \ref{l:Rreg} to show that
$\|\chi \RVz(\zeta)\Vd\RVz(\zeta)\chi\|=O(l^{-\delta})$.  Then
\begin{multline*}\| \chi(\RV(\zeta)-\RVz(\zeta))\chi\| =\left\| 
\chi \RVz(\zeta)\chi\left((I+\Vd \RVz(\zeta)\chi)^{-1}-I\right)\right\| 
\\
= \left\|  \chi \RVz(\zeta)\Vd \RVz(\zeta)\chi\right\|+O(l^{-\delta})
= O(l^{-\delta})
\end{multline*}
proving the lemma when $\tau_l(\zeta)\in i[0,\infty)$, $\sqrt{2l-1}/4<\tau_l(\zeta)/i< \sqrt{2l-1}-M/\sqrt{l}$.   A similar argument, 
singling out $\pr_{l}$ rather than
$\pr_{l-1}$,  handles 
the case with $\tau_l(\zeta)\in i[0,\infty)$, 
$\sqrt{2l-1}/4<|\tau_{l-1}(\zeta)|$. 
\end{proof}

\subsection{Proof of Theorem \ref{thm:0}}
Theorem \ref{thm:0}  
concerns
poles of $R_V$ arising as perturbations of threshold poles of $\RVz(\zeta)$.  Using separation of variables as in (\ref{eq:RVzsep}), these threshold poles, in turn,
correspond to a pole 
of $\RVzz(\lambda)$ at $\lambda=0$.

We begin with a lemma about poles of $\RVz(\lambda)$ at the 
origin.  This result is well-known if $V_0$ is real-valued.
\begin{lemma} \label{l:multat0} Suppose $V_0\in L^\infty_c(\Real^d)$, and near
$\lambda=0$
\begin{equation}\label{eq:singexp}
\RVzz(\lambda)=\sum_{k=1}^{k_0}\frac{1}{\lambda^k}A_k+A(\lambda),
\end{equation}
where $A$ is analytic in a neighborhood of  the origin.
Then $\mvzz(0)=\max_{0\leq t\leq1}\rank(A_1+tA_2)$.
\end{lemma}
Since $A_1,$ $A_2$ are finite rank,
the rank of $A_1+tA_2$ is it equal to its maximum for all but a finite number
of values of $t$ in $[0,1]$.
\begin{proof}
Using the expansion (\ref{eq:singexp}) and the identity 
$(-\Lz+V_0-\lambda^2)\RVzz(\lambda)=I$ shows that for $k>0$,
 $(-\Lz+V_0)A_{k}=A_{k+2}$, 
where we use the convention $A_{k+2}=0$ if $k+2>k_0$.  
 Just as in 
\cite[Theorem 2.5]{DyZw}, one can use this and the fact that $-\Lz+V_0$ 
commutes with $\RVzz$
to show that for $j\in \Natural$,
 $\Ran(A_{2j})\subset \Ran(A_2)$ and $\Ran(A_{2j+1})\subset\Ran(A_1).$  Here  $\Ran(A_k)$ denotes the range of the operator $A_k$ on $L^2_c(\Real^d)$.
Since $\mvzz(0)=\dim\left( \cup_{k=1}^{k_0} \Ran(A_k)\right)$, 
this shows $\mvzz(0)=\dim\left( \Ran A_1 \cup \Ran A_2 \right).$
But 
$$\dim \left( \Ran A_1 \cup \Ran A_2 \right) = 
\max_{t\in [0,1]}\dim \Ran (A_1+tA_2)=\max_{t\in[0,1]}\rank(A_1+tA_2),$$
proving the lemma.
\end{proof}

\begin{lemma}\label{l:residues}
Let $V\in L^\infty_c(X)$ satisfy (\ref{eq:Vhyp}).
 Let $\epsilon>0$ be chosen
so that $\RVzz(\lambda)$ has no poles in
$\{\lambda \in \Complex:0< |\lambda|<2\epsilon\}$, and let $\gamma_l\subset B_l(2\epsilon)\subset \zhat$
 be the curve $\{|\tau_l|=\epsilon\}$ with positive orientation.  Then for $t\in [0,1]$ and $l$ sufficiently large 
$$\rank \int_{\gamma_l} (1+t \tau_l(\zeta))R_V(\zeta)d\tau_l(\zeta)
\geq \rank \int_{\gamma_l} (1+t \tau_l(\zeta))\RVz(\zeta)d\tau_l(\zeta).
$$
\end{lemma}
\begin{proof}
We assume $\Vd$ is nontrivial, since otherwise there is nothing to prove.

We first point out that if $\RVzz(\lambda)=\sum_{k=1}^{k_0}\lambda^{-k}A_k +A(\lambda)$, with $A(\lambda)$ analytic near $\lambda =0$, then for $l$ 
sufficiently large
$$\int_{\gamma_l} (1+t \tau_l(\zeta))\RVz(\zeta)d\tau_l(\zeta)=
\int_{\gamma_l} (1+t \tau_l(\zeta))\RVzz(\tau_l(\zeta))\pr_l d\tau_l(\zeta)=2\pi i (A_1+tA_2)\pr_l.\; 
$$  

Let $\chi\in L^\infty_c(X)$ satisfy $\chi V=V$, with $\chi$ independent of $\theta$.  
Using Proposition \ref{p:Rest}, for $l$ sufficiently large
$$\left\| \int_{\gamma_l} (1+t \tau_l(\zeta))\chi \left(R_V(\zeta)-\RVz(\zeta)\right)\chi d\tau_l(\zeta)\right\|= O(l^{-\delta}).$$
Thus 
$$\left\| \int_{\gamma_l} (1+t \tau_l(\zeta))\chi R_V(\zeta)\chi d\tau_l(\zeta)
- 2\pi i \chi  (A_1+tA_2)\pr_l\chi \right\| =O(l^{-\delta}),$$
and this implies that for $l$ sufficiently large
\begin{equation}
\rank \int_{\gamma_l} (1+t \tau_l(\zeta))\chi R_V(\zeta)\chi d\tau_l(\zeta)
\geq 2 \rank(\chi (A_1+tA_2)\chi).
\end{equation}
But since $(-\Lz+V_0)^{k_0}A_j=0$ for $j=1,\;2$,
a unique continuation theorem (e.g. \cite{JeKe}) ensures that 
$\rank (A_1+tA_2)=\rank(\chi (A_1+t A_2)\chi)$, and 
similarily 
$$\rank \int_{\gamma_l} (1+t \tau_l(\zeta))\chi R_V(\zeta)\chi d\tau_l(\zeta)
= \rank \int_{\gamma_l} (1+t \tau_l(\zeta)) R_V(\zeta) d\tau_l(\zeta).$$
\end{proof}

\begin{lemma}\label{l:thresholdorder1} 
Let $V_0,\chi\in L^\infty_c(\Real^d)$, with $\chi V_0=V_0$.  Suppose $\RVz(\lambda)$ has a pole of order $1$ at the origin.  
Then for $l$ sufficiently large $2(\mvzz(0)-m_{00}(0))=M(I+V_0\Rz\chi,\zeta_l(0)).$
\end{lemma}
\begin{proof} We note here that the requirement that $l$ is
sufficiently large is to ensure that, using (\ref{eq:RVzsep}),
any poles of $\RVz$ at $\zeta_l(0)$
arise from poles of $\RVz$ at the origin.  Then via separation of variables it suffices to show
that $\mvzz(0)-m_{00}(0)=M(I+V_0\Rzz(\lambda)\chi,0)$.
 For $d=1$, then 
$\mvzz(0)=1$ and if $V_0$ is real-valued, this follows immediately from \cite[(2.2.31)]{DyZw}.  For complex-valued $V_0$, the proof is 
similar, if one uses the assumption that $\RVz$ has a simple pole at the origin.
For $d\geq 3$ is odd, the lemma  follows as in the proof of \cite[Theorem 3.15]{DyZw}.  
In each case, the assumption that the  pole is of order one is important.
\end{proof}

\begin{lemma}\label{l:poleorder1}
Let $V\in L^\infty_c(X)$ satisfy (\ref{eq:Vhyp}).
 Let $\epsilon>0$ be chosen
so that $\RVzz(\lambda)$ has no poles in
$\{\lambda \in \Complex:0< |\lambda|<2\epsilon\}$.
Suppose $\RVzz(\lambda)$ has a pole of order $1$ at the origin, with
residue of rank $\mvzz(0)$.  Then for $l$ sufficiently large
$$\sum_{\substack{\zeta\in \Dl(\epsilon)\\
\mV(\zeta)\not = 0}}\mV(\zeta)
\leq 2\mvzz(0).$$
\end{lemma}
\begin{proof} Let $\chi\in L^\infty_c(X)$ be independent of $\theta$ and
satisfy $\chi V=V$.
 We first claim that for any $\zeta_0\in \zhat$,
\begin{equation}
\label{eq:upbdmult}
\mV(\zeta_0)\leq M(I+V\Rz\chi,\zeta_0)+m_0(\zeta_0).
\end{equation}  If $\zeta_0$ does not correspond
to a threshold, then $m_0(\zeta_0)=0$ and this follows from the stronger Proposition \ref{p:mvsM}.  
If $\zeta_0$ does correspond to a threshold, this follows from a 
simplified adaptation of the proof of \cite[Theorem 3.15]{DyZw}.

Arguing as in the proof of Proposition \ref{p:prelimThm1}, using 
Lemmas \ref{l:multandprl}, \ref{l:RTcheck} and 
an operator Rouch\'e Theorem  (\cite[Theorem 2.2]{go-si}),
for $l$ sufficiently large 
\begin{equation}\label{eq:Mequals}\sum_{\substack{\zeta\in B_l(\epsilon)\\
M(I+ V\Rz \chi , \zeta)\not =0}}
M(I+V \Rz  \chi ,\zeta )\\
=\sum_{\substack{\zeta\in B_l(\epsilon)\\
M(I+V_0 \Rz  \chi , \zeta)\not =
0}} M(I+V_0 \Rz  \chi ,\zeta) = M(I+V_0\Rz \chi,\zeta_l(0)).
\end{equation}
But by our assumptions and Lemma
\ref{l:thresholdorder1}, for $l$ sufficiently large
$M(I+V_0\Rz \chi,\zeta_l(0))=2(\mvzz(0)-m_{00}(0)).$
Using this, (\ref{eq:Mequals}), and applying (\ref{eq:upbdmult}) completes
the proof.
\end{proof}

\vspace{2mm}
\noindent 
{\em Proof of Theorem \ref{thm:0} under
the assumption $\|V_m\|_{L^\infty}=O(m^{-\delta}).$}
Let $\epsilon>0$ be as in the statement of Lemma \ref{l:residues}.
By applying Lemmas \ref{l:multat0} and \ref{l:residues}, we see
that for $l$ sufficiently large,
$$\sum_{\substack{\zeta\in B_l(\epsilon)\\ \mV(\zeta)\not = 0}}\mV(\zeta)\geq 
\sum_{\substack{\zeta\in B_l(\epsilon)\\ m_{V_0}(\zeta)\not = 0}}m_{V_0}(\zeta)=2\mvzz(0).$$
Thus for $l$ sufficiently large
 $\RV$ has at least $2\mvzz(0)$ poles
in $B_l(\epsilon)$.  If 
$\RVzz(\lambda)$ has a simple pole at the origin,
then applying in addition Lemma \ref{l:poleorder1} we see that $\RV$ has at 
exactly $2\mvzz(0)$ poles
in $B_l(\epsilon)$.

To finish the proof of the theorem for the $L^\infty$ case we need to 
refine the estimate on the location of the resonances in $B_l(\epsilon)$.
We do this by showing that there is a $C>0$ so that there are no
resonances in $B_l(\epsilon)\setminus B_l(C l^{-\delta/r}) $ for $l$ 
sufficiently large.
This follows almost exactly the 
proof 
of Theorem \ref{thm:polefree}, point 2, with $\lambda_j$ 
replaced by $0$. The difference
here is that the bound on the singular part of $\chi \RVz \chi$ at the 
origin is given by $\|\chi \Xi(\RVz,0)\chi\|\leq C|\lambda|^{-r}$; that is,
$\mvzz(\lambda_j)$ is replaced by $r$ rather than $\mvzz(0)$.  Having made
this minor adaptation, the remainder of the proof follows without change.
\qed

\section{ Resonance-free regions, poles of $R_V$ and $R_{\overline{V}}$, and
the proofs of Corollary \ref{c:approachsequence} and Theorem \ref{thm:waveexp}}\label{ss:corollary}

Thus far we have focused on resonances
in  the sets $B_l(\rho)$, for $l$ large.  In this section we justify this
by showing that the high energy resonances ``near'' the physical space which
also have $\Re \tau_0(\zeta)>0$ lie in 
$B_l(\rho)$, for $\rho$ sufficiently large.  We do this  by 
showing the existence of large resonance-free regions in $B_l(\sqrt{2l-1})$.
We discuss $\zhat$ further, focusing on the 
region near the physical space.   We describe the relationship 
between the resolvents $R_V$ and $R_{\overline{V}}$, where $\overline{V}$ is the 
complex conjugate of $V$, see Lemma \ref{l:symmetrytype}.  This lemma 
shows that that we can understand the 
poles of $\RV$ which are near the physical space and have $\Re \tau_0(\zeta)<0$ by
understanding the poles of $R_{\overline{V}}$ which are near the physical space
and have $\Re \tau_0(\zeta)>0$.


\begin{lemma}\label{l:resfree} Let $V\in L^\infty_c(X)$.  Then for any 
$0<\gamma<1$ there are
$M_+,c_+>0$ so 
that the region:
$$U_l^+\defeq\{ \zeta\in B_l(\sqrt{2l-1}): M_+<\Re (\tau_l(\zeta))<\gamma \sqrt{2l},\; 
\Im \tau_l(\zeta)>-c_+\log \Re (\tau_l(\zeta))\}$$
contains no poles of $\RV$ for $l$ sufficiently large.
Likewise, for any $\alpha>0$, $0<\gamma<1$, there 
is a constant $M_->0$ so that
$$U_l^-\defeq\{ \zeta\in B_l(\sqrt{2l-1}):M_-<\Im(\tau_l(\zeta))< \gamma\sqrt{2l},\; 
\Re \tau_{l}(\zeta)>-\alpha\}$$
contains no poles of $\RV$ for $l$ sufficiently large.
\end{lemma}
The region $U_l^+$ is reminiscent of the logarithmic resonance-free regions 
familiar from potential scattering on $\Real^d$.  We note that 
there is substantial overlap between $U_l^+$ and $U_{l+1}^-$.
\begin{proof}
Let $\chi\in L^\infty_c(X)$ be independent of $\theta$ and satisfy 
$\chi V=V$ and $0\leq \chi \leq 1$.
To prove the lemma, we use $\chi \RV(\zeta)\chi= \chi \Rz(\zeta)(I+V\Rz(\zeta)\chi)^{-1}$
and the representation (\ref{eq:RVzsep}) via separation of variables.   

From (\ref{eq:RVzsep}) and the estimate $\|\chi \Rzz(\lambda)\chi\| \leq Ce^{(C\Im \lambda)_-}/|\lambda|$, there are constants $C_1, \; C_2$ so that 
$$\| V \Rz(\zeta)\chi\| \leq \sup_{j\in \Natural_0} \left( 
\frac{C_1e^{C_2(\Im \tau_j(\zeta))_-}}{|\tau_j(\zeta)|}\right).$$

First consider $U_l^+$.
Set $c_+=\frac{1}{C_2}-\delta_+$, where $\delta_+>0$, 
$\delta_+<1/C_2,$
and take $M_+>(2C_1)^{1/(\delta_+ C_2)}$.  Then if $\zeta\in U_l^+$, 
$\frac{C_1e^{C_2(\Im \tau_l(\zeta))_-}}{|\tau_l(\zeta)|}<1/2$.  If 
$j<l$ and $\zeta\in U^+_l$, then $|\tau_j(\zeta)|\geq |\tau_l(\zeta)|$
and  a computation shows
$$\frac{e^{C_2(\Im \tau_j(\zeta))_-}}{|\tau_j(\zeta)|}<\frac{e^{C_2(\Im \tau_l(\zeta))_-}}{|\tau_l(\zeta)|}.$$
On the other hand, for $j>l$, if $\zeta \in U_l^+$, then
$$\Re (\tau_j(\zeta))^2\leq \Re (\tau_{l+1}(\zeta))^2= (\Re \tau_l (\zeta))^2-2l
-(\Im \tau_l(\zeta))^2 -1\leq -2l(1-\gamma^2).$$
Since $\Im \tau_j(\zeta)>0$ for $j>l$ and $\zeta\in B_l(\sqrt{2l-1})$, this is enough to show that 
$\frac{C_1e^{C_2(\Im \tau_j(\zeta))_-}}{|\tau_j(\zeta)|}<1/2$ for $\zeta\in U_l^+$
and $l$ sufficiently large.  Then $\| V \Rz(\zeta)\chi\|<1/2$, and $I+V\Rz(\zeta)\chi$ is invertible.

For $U_l^-$, 
choose $M_->0$ so that $16\|V\|_{L^\infty}<M_-^2$.  Then using 
(\ref{eq:RVzsep}) and $\|\Rzz(\lambda)\| \leq 1/(\dist(\lambda^2,[0,\infty)))$ for $\Im \lambda>0$,
for $\zeta \in U^-_l$,
$$\|V\Rz(\zeta)\chi \sum_{j\geq l}\pr_j\| \leq
 \|V\|_{L^\infty}\sup_{j\geq l} \frac{1}{(\dist \tau_j^2,[0,\infty))}\leq 
8\|V\|_{L^\infty}/ M_-^2\leq 1/2.$$
Next we show that $\| V\Rz(\zeta)\sum_{0\leq j<l}\pr_j\chi\|\leq1/2$ in $U_l^-$ for 
sufficiently large $l$.   Using the orthogonality of the projections $\sum_{j\geq l} \pr_j$ and $\sum_{l^j}\pr_j $ this will complete our proof  that 
$I+V\Rz\chi$ is invertible.
Note
$$\tau_{l-1}^2= 2l-(\Im \tau_l)^2+(\Re \tau_l)^2-1+2i\Re(\tau_l)\Im(\tau_l).$$ 
Thus $|\tau_{l-1}|\geq \sqrt{(1-\gamma^2)2l}+O(1)$, and $-\Im(\tau_{l-1})\leq \frac{2 \alpha}{\sqrt{1-\gamma^2}}+O(l^{-1/2})$, so for $l$ sufficiently large we have 
$\frac{C_1e^{C_2(\Im \tau_{l-1}(\zeta))_-}}{|\tau_{l-1}(\zeta)|}<1/2$ for $\zeta\in U^-_l$.
But if $0\leq j<l-1$ and $\zeta\in U_l^-$, 
$\frac{C_1e^{C_2(\Im \tau_j(\zeta))_-}}{|\tau_j(\zeta)|}<\frac{C_1e^{C_2(\Im \tau_{l-1}(\zeta))_-}}{|\tau_{l-1}(\zeta)|}.$  This ensures that
$\| V\Rz(\zeta)\sum_{0\leq j<l}\pr_l\chi\|<1/2$  so that
$I+V\Rz(\zeta)\chi$ is invertible on $U_l^-$ for $l$ sufficiently large.
\end{proof}
We remark that we have not made an effort to optimize the results of 
Lemma \ref{l:resfree}, as in this paper we are
concentrating instead on regions near the thresholds, where, as we have seen, resonances can occur.

Before proving Corollary \ref{c:approachsequence}, we discuss $\zhat$ and 
the boundary of the physical space a bit more.  To motivate the discussion, 
consider the simpler case of the Schr\"odinger operator $-\Lz+V_0$ on 
$\Rd$, where we use $\lambda^2$ as the spectral parameter in defining
the (scattering) resolvent.  Thus, given a 
value $E>0$, there are two points, $\pm \sqrt{E}$ 
corresponding to the spectral parameter
$E$ on the boundary of the physical space,
with  $\RVzz(\pm \sqrt{E}) = (-\Lz+V_0-(\sqrt{E}\pm i0))^{-1}$.

There is a similar phenomena in the case of $-\Delta+V$ on 
$\Rd\times \Sphere^1$, but it is notationally harder to describe.  Given 
$E>0$, let $\sqrt{E}\pm i0\in \zhat$ be the points on $\zhat$ with
$\RV(\sqrt{E}\pm i0)=(-\Delta +V-E \mp i0)^{-1}$.  Equivalently, we could
define
 $\sqrt{E}\pm i0$ to the the point in $\zhat$ 
with $\tau_j(\sqrt{E}\pm i0) =\pm \sqrt{E-j^2}$ if $j^2\leq E$
and $\tau_j(\sqrt{E}\pm i0) =i\sqrt{j^2-E}$ if $j^2>E$.  
 By our definition of
$B_l(\rho)$, if $l_E=\lfloor \sqrt{E} \rfloor$ and $l_E>0$, then 
$\sqrt{E}+i0\in B_{l_E}(\sqrt{2l_E-1})$, but 
$\sqrt{E}-i0\not \in B_{l_E}(\sqrt{2l_E-1})$.  Thus there is some 
sense in which we have been studying only ``half'' of the boundary of
the physical space.  However, we shall see in Lemma \ref{l:symmetrytype} that this
suffices for understanding the behavior of the resolvent, if we consider both the resolvent of $-\Delta+ V$ and that of $-\Delta+\overline{V}$.


Thus, to fully cover points on the boundary of the physical space, we
need to define another type of open set in $\zhat,$ analogous to $B_l(\rho)$.
For $l\in \Natural,$ $\rho>0$, denote by 
$B_{l}^\pm(\rho)$ the connected component of 
$\{ \zeta \in \zhat: |\tau_l(\zeta)|<\rho \}$ 
which intersects the physical space and includes a region with
$\pm \Re \tau_0(\zeta)>0$.  With the $+$ sign, we get the 
set $B_l(\rho)$ defined in the introduction: $B_l^+(\rho)=B_l(\rho)$.  If $l_E=\lfloor \sqrt{E}\rfloor$ and $\sqrt{E_l}-l_E<\rho$, then the 
point $\sqrt{E}-i0$ corresponding to $E$ on the boundary of the physical space
as defined above has $\sqrt{E}-i0\in B_{l_E}^-(\rho)$.  Hence
any point on the boundary of the physical space lies in 
$B_0^+(1)\cup \left(\cup_{l=1}^\infty B^+_l(\sqrt{2l-1})\right) 
\cup \left( \cup_{l=1}^\infty B_l^-(\sqrt{2l-1})\right)$.  
As before, we make the choice of $\sqrt{2l-1}$ for $\rho$ as that is
the largest value of $\rho$ for which $B_l^{\pm}(\rho)$ contains only 
a single point corresponding to a threshold.
 For 
certain combinations of $l$ and $\rho$ it can happen that $B_l^+(\rho)=B_l^-(\rho)$. 

Consider a Schr\"odinger operator on $d$-dimensional Euclidean space
 with potential $V_0\in L^\infty_c(\Rd)$ and
scattering resolvent $\RVzz(\lambda)$.  When $\Im \lambda>0$, that is, $\lambda$ is 
in the physical space, $$\RVzz(\lambda)=
\left(-\Lz+V_0-\lambda^2\right)^{-1}=
\left( \left(-\Lz+\overline{V_0}-\overline{\lambda}^2\right)^{-1}\right)^*= 
\left(R_{\overline{V_0}0}(-\overline{\lambda})\right)^*.$$
Here $\overline{V_0}$, $\overline{\lambda}$ denote the usual complex conjugates.
For odd $d$ the
 identity $\RVzz(\lambda)=\left(R_{\overline{V_0}0}(-\overline{\lambda})\right)^*$
then holds by meromorphic continuation for all $\lambda \in \Complex$.  In particular,
this implies $\lambda_0$ is a pole of $\RVzz(\lambda)$ if
and only if $-\overline{\lambda_0}$ is a pole of $R_{\overline{V_0}0}(\lambda)$.
For real-valued $V$, this is the well-known symmetry of resonances for
symmetric  Schr\"odinger operators in odd dimensions.

We turn to the analog of this result for $R_V$, which
is shown in a similar way.  Suppose $\zeta$ is in the physical space,
here identified with the upper half plane, so that
$R_V(\zeta)=(-\Delta +V-\zeta^2)^{-1}$.  Thus 
$(R_{\overline{V}}(-\overline{\zeta}))^*= R_V(\zeta)$.  For general
$\zeta\in \zhat$, we define $-\zconj\in \zhat$ to be the 
point in $\zhat$ satisfying 
$\tau_j(-\zconj)=-\overline{\tau_j(\zeta)}$ for all $j$.  This is an 
antiholomorphic mapping, and if $\zeta$ is in the physical space, 
identified with the upper half plane, the mapping $\zeta \mapsto -\zconj$
agrees with the mapping $\zeta\mapsto -\overline{\zeta}$.
Then the identity 
\begin{equation}\label{eq:RVconj}
(R_{\overline{V}}(-\zconj))^*= R_V(\zeta),\; \text{where} \; 
\tau_j(-\zconj)=-\overline{\tau_j}(\zeta),\; \text{for all $j\in \Natural_0$}
\end{equation}
holds for all $\zeta\in \zhat$ by meromorphic continuation.
  In particular, this means that $\zeta_0\in \zhat$ is a pole of $R_V(\zeta)$ 
if and only if $-\zconj_0$ is a pole of $R_{\overline{V}}(\zeta)$. 
Note that if $\zeta\in B_l^+(\rho)=B_l(\rho)$, then 
$-\zconj\in B_l^-(\rho)$.  Thus to study the poles of $R_V(\zeta)$ in 
$B_l^-(\rho)$ it suffices to study the poles of $R_{\overline{V}}(\zeta)$ in $B_l^+(\rho)=B_l(\rho)$.  
Likewise, an estimate on $R_{\overline{V}}$ in $B_l^+(\sqrt{2l-1})$ implies an estimate on $\RV$ in 
$B_l^-(\sqrt{2l-1})$.

We summarize these results in the following lemma.
\begin{lemma}\label{l:symmetrytype}
If $V_0\in L^\infty_c(\Rd)$, then $\lambda_0$ is a pole of $\RVzz(\lambda)$ if
and only if $-\overline{\lambda_0}$ is a pole of $R_{\overline{V_0}0}(\lambda)$.  Let $V\in L^\infty_c(X)$.  Then $\zeta_0\in \zhat$ is a pole of $R_V(\zeta)$ if
and only if $-\zconj_0$ is a pole of $R_{\overline{V}}(\zeta)$.
 Here
$\overline{\lambda_0}$, $\overline{V}$, and $\overline{V_0}$ are respectively the 
complex conjugates of $\lambda_0$, $V$, and $V_0$, and 
$-\zconj$ is as 
defined in (\ref{eq:RVconj}).
\end{lemma}

We define a distance on $\hat{Z}$ as follows: for $\zeta,\;\zeta'\in \hat{Z}$,
\begin{equation}
d_{\zhat}(\zeta,\zeta')\defeq \sup_{j}| \tau_j(\zeta)-\tau_j(\zeta')|.
\end{equation}
That this is  well-defined and 
a metric is shown in \cite[Section 5.1]{cd}.  Note that 
if $\zeta,\; \zeta'\in \zhat$ satisfies $\tau_j(\zeta)\not = -\tau_j(\zeta')$,
then since $\tau_j(\zeta)^2-\tau_j(\zeta')^2=\tau_l(\zeta)^2-\tau_l(\zeta')^2$,
 $$|\tau_j(\zeta)-\tau_j(\zeta')|=|\tau_l(\zeta)-\tau_l(\zeta')|
\left| \frac{\tau_l(\zeta)+\tau_l(\zeta')}{\tau_j(\zeta)+\tau_j(\zeta')}\right|.
$$
In particular, this implies that for any $\rho>0$ there is an $L=L(\rho)$
so that if $l\geq L$ and  $\zeta,\; \zeta'\in B_l(\rho)$ then $d_{\zhat}(\zeta,
\zeta')=|\tau_l(\zeta)-\tau_l(\zeta')|$.

\vspace{2mm}
\noindent
{\em Proof of Corollary \ref{c:approachsequence}.}  
Recall our hypotheses include that $V$ is real-valued, ensuring
 that $V_0$ is real-valued as well. 


The operator-valued function
$R_V(\zeta)$ has a sequence $\{ \zsharp_j\}$ of  poles satisfying $|\tau_0(\zsharp_j)|\rightarrow \infty$ as $j\rightarrow \infty$ and 
$d_{\zhat}(\zsharp_j,\text{physical space})
\rightarrow 0$ only if either $\RV(\zeta)$ has infinitely many poles 
in $\cup_{l=1}^\infty B_l(\sqrt{2l-1})$ or infinitely many poles in 
$\cup_{l=1}^\infty B_l^{-}(\sqrt{2l-1})$ (or both).  If $\RV(\zeta)$ has infinitely many poles 
in $\cup_{l=1}^\infty B_l^{-}(\sqrt{2l-1})$, then by Lemma \ref{l:symmetrytype},
$R_{\overline{V}}(\zeta)=\RV(\zeta)$ has infinitely many poles in $\cup_{l=1}^\infty B_l(\sqrt{2l-1})$.
Thus it suffices to study sequences of poles in $\cup_{l=1}^\infty B_l(\sqrt{2l-1})$.

Note that while $B_l(\sqrt{2l-1})$ contains only a single threshold, 
$B_l (\sqrt{2l-1})$ and $B_{l+1}(\sqrt{2l+1})$
are not disjoint and in fact have substantial overlap which contains an interval of the 
continuous spectrum. Moreover, for $l$ sufficiently large the 
sets $U_l^+$ and $U_{l+1}^-$  of Lemma \ref{l:resfree}
have nontrivial intersection.   Applying Lemma \ref{l:resfree} we see that in order to have a sequence of 
resonances contained in $\cup_{l=1}^\infty B_l(\sqrt{2l-1})$ and 
approaching the continuous spectrum (and with $|\tau_0|\rightarrow \infty$),
 the resonances must lie in 
$\cup_{l=1}^\infty B_l(M)$ for some $M$.  But then the corollary follows
from an application of Theorems \ref{thm:poleexist}, \ref{thm:polefree},
and \ref{thm:0}.
\qed

We now have the ingredients we need to prove Theorem \ref{thm:waveexp}.

  \vspace{2mm}
\noindent {\em Proof of Theorem \ref{thm:waveexp}}.
The hypotheses on $-\frac{d^2}{dx^2}+V_0$ and the expression (\ref{eq:RVzsep})
mean that the resolvent $\RVz(\zeta)$ has no poles on the boundary
of the physical space.  Moreover, since for any  
$\tilde{\chi}\in C_c^\infty(\Real)$, 
there is a constant $C$ so that $\| \tilde{\chi} \RVzz(\lambda)\tilde{\chi}\|\leq C$ for all
$\lambda \in \Real\cup i[0,\infty)$, for  
  any $\chi \in C_c^\infty(X)$ 
there is a  $C_1>0$ so that
$\| \chi \RVz(\zeta)\chi\|\leq C_1$ for all $\zeta$ in the boundary of the 
physical space.  

Corollary \ref{c:approachsequence}
shows that there are no poles of the resolvent $\RV$ in the continuous spectrum
at high energy.
 Proposition \ref{p:Rest} and Lemma \ref{l:bdryphys} show
that  when $\zeta$ is 
in the boundary of the physical space and $\zeta \in B_l(\sqrt{2l-1})$, 
the cut-off resolvent of $-\Delta +V$ satisfies
$\| \chi \RV(\zeta)\chi- \chi \RVz(\zeta)\chi\|=O(l^{-1/2}).$ %
Thus $\| \chi \RVz(\zeta)\chi\|$ is uniformly bounded on the boundary
of the physical space when $|\tau_0(\zeta)|$ is 
sufficiently large.  Hence by \cite[Theorem 5.6]{cd} the hypotheses of \cite[Theorem 4.1]{cd2}
hold.  Theorem \ref{thm:waveexp}  then follows directly from \cite[Theorem 4.1]{cd2}.
\qed

\section{Larger neighborhoods of the threshold $l^2$}
In this section we consider poles of $\RV(\zeta)$
 in neighborhoods $B_l(\alpha \log l)$ 
and $B_l(\alpha (\log l)^{1-\epsilon})$ of 
the $l$th threshold.  We prove 
 Theorem \ref{thm:V0iszero} for potentials with $V_0\equiv 0$,
and the related, but weaker, Theorem
\ref{thm:biggerdisk}
which holds for a general potential $V\in L^\infty_c(X)$.

The proof of Theorem \ref{thm:V0iszero} is similar
to that of the proof of Theorem \ref{thm:polefree} for $L^\infty$ potentials.

\vspace{2mm}
\noindent
{\em Proof of Theorem \ref{thm:V0iszero}.}  
Choose $\chi \in L^\infty_c(X)$, $\chi V=V$, and $\chi$ independent of $\theta$.
We write
\begin{multline}\label{eq:rewriteproduct}
\chi \Rz V\Rz \chi = \chi \Rz\pr_l  V\Rz \pr_l \chi+ 
\chi \Rz (1-\pr_l )V\Rz \pr_l\chi\\+\chi \pr_l \Rz V\Rz (1-\pr_l )\chi +
\chi \Rz (1-\pr_l )V\Rz (1-\pr_l )\chi.
\end{multline}
Let $\alpha'>0$,
and let $\zeta \in B_l(\alpha' |\log l|)$, where $l$ is large enough that
$B_l(\alpha' |\log l|)$ contains only a single point of $\zhat$ which
corresponds
to a threshold.
Let $\zeta \in B_l(\alpha'\log l)$ satisfy $|\tau_l(\zeta)|\geq 1$.  Then
by Lemma \ref{l:lognbhd}, $\| \chi \Rz(\zeta) (1-\pr_l )\chi\|=O(l^{-1/2})$, 
and by (\ref{eq:RzzExplicit}) and \cite[Theorem 3.1]{DyZw}, 
$\| \chi \Rz(\zeta)\pr_l\chi\|=O(e^{C (\Im \tau_l(\zeta))_{-}}/|\tau_l(\zeta)|)$ for some $C>0$. 
Using this  estimate and $\pr_l V \pr_l =O(l^{-\delta})$ in (\ref{eq:rewriteproduct}) shows
$ \|\chi \Rz(\zeta)  V\Rz(\zeta) \chi \| = O(l^{-\delta} e^{2C (\Im \tau_l(\zeta))_{-}})$.  Thus from (\ref{eq:rewriteproduct}) there is a $C_1>0$ so that $I+V\Rz(\zeta)\chi$ is invertible if
$l$ is sufficiently large,
$\zeta\in B_l(\alpha' \log l)$,
$|\tau_l(\zeta)|\geq 1$, and 
$e^{2C (\Im \tau_l(\zeta))_{-}}\leq C_1 l^{\delta}$. 
 This last may be ensured by requiring
$ |\tau_l|\leq \alpha \log l$, for suitably chosen $\alpha>0$,
$\alpha\leq \alpha'$, and taking $l$ sufficiently large.  Recall that
$-\Delta +V$ has no resonances in regions where $I+V\Rz \chi$ is invertible,
see Proposition \ref{p:mvsM}.

Applying Theorems \ref{thm:polefree} and \ref{thm:0}
shows that if $d=1 $ there is a $c_0>0$ so that when $l$ is sufficiently large
the region 
$\{\zeta\in B_l(\alpha \log l):\;  1\geq |\tau_l(\zeta)|>c_0l^{-\delta}\}$ 
contains no resonances, and if $d>1$ there are no resonances in $B_l(1)$ 
for $l$ sufficiently large.
\qed

A similar proof gives the next theorem.
\begin{thm}\label{thm:biggerdisk} 
Let $V\in L^\infty_c(X)$ satisfy (\ref{eq:Vhyp})
and let
 $\epsilon>0$.  Then there is a $c_0=c_0(\epsilon, V)>0$ 
so that for $l$ sufficiently large, the region
$$\{ \zeta \in B_l(c_0(\log l)^{1/(d+\epsilon)}): \: 
|\tau_l(\zeta)-\lambda'|\geq (1+|\lambda'|^2)^{-(d+\epsilon)/2}\; \text{for every}
\; \lambda'\in \Complex:\; \mvzz(\lambda')>0\}$$
contains no poles of $\RV(\zeta)$.
\end{thm}
\begin{proof}
We assume $\Vd=V-V_0\not \equiv 0$, since otherwise there is nothing to prove.

Choose $\chi \in L^\infty_c(X)$ so that $\chi V=V$ and $\chi$ is independent of 
$\theta$.  We may think of $\chi \in L^\infty_c(\Rd)$ as well.

Set 
$$A_\epsilon\defeq \{ \lambda \in \Complex:\; |\lambda-\lambda'|\geq (1+|\lambda'|^2)^{-(d+\epsilon)/2}\; \text{for every}
\; \lambda'\in \Complex:\; \mvzz(\lambda')>0\}.$$
We shall use, from the proof of \cite[Theorem 3.54]{DyZw}, 
there is a $C>0$ so that 
\begin{equation}
\| (I+V_0\Rzz(\lambda))^{-1}\| \leq C\exp(C|\lambda|^{d+\epsilon})\; 
\text{if }\; \lambda \in A_\epsilon.
\end{equation}
Choose $\alpha'>0$.  If $\zeta \in B_l(\alpha' \log l)$,
$$\chi \RVz(\zeta)\pr_l\chi = \chi \RVzz(\tau_l(\zeta))\pr_l\chi
= \chi \Rzz(\tau_l(\zeta))\chi(I+V_0\Rzz(\tau_l(\zeta))\chi)^{-1}\pr_l.$$
Thus, if  $\zeta \in B_l(\alpha' \log l)$ with $\tau_l\in A_\epsilon$
and $|\tau_l(\zeta)|\geq 1$, then
\begin{equation}\label{eq:prlbd}
\|\chi \RVz(\zeta)\pr_l\chi\|\leq C \exp(C (\Im \tau_l(\zeta)_-))\exp\left(C|\tau_l(\zeta)|^{d+\epsilon}\right)\leq C \exp(C|\tau_l(\zeta)|^{d+\epsilon}).
\end{equation}
Here and below
 we allow the constant $C$ to change from line to line, and note that 
it depends on $V,\; \epsilon,$ and $\chi$, but not $l$.

Let $\zeta \in B_l(\alpha' \log l)$ with $\tau_l\in A_\epsilon$
and $|\tau_l(\zeta)|\geq 1$.
Writing $\chi \RVz \chi$ as in (\ref{eq:rewriteproduct}) and applying
Lemma \ref{l:lognbhd} and (\ref{eq:prlbd}), we find that 
for these $\zeta$, if $l$ is sufficiently large
\begin{equation}\label{eq:badloc}
\|\chi \RVz(\zeta)\Vd \RVz(\zeta) \chi\|\leq  C_1l^{-\delta} \exp (C_1|\tau_l(\zeta)|^{d+\epsilon})
\end{equation}
for some $C_1$.
Now we can choose $c_0>0$ sufficiently small and $L>0$ sufficiently large so  that
$$
\text{if}\; |\tau_l(\zeta)|\leq c_0(\log l)^{1/(d+\epsilon)}\; \text{and}\; l>L\; \text{then}\;
   C_1 l^{-\delta} \exp (C_1|\tau_l(\zeta)|^{d+\epsilon})\leq 1/2$$
  ensuring that 
$I+\Vd \RVz(\zeta)\chi$ is invertible.

Recalling that with $\Vd$ nontrivial if
 $I+\Vd\RVz(\zeta)\chi$ is invertible then $\zeta$
is not a resonance of $-\Delta+V$ proves the theorem.
\end{proof}

\section{Expansion of $\pr_l(I+\Vd\RVz^{\reg}\chi)^{-1}\Vd\pr_l$ for smooth $V$}
\label{s:smoothprelims}
This section contains preliminary computations which allow
us to refine some of our results when $V$ is smooth.  We begin with a
straightforward lemma about Schr\"odinger operators on $\Real^d$.

\begin{lemma} \label{l:RVzzexpand}Let $V_0,\; \chi\in C^\infty_c(\Rd)$, and $J\in \Natural$.
Then as an operator from $H^s(\Rd)$ to $H^{s-2J}(\Rd)$,
\begin{equation}\label{eq:RVzzexp} \chi \RVzz(\lambda)\chi
=-\sum_{j=1}^J  \frac{1}{\lambda^{2j}}\chi
 \left(-\Lz+V_0\right)^{j-1}\chi +  \frac{1}{\lambda^{2J}}
\chi \RVzz(\lambda) \left(-\Lz+V_0\right)^{J}\chi.
\end{equation}
\end{lemma}
\begin{proof}
First assume $\lambda$ is in the physical region, that is, 
$\Im \lambda>0$.  Then the $J=1$ case follows
from rearranging the equality
$$\left( -\Lz+V_0-\lambda^2\right)\RVzz(\lambda)=\RVzz(\lambda)\left( -\Lz+V_0-\lambda^2\right) =I$$
to get 
$$\RVzz(\lambda)=
\frac{1}{\lambda^2}\left(-I + \RVzz(\lambda)\left( -\Lz+V_0\right)\right).$$
The general case follows by induction.

Since both sides of (\ref{eq:RVzzexp}) have meromorphic continuations to the 
complex plane, the equality holds for all $\lambda$.
\end{proof}

 We shall use the following
Hilbert spaces: for $n\in \Natural_0$
$$H_{(0,n)}(X)\defeq\left\{ u\in L^2(X): \; \frac{\partial^\alpha}{\partial x^\alpha}u\in L^2(X) \;\text{if}\;
 |\alpha|\leq n\right\}\; \text{with} \; \|u\|^2_{H_{(0,n)}}= \sum_{|\alpha|\leq n}
 \left\| \frac{\partial ^\alpha}{\partial x^\alpha} u\right\|^2_{L^2(X)}.$$
Here we use the usual muliindex notation for $\alpha=(\alpha_1,...,\alpha_d)$.
This allows us to indicate mapping properties of operators which
act differently in the $x$ and $\theta$ variables.

One of the main results of this section is the following proposition.  Recall
that $\RVz^{\reg}(\zeta)=\RVz^{\reg}(\zeta;\lambda_0,l)$ is defined in (\ref{eq:RVzreg}).

\begin{prop}\label{p:firstorderapprox}
Let $V,\;\chi \in C_c^\infty(X)$ satisfy $\chi V=V$.  In addition,
suppose $\chi$ is independent of $\theta$. Let $\lambda_0\in \Complex$,
and suppose $\RVzz(\lambda)$ is analytic on $0<|\lambda-\lambda_0|\leq\epsilon$.
 Then, for $\RVz^{\reg}(\zeta)=\RVz^{\reg}(\zeta;\lambda_0,l)$ and
 $\zeta \in \Dl(\lambda_0,\epsilon)$
\begin{multline*}
\left\|\pr_l(I+\Vd\RVz^{\reg}(\zeta)\chi)^{-1} \Vd\pr_l \right. \\+ \left.
\frac{1}{l^2}\sum_{\substack{k\in \Integers\\k\not = 0}} \left( \frac{\tau_l^2-k^2}{4k^2}V_{-k}V_k-
\frac{V_{-k}}{4k^2} \left(-\Lz+V_0\right)V_k \right) \pr_l\right\|_{H_{(0,6)}(X)\rightarrow L^2(X)}= O(l^{-3})
\end{multline*}
where the error is uniform on $\Dl(\lambda_0,\epsilon)$ for $l$ sufficiently
large.
\end{prop}
To prove this proposition we use Lemmas \ref{l:oneRV0}-\ref{l:bigj}.  In 
each of these, $V$, $\lambda_0$, $\RVz^{\reg}(\zeta)$, and $\epsilon$ are as
in Proposition \ref{p:firstorderapprox}.  These computations rely on the identity $e^{\pm ik \theta} e^{\pm i l \theta}= e^{\pm i (k+l)\theta}$ and hence
use the structure of the eigenfunctions of the Laplacian on $\Sphere^1$ in an essential way.

For $l\in \Natural$,
let $\pr_{l\pm}:L^2(X)\rightarrow L^2(X)$ denote orthogonal projection onto $L^2(\Rd_x)e^{\pm il\theta}$,
so that 
$$(\pr_{l\pm}f)(x,\theta)=\frac{1}{2\pi}\int_0^{2\pi}f(x,\theta')e^{\pm il(\theta-\theta')}d\theta'.$$
  For $l>0$,
$\pr_l=\pr_{l+}+\pr_{l-}$.

\begin{lemma}\label{l:oneRV0}  
Under the hypotheses of Proposition \ref{p:firstorderapprox},
\begin{multline*}
\left\|\pr_l \Vd \RVz^{\reg}(\zeta)\Vd\pr_l-
\frac{1}{l^2}\sum_{k\in \Integers,
k\not = 0} \left( \frac{\tau_l^2-k^2}{4k^2}V_{-k}V_k-
\frac{V_{-k}}{4k^2} \left(-\Lz+V_0\right)V_k \right) \pr_l \right\|_{H_{(0,n+6)}\rightarrow H_{(0,n)}}\\=O(l^{-3})
\end{multline*}
uniformly
for $\zeta \in \Dl(\lambda_0,\epsilon)$ when $l$ is sufficiently large.
\end{lemma}
\begin{proof}
Since $V\in C_c^\infty(X)$, $\|V_m\|_{L^\infty}=O(m^{-N})$ for any $N$, so 
$\|\pr_l\Vd \pr_l\|=O(l^{-N})$.  Thus, choosing $l$ sufficiently
large that (\ref{eq:RVzreglsuff}) holds, it suffices to consider 
$\pr_l \Vd \RVz^{\reg}(\zeta)(I-\pr_l)\Vd \pr_l= 
\pr_l \Vd \RVz(\zeta)(I-\pr_l)\Vd \pr_l$.

Then
\begin{align*}
\pr_l\Vd \RVz^{\reg}(I-\pr_l)\Vd \pr_l& = 
\sum_\pm \sum_{\substack{k\in \Integers\\
0<|k|, k\not =- l}}
V_{\mp k}\RVzz(\tau_{l+k} ) V_{\pm k}\pr_{l\pm}\\ & =\sum_\pm \sum_{\substack{k\in \Integers\\
0<|k|<l^{1/2}}}
V_{\mp k}\RVzz(\tau_{l+k} ) V_{\pm k}\pr_{l\pm}+O_{L^2\rightarrow L^2}(l^{-N}).
\end{align*}
Here we use the rapid decay of $\|V_m\|$ to bound the error obtained when we restrict the
values of $k$ in the sum.
Using Lemma \ref{l:RVzzexpand} with $J=3$ gives
\begin{equation}\label{eq:firsttermexp1}
\pr_l\Vd \RVz^{\reg}(I-\pr_l)\Vd \pr_l =\sum_\pm \sum_{\substack{k\in \Integers\\
0<|k|<l^{1/2}}}
V_{\mp k}\left(\frac{-1}{\tau_{l+k}^2}-\frac{1}{\tau_{l+k}^4}
\left(-\Lz+V_0\right)\right) V_{\pm k}\pr_{l\pm}+O(l^{-3})
\end{equation} where the error is as an operator from $H_{(0,n+6)}(X)$
to $H_{(0,n)}(X)$ and is uniform in $\Dl(\lambda_0,\epsilon)$.
  Since we have restricted $|k|$ to be relatively small 
compared with $l$, we can expand $\tau_{l\pm k}$ asymptotically in $l$.
Thus, with each sum over $k\in \Integers$ with $0<|k|<l^{1/2}$,
using $\tau_{l\pm k}^2=\tau_l^2\mp 2lk-k^2$ gives
\begin{align}\label{eq:first} \nonumber 
\sum_{\substack{
0<|k|<l^{1/2}}} \frac{1}{\tau_{l+k}^2} V_{-k}V_k& =\frac{1}{2}
\sum_{\substack{
0<|k|<l^{1/2}}}
\left(\frac{1}{\tau_{l+k}^2}+\frac{1} {\tau_{l-k}^2}\right) V_{-k}{V_k}
\\ 
\nonumber & = \sum_{0<|k|<l^{1/2}}\frac{ \tau_l^2-k^2}{(\tau_l^2-k^2)^2-4k^2l^2}V_{-k}V_k
\\ 
 &
= \frac{-1}{4l^2}\sum_{0<|k|<l^{1/2}}\left( \frac{\tau_l^2-k^2}{k^2}
\right)V_{-k}V_k+ O(l^{-4}).
\end{align}  Here and  below the error is uniform in $\Dl(\lambda_0,\epsilon)$
when $l$ is sufficiently large.

For the second term in (\ref{eq:firsttermexp1}), we write
\begin{align*} 
\sum_{0<|k|<l^{1/2}}\frac{1}{\tau_{l+k}^4}V_{\mp k}\left(-\Lz +V_0\right)
V_{\pm k} & = \sum_{0<|k|<l^{1/2}}\frac{1}{(\tau_{l}^2-2lk-k^2)^2}
V_{\mp k}\left(-\Lz+V_0\right)
V_{\pm k}\nonumber \\ & 
= \frac{1}{4l^2}\sum_{0<|k|<l^{1/2}} \frac{1}{k^2}V_{\mp k}\left(-\Lz +V_0\right)V_{\pm k}
+O(l^{-3}).
\end{align*}
Note that 
\begin{equation}\label{eq:second}
\sum_{0<|k|<l^{1/2}} \frac{1}{k^2}  V_{\mp k}\left(-\Lz+V_0\right)
V_{\pm k} = \sum_{0<|k|<l^{1/2}} \frac{1}{k^2}  V_{- k}\left(-\Lz+V_0\right)
V_{k},\end{equation}
since the sum is over $k\in \Integers$, with $0<|k|<l^{1/2}$.  
The rapid decay in $m$ of $\|V_m\|_{C^p}$ means we can replace the sums
in (\ref{eq:first}) and (\ref{eq:second}) over
$0<|k|<l^{1/2}$ by sums over all nonzero $k\in \Integers$, with an error which is $O(l^{-N})$.
\end{proof}

The next lemma is an algebraic identity.
\begin{lemma}\label{l:sumis0}  For any $V\in C_c^\infty(X)$
$$\sum_{\substack{m,j\in \Integers\\m,j\not = 0, m\not = -j}} \frac{1}{j(j+m)}V_mV_jV_{-m-j}=0.$$
\end{lemma}
\begin{proof}
We prove the lemma by showing that for each $j_0\not =0$, $m_0\not =0$ the 
coefficient of $V_{j_0}V_{m_0}V_{-j_0-m_0}$ in the sum is $0$.

If $m_0\not = \pm j_0$, then there are six possibilities for the pair $(j,m)$ 
which will give a term with
$V_{m_0}V_{j_0}V_{-m_0-j_0}$: $(j_0,m_0)$, $(m_0,j_0)$,
$(-m_0-j_0,m_0)$, $(m_0,-j_0-m_0)$, $(j_0,-m_0-j_0)$, $(-m_0-j_0,j_0)$.
Thus the sum of the coefficients of $V_{m_0}V_{j_0}V_{-m_0-j_0}$
is 
$$\frac{1}{j_0(j_0+m_0)}+\frac{1}{m_0(j_0+m_0)}+ \frac{1}{j_0(j_0+m_0)}
-\frac{1}{j_0m_0}-\frac{1}{j_0m_0}+ \frac{1}{m_0(j_0+m_0)}=0.$$
A similar argument when $j_0=m_0$ shows the coefficient of $V_{j_0}^2V_{-2j_0}$
is $0$ as well.
\end{proof}

\begin{lemma} \label{l:2RVzs}
Under the hypotheses of Proposition \ref{p:firstorderapprox}, if $l$ 
is sufficiently large
 $$\|\pr_l (\Vd \RVz^{\reg} )^2 \Vd \pr_l\|_{H_{(0,n+6)(X)}\rightarrow H_{(0,n)(X)}}=O(l^{-3})\; \text{uniformly for}\; \zeta \in \Dl(\lambda_0,\epsilon).
$$
\end{lemma}
\begin{proof}
Again we use that $\pr_l \Vd\pr_l=O(l^{-N})$ for any $N$.  This implies
\begin{equation*}
\pr_l (\Vd \RVz^{\reg} )^2\Vd\pr_l = \pr_l (\Vd \RVz^{\reg} (I-\pr_l))^2 \Vd\pr_l
+O_{L^2\rightarrow L^2}(l^{-N}).
\end{equation*}
Note that for $\zeta \in \Dl(\lambda_0,\epsilon)$
and $l$ sufficiently large, $\RVz^{\reg}(\zeta) (I-\pr_l)=\RVz(\zeta) (I-\pr_l)$.
Then 
\begin{align}& 
\pr_l (\Vd\RVz (I-\pr_l))^2 \Vd\pr_l\nonumber\\ & 
=\pr_l \sum_{\pm}
 e^{\pm i (j+k+m)\theta}
\sum_{\substack{k,m,j\in \Integers\\ k,\;m+k\not = 0, -2l\\
m,j\not =0}}V_{\pm j}\RVzz(\tau_{l+k+m})V_{\pm m}
\RVzz(\tau_{l+k})V_{\pm k} \pr_{l\pm} \nonumber
\\ &
 =\sum_{\pm}\sum_{\substack{
k,\; m+k\not = 0, -2l\\
m\not = 0, k,m\in \Integers}}V_{\mp(k+m)}\RVzz(\tau_{l+k+m})V_{\pm m}
\RVzz(\tau_{l+k})V_{\pm k} \pr_{l\pm} +O(l^{-N}).
\end{align}
By Lemma \ref{l:RVzzexpand}, for $k,\; m+k\not =0,-2l$,
\begin{multline*}
\left \| \chi \RVzz (\tau_{l+k+m})V_{\pm m}
\RVzz(\tau_{l+k})V_{\pm k} - \frac{1}{\tau_{l+k+m}^2 \tau_{l+k}^2}\chi
V_{\pm m}V_{\pm k} \right\|_{H^{n+6}(\Real^d)\rightarrow H^n(\Real^d)}
\\=O(l^{-3}\|V_{\pm k}\|_{C^4}
\|V_{\pm m}\|_{C^4}).
\end{multline*}
This implies (with sums still over $\Integers$)  using $\|V_m\|_{C^p}=O(m^{-N})$,
\begin{align}\label{eq:squared}
\pr_l (\Vd\RVz (I-\pr_l))^2 \Vd\pr_l\nonumber & 
 =\sum_{\pm}\sum_{k,m,k+m\not =0, -2l}\frac{1}{\tau_{l+k+m}^2\tau_{l+k}^2}V_{\mp(k+m)}V_{\pm m}V_{\pm k}\pr_{l\pm} +O(l^{-3}) \nonumber\\
 & = \sum_{\pm}\sum_{0<|k|,|k+m|<l^{1/2},m\not = 0}\frac{1}{\tau_{l+k+m}^2 \tau_{l+k}^2} V_{\mp(k+m)}V_{\pm m}V_{\pm k}\pr_{l\pm}+O(l^{-3}) \nonumber \\
& = \sum_{\pm}\sum_{0<|k|,|k+m|<l^{1/2},m\not =0 } \frac{1}{4l^2k(k+m)}V_{\mp(k+m)}V_{\pm m}V_{\pm k}\pr_{l\pm} + O(l^{-3})\nonumber\\
& = \sum_{\pm}\sum_{0\not = k,k+m,m} \frac{1}{4l^2k(k+m)}V_{\mp(k+m)}V_{\pm m}V_{\pm k} \pr_{l\pm}+ O(l^{-3} ).
\end{align}
Here errors are as operators from $H_{(0,n+6)}(X)$ to $H_{(0,n)}(X)$, and are 
uniform in $\Dl(\lambda_0,\epsilon)$ when $l$ is sufficiently large.
But the final  sum in (\ref{eq:squared}) is $0$ by Lemma \ref{l:sumis0}.
\end{proof}

\begin{lemma}\label{l:bigj}Under the hypotheses of Proposition \ref{p:firstorderapprox} for $j\geq 3$, $j\in \Natural$,  and $l$ sufficiently large
$$\left\|\left( \Vd \RVz^{\reg}(\zeta)\right)^j\Vd \pr_l\right\|_{H_{(0,8)}(X)\rightarrow L^2(X)} =O(l^{-3})$$ 
uniformly for  $\zeta \in \Dl(\lambda_0,\epsilon)$.
\end{lemma}
\begin{proof}
By Lemma \ref{l:2RVzs},
 $\|\pr_l (\Vd \RVz^{\reg} )^2 \Vd\pr_l\|_{H_{(0,n+6)}\rightarrow H_{(0,n)}}=O(l^{-3})$.  This 
gives
\begin{align}\label{eq:power3,1}
(\Vd \RVz^{\reg})^3\Vd\pr_l& = \Vd \RVz^{\reg}(I-\pr_l)(\Vd \RVz^{\reg})^2\Vd \pr_l
+ \Vd \RVz^{\reg}\pr_l(\Vd \RVz^{\reg})^2 \Vd \pr_l \nonumber \\
& = \Vd \RVz^{\reg}(I-\pr_l)(\Vd \RVz^{\reg})^2\Vd \pr_l +O(l^{-3})
\end{align}
as an operator from $H_{(0,n+6)}(X)$ to $H_{(0,n)}(X)$.  
Using again that $\pr_l$ commutes with $\RVz$ and $\pr_l\Vd \pr_l=O(l^{-N})$ for
any $N$ gives
$$
\left( \Vd \RVz^{\reg}\right)^2\Vd \pr_l
= \left( \Vd \RVz(I-\pr_l)\right)^2\Vd \pr_l +
\Vd \RVz^{\reg}\pr_l\Vd \RVz (I-\pr_l)\Vd \pr_l +O_{L^2\rightarrow L^2}(l^{-N}).$$
Using this in (\ref{eq:power3,1}) yields
\begin{multline}\label{eq:cube}\left( \Vd \RVz^{\reg}\right)^3\Vd \pr_l
\\ = \left( \Vd \RVz(I-\pr_l)\right)^3\Vd \pr_l + \Vd \RVz(I-\pr_l)\Vd\RVz^{\reg}\pr_l\Vd \RVz(I-\pr_l)\Vd\pr_l+O_{L^2\rightarrow L^2}(l^{-3}).
\end{multline}
For large $l$ 
Lemma \ref{l:RVzzexpand} applied with $J=1$ shows $\| \left( \Vd \RVz(I-\pr_l)\right)^3\Vd \pr_l \|_{H_{(0,6)}(X) \rightarrow L^2(X)}=O(l^{-3})$.
Choose $\chi \in C_c^\infty(X)$ independent of $\theta$ so that 
$V\chi=V$.  We write the second
term on the right in (\ref{eq:cube}) as the composition of three operators, with the 
grouping indicated below by  the large parentheses:  
\begin{multline}
 \Vd \RVz(I-\pr_l)\Vd\RVz^{\reg}\pr_l\Vd \RVz(I-\pr_l)\Vd\pr_l \\
= \Big(\Vd \RVz(I-\pr_l)\Vd\Big) \Big(\chi \RVz^{\reg}\pr_l \chi\Big)
\Big(\pr_l \Vd \RVz(I-\pr_l)\Vd\pr_l\Big).
\end{multline}
By Lemma \ref{l:RVzzexpand} $\| Vd \RVz(I-\pr_l)\Vd \|_{H_{(0,N+2)}\rightarrow H_{(0,n)}}= O(l^{-1})$.
 The second operator, $ \chi \RVz^{\reg}\pr_l \chi$, is bounded. By lemma \ref{l:oneRV0}, the third is 
$O(l^{-2})$ as an operator from $H_{(0,n+6)}$ to $H_{(0,n)}$ by Lemma \ref{l:oneRV0}.  Thus we have proved the lemma when $j=3$.

The case of $j>3$ follows from the $j=3$ case.
\end{proof}

We now can prove Proposition \ref{p:firstorderapprox}.

\vspace{2mm}
\noindent {\em Proof of Proposition \ref{p:firstorderapprox}.}  For $l$ 
sufficiently large,  on
 $\Dl(\lambda_0,\epsilon)$
$$\pr_l(I+\Vd \RVz^{\reg}(\zeta)\chi)^{-1}\Vd \pr_l= \pr_l \sum_{j=0}^\infty
(-\Vd \RVz^{\reg}(\zeta)\chi)^{j}\Vd \pr_l.$$
The proposition then follows from an application of Lemmas \ref{l:oneRV0},
\ref{l:2RVzs}
and \ref{l:bigj}, and recalling that $\| \pr_l \Vd \pr_l\|=O(l^{-N})$.
\qed

\vspace{2mm}

The proof of Theorem \ref{thm:leadcorrection} uses
the   next lemma, which  computes an expression related to the leading term of 
$\pr_l (I+\Vd \RVz^{\reg}(\zeta_l(z))^{-1}\Vd \pr_l$.
\begin{lemma}\label{l:ucomp}  
Suppose $V\in C_c^\infty(X)$ and $u\in C^\infty(\Rd)$ satisfies 
$\left(-\Lz+V_0-\lambda_0^2\right)u=0$.  Then
\begin{multline*}
 - \int_{\Real^d} u \left((z^2-k^2)V_{-k}V_k u
-V_{-k}(-\Lz+V_0)(V_ku)\right)  dx\\
= \int_{\Real^d} \left((k^2+\lambda_0^2-z^2)u^2  V_{-k}V_k+ u^2 \grz V_{-k}\cdot \grz V_k\right) dx.
\end{multline*}
\end{lemma}
\begin{proof}
We first compute
$\int_\Real  u  V_{-k}\left(-\Lz+V_0\right)(V_k u)dx.$
Expanding and then integrating by parts,
\begin{align}\label{eq:onesetderiv}
& \int_{\Real^d}  u V_{-k}\left(-\Lz+V_0\right)(V_k u)\;dx\nonumber\\ & 
= - \int_{\Real^d}\left( u^2 V_{-k}\Lz V_k +2 V_{-k}u \grz V_k \cdot \grz u \right)dx + \int_{\Real^d} u V_{-k}V_k\left(-\Lz+V_0\right) u\; dx\nonumber \\ &
= -\int_{\Real^d} u^2 V_{-k}\Lz V_k dx + \int_{\Real^d} u^2 \sum_{j=1}^d \frac{\partial}{\partial x_j}\left(V_{-k}\frac{\partial}{\partial x_j}V_k\right)\;dx
+\lambda_0^2 \int_{\Real^d} u^2 V_{-k}V_k \;dx \nonumber \\
& = \int_{\Real^d} u^2\grz V_{-k}\cdot \grz V_k \;dx+ \lambda_0^2 \int_{\Real^d} u^2 V_{-k}V_k\;dx.
\end{align}

Using this, we find
\begin{multline*}
\int_{\Real^d} \left((z^2-k^2)V_{-k}V_ku^2-  u V_{-k}\left(-\Lz+V_0\right)(V_k u)\right)dx\\
=  -\int_{\Real^d} \left( ((k^2+\lambda_0^2-z^2)V_{-k}V_k+ \grz V_{-k}\cdot \grz V_k)u^2\right) dx
\end{multline*}
completing
the proof.
\end{proof}

The proof of the next lemma uses some of the same ideas
as that of Proposition \ref{p:firstorderapprox}.  This result 
will be used in the proof of Theorem \ref{thm:polesequence}.
\begin{lemma}\label{l:symmetricapprox}
Suppose $V\in C_c^\infty(X;\Real)$.
Let $\lambda_0\in i\Real$ be a simple pole of $\RVzz(\lambda)$
with residue of rank $1$.  Let $ M>|\lambda_0|$ and $N\in \Natural$,
and suppose
$\RVzz(\lambda)-\Xi(\RVzz(\lambda),\lambda_0)$
 is analytic  for $|\lambda-\lambda_0|\leq \epsilon$.
Then if $\chi\in C_c^\infty(X;\Real)$ is independent of $\theta$ and satisfies
$V\chi=V$, there is an $s=s(N)\in \Natural$ and an $A_N=A_N(\tau_l, l):H_{(0,s)}(X)\rightarrow L^2(X)$
so that for $l$ sufficiently large
\begin{equation}\label{eq:ANprop1}
\left\| \pr_l(I+\Vd\RVzz^{\reg}(\zeta)\chi)^{-1}\Vd\pr_l-A_N(\tau_l(\zeta),l)\right\|_{H_{(0,s)}(X)\rightarrow L^2(X)}=O(l^{-N})
\end{equation}
uniformly for $\zeta \in \overline{\Dl}(\lambda_0,\epsilon)$.
Moreover, $A_N(z,l)$ depends analytically on  $z$ in the set 
$\{ z\in \Complex:\; |z-\lambda_0|\leq\epsilon\}$ 
and
if $z\in i\Real$, then  $A_N(z,l)$ is symmetric on $C_c^\infty(X)\subset L^2(X)$. 
Furthermore,
$\|\pr_{l\pm} A_N\pr_{l\mp}\|_{H_{(0,s)}(X)\rightarrow L^2(X)}=O(l^{-N})$ for any $N$.
\end{lemma}
\begin{proof}
By Lemma \ref{l:Rreg}, if $j>2N$, then on $\overline{\Dl}(\lambda_0,\epsilon)$ 
 $\| (\Vd\RVz^{\reg}(\zeta)\chi)^{j}\|_{L^2(X)\rightarrow L^2(X)}=O(l^{-N})$.
Thus
\begin{equation}\label{eq:simp1} \left\| (I+\Vd\RVz^{\reg}(\zeta)\chi)^{-1}-\sum_{j=0}^{2N}
(-\Vd\RVz^{\reg}(\zeta)\chi)^{j}\right\|_{L^2(X)\rightarrow L^2(X)}
= O(l^{-N}).\end{equation}

Now we write, for $l$ sufficiently large,
\begin{equation}
\RVz^{\reg}=\RVz^{\reg}\pr_l+  \RVz(I-\pr_l).
\end{equation}
From our assumptions on $V_0$ and the pole of $\RVzz$ at $\lambda_0$, there 
is a $u\in C^\infty(\Real^d;\Real)$ 
so that  
$\RVzz(\lambda)-\frac{i}{\lambda-\lambda_0}u\otimes u$
 is analytic  for $|\lambda-\lambda_0|\leq \epsilon$. Then  for $l$ 
sufficiently large
$$\RVz^{\reg}(\zeta)\pr_l=\RVz^{\reg}(\zeta;\lambda_0,l)\pr_l= \RVzz(\tau_l(\zeta))\pr_l- \frac{i}{\tau_l(\zeta)-\lambda_0}(u\otimes 
u)\pr_l.$$
If $\tau_l=\tau_l(\zeta)\in i\Real$ and $\zeta \in \overline{\Dl}(\lambda_0,\epsilon)$, the operator $\chi \RVz^{\reg}(\zeta)\pr_l\chi $  is symmetric on $C_c^\infty(X)$.
On the other hand, for $k\not = l$, writing $\tau_k$ for $\tau_k(\zeta)$
and using Lemma \ref{l:RVzzexpand}
\begin{align}\label{eq:ex1}
\chi \RVz \pr_k \chi & = \chi \RVzz(\tau_k) \pr_k\chi  \nonumber \\ & 
= -\chi   \sum_{j=1}^{N}\frac{1}{(\tau_l^2+l^2-k^2)^j }\left(
-\Lz+V_0\right)^{j-1}\pr_k \chi \nonumber \\
 &\;\;\; \hspace{3mm} +\chi  \frac{1}{(\tau_l^2+l^2-k^2)^{N} }\RVz(\tau_k)
\left(-\Lz +V_0\right)^N\pr_k\chi.
\end{align}
If $\tau_l^2\in \Real$, then $\chi \frac{1}{(\tau_l^2+l^2-k^2)^j }\left(
-\Lz+ V_0\right)^{j-1}\pr_k\chi$ is symmetric on $C_c^\infty(X)$.
Set \begin{equation}
T_N=T_N(\tau_l,l)= \RVzz^{\reg}(\tau_l)\pr_l 
-\sum_{k\not =l} \sum_{j=1}^{N}\frac{1}{(\tau_l^2+l^2-k^2)^j }
\left(-\Lz +V_0\right)^{j-1} \pr_k.
\end{equation}
Note that $T_N$
is an 
analytic operator-valued function of $\tau_l$ for $\zeta \in\overline{ \Dl}(\lambda_0,\epsilon)$ where
$|\tau_l -\lambda_0|\leq\epsilon$.
Using (\ref{eq:ex1}) 
$$\|\chi(\RVz^{\reg}-T_N)\chi\|_{H_{(0,2N+t)}(X)\rightarrow H_{(0,t)}(X)}=O(l^{-N}),$$
if $|\tau_l-\lambda_0|\leq \epsilon$, and $\chi T_N(\tau_l,l)\chi$ is symmetric on $C_c^\infty(X)$ if $\tau_l\in i\Real$.
Moreover, by (\ref{eq:simp1})
$$ \left\| (I+\Vd\RVz^{\reg}(\zeta_l(\tau_l))\chi)^{-1} -\sum_{j=0}^{2N}
(-\Vd T_N(\tau_l,l))^j\chi\right\|_{H_{(0,s(N))}\rightarrow L^2}=O(l^{-N})$$ 
if $s(N)\geq 4N^2$.  
Thus if we define
\begin{equation}\label{eq:AN}
A_N=A_N(\tau_l,l)=\pr_l \sum_{j=0}^{2N}(-\Vd T_N)^j\Vd \pr_l
\end{equation}
then $A_N$ satisfies (\ref{eq:ANprop1}), $A_N$ is an analytic function
of $\tau_l$ if $|\tau_l-\lambda_0|\leq \epsilon$, and $A_N(\tau_l,l)$ is symmetric on $C_c^\infty(X)$
if $\tau_l\in i\Real$.


To show that $\|\pr_{l\pm }A_N\pr_{l\mp}\|_{H_{(0,s)}\rightarrow L^2}=O(l^{-N})$, consider a term
$\pr_{l+} (\Vd T_N)^j\Vd \pr_{l-}$.  We write
$$\pr_{l+}(\Vd T_N)^j\Vd \pr_{l-}= \sum_{\mathclap{\substack{m_1+m_2+...+m_{j+1}=2l\\
 m_k\not = 0}}}
V_{m_1}e^{im_1\theta}T_N V_{m_2}e^{im_2\theta}T_N\cdot \cdot \cdot V_{m_j}e^{im_j\theta}T_NV_{m_{j+1}}e^{im_{j+1}\theta}\pr_{l-}.$$
Thus we see that at least one $m_n$ must have absolute value at least 
$2l/(j+1)$.  
Since $\|V_m\|_{C^r}=O(m^{-p})$ for any fixed $r$, any $p$, we obtain
$$\|\pr_{l+}(\Vd T_N)^j\Vd \pr_{l-}\|_{H_{(0,s)}\rightarrow L^2}=O(l^{-N})$$ for some
sufficiently large $s$.  Thus the result for $\pr_{l+ }A_N\pr_{l-}$ follows
from our expression (\ref{eq:AN}) for $A_N$. The result for $\pr_{l- }A_N\pr_{l+}$
follows similarly.
\end{proof}

\section{Proofs of smooth case of Theorem \ref{thm:polefree} and Theorem
\ref{thm:0}}
The first application of our results in the previous section is to 
improve the localization of the resonances when $V\in C_c^\infty(X)$.

\vspace{2mm}
\noindent
{\em Proof of Theorem \ref{thm:polefree} for $V\in C_c^\infty(X)$.}  Let $\lambda_j\in \Lambda_\rho$ and choose $\epsilon>0$
so that there are no poles of $\RVz(\lambda)$ in $0<|\lambda-\lambda_j|\leq \epsilon$.  
We will show that there is a $C_j>0$ so that there are no poles of $R_V(\zeta)$ in $\zeta \in D_l(\lambda_j,\epsilon)$ with
$|\tau_l(\zeta)-\lambda_j|>C_j l^{-2/(\mvzz(\lambda_j))}$ when $l$ is sufficiently large. 

Choose $\chi \in C_c^\infty(X)$ so that $\chi V=V$ and $\chi$ is independent of $\theta$.
As previously, if $l$ is sufficiently large, 
$$
\RVz^{\reg}(\zeta)=\RVz^{\reg}(\zeta;\lambda_j,l)= \RVz(\zeta)-\Xi(\RVzz,\lambda_j)|_{\lambda=\tau_l(\zeta)}\pr_l$$
and note that $\RVz^{\reg}(\zeta;\lambda_j,l)$ is analytic on $\overline{\Dl}(\lambda_j,\epsilon)$.
By (\ref{eq:V0asmodel}), any poles of $R_V(\zeta)$ in $\Dl(\lambda_j,\epsilon)$
are points at which
$I+\pr_l(I+\Vd \RVz^{\reg}(\zeta)\chi)^{-1}\Vd \Xi(\RVzz,\lambda_j)\chi \pr_l$ has nontrivial null space.

 Using the smoothness of $V$, for any fixed $s\in \Natural$ there is a constant 
$C>0$ (depending on $s$, $V_0$, $\lambda_j$)
\begin{equation}
\label{eq:singest}
\left \|\Vd  \Xi(\RVzz,\lambda_j)|_{\lambda=\tau_l(\zeta)}\chi \pr_l \right\|_{L^2(X)\rightarrow H_{(0,s)}(X)}\leq \frac{C}
{|\tau_l(\zeta)-\lambda_j|^{\mvzz(\lambda_j)}},
\end{equation}
\cite[Theorems 2.5, 2.7, 3.9, 3.17]{DyZw}.  Thus on $D_l(\lambda_j,\epsilon)$,
for $l$ sufficiently large by Proposition \ref{p:firstorderapprox}
$$\left\| \pr_l (I+\Vd \RVz^{\reg}(\zeta)\chi)^{-1}\Vd \Xi(\RVzz,\lambda_j)_{\lambda=\tau_l(\zeta)}\chi \pr_l \right\|_{L^2(X)\rightarrow L^2(X)} 
\leq \frac{C}{l^2 |\tau_l(\zeta)-\lambda_j|^{\mvzz(\lambda_j)}}.$$
for some $C$.  Thus there is a $C_j>0$ so that if $\zeta\in D_l(\lambda_j,\epsilon)$ and 
$|\tau_l(\zeta)-\lambda_j|>C_j l^{-2/\mvzz(\lambda_j)}$, then $I+\pr_l(I+\Vd \RVz^{\reg}(\zeta)\chi)^{-1}\Vd \Xi(\RVzz,\lambda_j)\pr_l$ 
is invertible, and $\zeta$ is not a resonance.

 Since $\lambda_j\in\Lambda_\rho$ is arbitrary, $\Lambda_\rho$ contains only finitely many 
elements and we have already
proved the theorem for the case of $L^\infty$ potential $V$, this  suffices to prove the smooth version of the theorem.
\qed

\vspace{2mm}

The proof of the smooth case of Theorem \ref{thm:0} is almost
identical, given our earlier results.

\vspace{2mm}
\noindent 
{\em Proof of Theorem \ref{thm:0} for $V\in C_c^\infty(X)$.}
Recall that we have already proved the $L^\infty$ case of this theorem.
Thus, the proof follows just as in the  proof of the smooth case of 
Theorem \ref{thm:polefree}, except that the estimate (\ref{eq:singest})
is replaced by 
$$\|\Vd \Xi (\RVzz,0)|_{\lambda=\tau_l(\zeta)}\chi \pr_l\|_{L^2(X)\rightarrow H_{(0,s)}(X)}\leq \frac{C}{|\tau_l(\zeta)|^r}.$$
\qed

\section{Proofs of Theorems \ref{thm:leadcorrection} and \ref{thm:polesequence}}
\label{s:smoothresults}
We prove Theorems \ref{thm:leadcorrection} and \ref{thm:polesequence}
 in this section, using results of Section \ref{s:smoothprelims}.  
We begin with a preliminary lemma.

\begin{lemma}\label{l:MsandRVreg} Let $\lambda_0$ be a pole of $\RVzz$,
and set $\RVz^{\reg}(\zeta)=\RVz^{\reg}(\zeta;\lambda_0,l)$.  
Let $\chi \in C_c^\infty(X)$ be independent of $\theta$ and satisfy $\chi V=V$,
with $\chi$ nontrivial.  Suppose $\RVzz(\lambda)$ is analytic for
$0<|\lambda -\lambda_0|\leq \epsilon.$
Then there 
is an $L>0$ so that for $l>L$, if $\zeta_0\in \Dl(\lambda_0,\epsilon)$,
then
$$M(I+\Vd \RVz (\zeta)\chi, \zeta_0)=
M(I+(I+\Vd \RVz^{\reg}(\zeta) \chi)^{-1}\Vd
\Xi(\RVzz(\lambda),\lambda_0)_{\upharpoonright \lambda=\tau_l(\zeta)}\pr_l,\zeta_0).$$
\end{lemma}
\begin{proof}  
By Lemma \ref{l:Rreg}, there is an $L>0$ so that
$I+\Vd \RVz^{\reg}(\zeta) \chi$ is invertible on $\Dl(\lambda_0,\epsilon)$ for 
$l>L$.  
Then if $l>L$ and $\zeta_0\in\Dl(\lambda_0,\epsilon)$
\begin{align*}
& M(I+\Vd \RVz \chi, \zeta_0)\\&= M\left((I+\Vd \RVz^{\reg}(\zeta) \chi)\left(I+(I+\Vd \RVz^{\reg}(\zeta) \chi)^{-1}\Vd
\Xi(\RVzz(\lambda),\lambda_0)_{\upharpoonright\lambda=\tau_l(\zeta)}\pr_l\right),\zeta_0\right) \\& =
M\left(I+(I+\Vd \RVz^{\reg}(\zeta) \chi)^{-1}\Vd
\Xi(\RVzz(\lambda),\lambda_0)_{\upharpoonright\lambda=\tau_l(\zeta)}\pr_l, \zeta_0\right)
\end{align*}
where the second equality uses Lemma \ref{l:Mproduct}.  
\end{proof}

Given $f\in C_c^\infty(\Rd)$, define $h_{\pm l}\in C_c^\infty(X)$
by $h_{\pm l}(x,\theta)=\frac{1}{\sqrt{2\pi}}f(x)e^{\pm i l \theta}$.
For $z_0\in \Complex$ and an operator $A: H_{(0,s)}(X)\rightarrow L^2(X)$ set
\begin{equation}\label{eq:mcddef}\mcd_A(z)= \det\left(I + \frac{i}{z-z_0}(Ah_l\otimes h_{-l}+Ah_{-l}\otimes h_l)\right).\end{equation}
Here $\det$ is the Fredholm determinant.  In this special case it 
is easily calculated to be
\begin{multline}\label{eq:specialdetermin}
\mcd_A(z)= \frac{1}{(z-z_0)^2}\left\{ \left(z-z_0+i\int_X h_{-l}(Ah_l)\right)
\left(z-z_0+i\int_X h_{l}(Ah_{-l})\right)\right.\\\left.+ \int_X h_{-l}(Ah_{-l})
\int_X h_{l}(Ah_{l})\right\}.
\end{multline}

\begin{prop}\label{p:approx2}
Let $z_0\in \Complex$, $\epsilon>0$, and set 
$U_\epsilon=\{z\in \Complex: \; |z-z_0|<\epsilon\}$.  Suppose there are $L_1,
m_0\geq 1/2$ 
and $s\in \Natural$
so that for $l>L_1$, $l\in \Natural$ and $z\in U_\epsilon$
 there are linear operators
$S_l=S_l(z), \; T_l=T_l(z):H_{(0,s)}(X)\rightarrow L^2(X)$ which are operator-valued
functions analytic on $U_\epsilon$ satisfying:
\begin{itemize}
\item 
$\sup_{z\in U_\epsilon}\|\pr_lS_l(z)\pr_l-T_l(z)\pr_l\|_{H_{(0,s)}(X)\rightarrow L^2(X)}=O(l^{-m_0})$
\item $T_l(z)\pr_l=\pr_{l+}T_l(z)\pr_{l+}+\pr_{l-}T_l(z)\pr_{l-}, $ $ \sup_{z\in U_\epsilon}\|T_l(z)\|_{H_{(0,s)}(X)\rightarrow L^2(X)}=O(l^{-1/2})$.
\end{itemize}
Then given $f\in C_c^\infty(\Rd)$, for $l$ sufficiently large
the functions $(z-z_0)^2 \mcd_{S_l}(z)$ and $(z-z_0)^2 \mcd_{T_l}(z)$ have exactly two zeros, counted with 
multiplicity, in $U_\epsilon$, and they lie in $U_{\epsilon/2}$.  
Moreover, there is
a labeling of these two sets of zeros as $z_{S_l\pm}$, $z_{T_l\pm}$,
so that $|z_{S_l\pm}-z_{T_l\pm}|=O(l^{-m_0})$. 
\end{prop}
\begin{proof}
By translating if necessary, we may assume $z_0=0$.

Our assumptions on $T_l$ imply that 
$F_{\pm}(z)=F_\pm(z;l)\defeq z+i\int_X\big( h_{\mp l}(T_l(z)h_{\pm l})\big)$ is 
analytic on $U_\epsilon$ and satisfies
$F_{\pm}(z)=z+O(l^{-1/2})$ uniformly on $U_\epsilon$. 
 Applying Rouch\'e's Theorem to the pair $F_{\pm}(z)$
 and the function $z$, we see that $F_{\pm}$ has, for $l$ sufficiently 
large, exactly one zero in the set $U_{\epsilon/4}$, and no
zeros in $U_\epsilon\setminus U_{\epsilon/4}$.  We label this
zero as $z_{T_l\pm}$.  Since $\int_X h_{\pm l}(T_l h_{\pm l})=0$, 
$z^2\mcd_{T_l}(z)=F_+(z)F_-(z)$, and $z_{T_l \pm}$ are the zeros of $z^2\mcd_{T_l}$.


 We write
\begin{equation}\label{eq:factor}
F_\pm(z;l)= z+i\int_X h_{\mp l}(T_l(z)h_{\pm l})= (z-z_{T_l\pm})\varphi_{\pm}(z;l),
\end{equation}
with $\varphi_\pm$ analytic on $U_\epsilon$ for $l$ sufficiently large. 
An application of the maximum principle 
shows that there is a $C>0$ independent of $l$ so that for
$l$ sufficiently large 
\begin{equation}\label{eq:varphiest}
1/C\leq |\varphi_{\pm}(z;l)|\leq C\; \text{
for all $z\in U_{3\epsilon/4}$}.
\end{equation}

Next consider the intermediary 
$$G_\pm(z)=G_\pm(z;l)\defeq z+i\int_X h_{\mp l}(S_l(z)h_{\pm l}) = 
z+i\int_X h_{\mp l}(T_l(z)h_{\pm l}) +O(l^{-m_0}).$$
 Our estimate $G_\pm-F_\pm=O(l^{-m_0})$, (\ref{eq:factor}) and (\ref{eq:varphiest})
allow an application of Rouch\'e's Theorem to the pair  
$F_\pm$, $G_\pm$ on a disk
with center $z_{T_l\pm}$ and radius  $c_0 l^{-m_0}$ for 
an appropriate choice of $c_0>0$ and for $l$ sufficiently large.  This
shows that for $l$ sufficiently large, $G_\pm$ has 
 exactly one zero (counting multiplicity) in $U_{\epsilon/3}$.  We label
this zero $z_{I,l,\pm}$ (the ``I'' here stands for intermediate, as 
this is an intermediate step).  We have shown $|z_{I,l,\pm}-z_{T_l,\pm}|=O(l^{-m_0})$.
As before, by the maximum principle we may write
\begin{equation}
G_\pm(z;l)=(z-z_{I,l,\pm})\varphi_{I\pm}(z;l),\; \text{with}\; 1/C\leq  |\varphi_{I\pm}(z;l)|\leq C, \;\text{for all}\; z\in U_{3\epsilon/4}
\end{equation}for some constant $C$ independent of $l$, and for $l$ sufficiently
large.

Now consider $z^2\mcd_{S_l}(z)$.  By our assumptions on 
$S_l$ and $T_l$, 
$$z^2\mcd_{S_l}(z)= G_+(z)G_-(z)+O(l^{-2m_0})=
(z-z_{I,l,+})(z-z_{I,l,-})\varphi_{I+}(z)\varphi_{I-}(z)+O(l^{-2m_0}) .$$
Thus we can apply Rouch\'e's Theorem again, this time to the pair $z^2\mcd_{S_l}(z)$ and $G_+(z;l)G_-(z;l)$ at a distance proportional to $l^{-m_0}$ of $z_{I,l,\pm}$, proving
the proposition.
\end{proof}

We apply this proposition in the proof of Theorem \ref{thm:leadcorrection}.

\vspace{2mm}
\noindent
{\em Proof of Theorem \ref{thm:leadcorrection}.}
We assume  that $\Vd\not \equiv 0$, since
otherwise there is nothing to prove.  Choose $\chi \in C_c^\infty(X)$
with $\chi V=V$,
and $\chi $ independent of $\theta$.

Let $\RVz^{\reg}(\zeta)=\RVz^{\reg}(\zeta;\lambda_0,l)$, and let $\epsilon,L>0$ 
be as in 
Lemma \ref{l:MsandRVreg}. 
 For $l>L$ the function
$$F_l(\zeta)\defeq (\tau_l(\zeta)-\lambda_0)^2\det(I+(I+\Vd \RVz^{\reg}(\zeta) \chi)^{-1}\Vd
\Xi(\RVzz,\lambda_0)_{\lambda=\tau_l(\zeta)}\pr_l)$$
is analytic on $\Dl(\lambda_0,\epsilon)$.
Moreover, the order of vanishing of $F_l$  at $\zeta_0 \in \Dl(\lambda_0,\epsilon)$ is given by 
$$ M(I+(I+\Vd \RVz^{\reg}(\zeta) \chi)^{-1}\Vd
\Xi(\RVzz,\lambda_0)_{\lambda=\tau_l(\zeta)}\pr_l,\zeta_0)+
m_{V_0}(\zeta_0)$$ see \cite[Theorem 5.1]{go-si}.
Note that for $\zeta_0\in \Dl(\lambda_0,\epsilon)$ and $l $ sufficiently
large, $m_{V_0}(\zeta_0)\not = 0$ if and only if $\tau_l(\zeta_0)=\lambda_0$.
For $\lambda_0\not =0$, combining this with Lemmas \ref{l:MsandRVreg} and  \ref{l:Mvdrv0},
we see that the poles of $\RV$ in $\Dl(\lambda_0,\epsilon)$ are, for $l>L$, given by
the zeros of $F_l$, and the multiplicities agree.  If $\lambda_0=0$, the same
is true, but as in the proof of Theorem \ref{thm:0} we use
Lemmas \ref{l:multat0}, \ref{l:residues}, and \ref{l:poleorder1}.

To prove the theorem, we will apply Proposition \ref{p:approx2} with 
the following choices: $z=\tau_l(\zeta)$, $z_0=\lambda_0$, $f(x)=\chi(x)u(x)$ so that
$h_{\pm l}(x,\theta)=\frac{1}{\sqrt{2\pi}}\chi(x) u(x)e^{\pm i l \theta}$, 
\begin{align*}S_l& =S_l(z)=(I+\Vd \RV^{\reg}(\zeta_l(z)))^{-1}\Vd\pr_l,\\
T_l&=T_l(z)= \frac{-1}{l^2}\sum_{k\not = 0} \left( \frac{z^2-k^2}{4k^2}V_{-k}V_k-
\frac{1}{4k^2} V_{-k}\left(-\Lz +V_0\right)V_k \right) \pr_l,
\end{align*} and $s=6$.  
By Proposition \ref{p:firstorderapprox} we have,
in the notation of Proposition \ref{p:approx2}, $m_0=3$. Note
that using the coordinate $z=\tau_l(\zeta)$, $F_l(\zeta_l(z))=(z-\lambda_0)^2\mcd_{S_l}(z)$, where $\mcd_{S_l}$ is as defined via (\ref{eq:mcddef}).

The function $(z-\lambda_0)^2\mcd_{T_l}(z)$ has a single zero of multiplicity $2$ in $U_\epsilon$, and by Lemma \ref{l:ucomp}   this is the  zero of 
$$z-\lambda_0+\frac{i}{4l^2}\sum_{k\not = 0} \int_\Real \left(\frac{k^2+\lambda_0^2-z^2}{k^2}u^2  V_{-k}V_k+ \frac{u^2 \grz V_{-k}\cdot \grz V_k}{k^2}\right)$$
near $z=\lambda_0$.   This zero is given by 
$$z_{T_l\pm}= \lambda_0-\frac{i}{4l^2}\sum_{k\not = 0} \int_\Real \left(u^2  V_{-k}V_k+ \frac{u^2 \grz V_{-k}\cdot \grz V_k }{k^2}\right)+O(l^{-4}).$$

By Proposition \ref{p:approx2}, the zeros of $(z-\lambda_0)^2\mcd_{S_l}(z)$ in $U_\epsilon$ 
are  
within $O(l^{-m_0})=O(l^{-3})$ of the zero (of multiplicity $2$) of $(z-\lambda_0)^2\mcd_{T_l}(z)$
in $U_\epsilon$, thus completing the proof.
\qed
\vspace{2mm}

The proof of Theorem \ref{thm:polesequence} is similar.

\vspace{2mm}
\noindent 
{\em Proof of Theorem \ref{thm:polesequence}.}
We prove the theorem by showing that for any $N\in \Natural$ there is an
$\epsilon>0$ so  for $l\in \Natural$ sufficiently large if
 $\zsharp_l\in D_l(\lambda_0,\epsilon)$ and $\zsharp_l$
 is a pole of $R_V(\zeta)$, then $\Re \tau_l(\zsharp_l)=O(l^{-N})$.

Choose $\chi \in C_c^\infty(X;\Real)$ so that $\chi V=V$ and $\chi$ is independent of $\theta$.
Choose $\epsilon,\; L>0$, as in Lemma \ref{l:MsandRVreg}.  
Let $u \in C^\infty(\Rd)$ be such that 
$\RVzz(\lambda)-\frac{i}{\lambda-\lambda_0}u\otimes u$ is analytic 
for $\lambda$ 
near $\lambda_0$.
To prove the theorem, we apply Proposition \ref{p:approx2} in a way 
very similar to the proof of Theorem \ref{thm:leadcorrection}.
We make the following choices:
$z=\tau_l(\zeta)$, $z_0=\lambda_0$, $h_{\pm l}(x,\theta)=\frac{1}{\sqrt{2\pi}}\chi(x) u(x)e^{\pm i l \theta}$, and
 $S_l=S_l(z)=(I+\Vd \RV^{\reg}(\zeta_l(z))\chi)^{-1}\Vd\pr_l$ where $\RV^{\reg}(\zeta)=\RV^{\reg}(\zeta;\lambda_0,l)$. For $l$ sufficiently
large, $S_l$ is analytic on $U_\epsilon$.   
 Let $A_N=A_N(z,l)$ be the operator of Lemma \ref{l:symmetricapprox}, and set
$T_l=T_l(z;N)=\pr_{l+}A_N\pr_{l+}+\pr_{l-}A_N\pr_{l-}$.   By Lemma \ref{l:symmetricapprox}, there is an 
$s\in \Natural$ so that 
$$\|\pr_l S_l (z)\pr_l -T_l(z)\|_{H_{(0,s)}(X)\rightarrow L^2(X)}=O(l^{-N})$$
uniformly for $z\in U_\epsilon$.  Thus for our application of Proposition \ref{p:approx2}  we have $m_0=N$.

As in the proof of Theorem \ref{thm:leadcorrection}, the poles of $\RV$ in $\Dl(\lambda_0,\epsilon)$ 
are determined by the zeros of $(z-\lambda_0)^2\mcd_{S_l}(z)$ in $U_\epsilon$, using $U_\epsilon\ni z=\tau_l(\zeta)$.
By Proposition \ref{p:approx2}, these zeros are approximated by those of $(z-\lambda_0)^2\mcd_{T_l}(z)$ in $U_\epsilon$,
with an error which is $O(l^{-N})$.  We compete the proof by showing that for $l$ sufficiently large the 
zeros of $\mcd_{T_l}(z)$ in $U_\epsilon$ lie on the imaginary axis.

Set
$a_\pm(z;l)\defeq\int_Xh_{\mp l} (T_l(z)h_{\pm l}) = \int_X \overline{h_{\pm l}} (T_l(z)h_{\pm l})$.
From Lemma \ref{l:symmetricapprox} and the definition of $T_l$,  if $z\in U_{\epsilon} \cap i\Real$, then
$T_l(iz)$ is symmetric on $C_c^\infty(X)\subset L^2(X)$.  
In particular, this implies
 that if $z\in i\Real\cap U_\epsilon$ then $a_\pm(z;l)\in \Real$.
Since $a_\pm(z;l)$ is analytic for $z\in U_\epsilon$ 
and is real-valued for $z\in i\Real\cap U_\epsilon$, we must have
\begin{equation}\label{eq:symmetry}
a_{\pm}(z;l)=\overline{a}_{\pm}(-\overline{z};l)\; \text{ for }\; z\in U_\epsilon.
\end{equation}
We remark that since $\lambda_0\in i\Real$, $z\in U_\epsilon$ if and 
only if $-\overline{z}\in U_\epsilon$.

From the proof of Proposition \ref{p:approx2}, the 
zeros of $(z-\lambda_0)^2\mcd_{T_l}(z)$ in $U_\epsilon$ are given by the zeros of
 $z-\lambda_0+i a_\pm(z,l)$ in $U_\epsilon$, and there is, for $l$ sufficiently
large, exactly one such zero for each choice of $\pm$.   We denote
these zeros by $z_{T_l\pm}$, and focus on the zero for
the ``$+$'' sign, $z_{T_l+}$.    Using $\lambda_0\in i\Real$,
\begin{align*}
z_{T_l+}-\lambda_0+ia_+(z_{T_l+};l)&=0\\ &
=\overline{z_{T_l+}-\lambda_0+ia_+(z_{T_l+};l)}
\\ & 
= -\left( -\overline{z_{T_l+}}-\lambda_0+i\overline{a}_+(z_{T_l+};l)\right)\\ &
= -\left( -\overline{z_{T_l+}}-\lambda_0+ia_+(-\overline{z_{T_l+}};l)\right)
\end{align*}
where the last equality uses (\ref{eq:symmetry}).  Hence 
$-\overline{z_{T_l+}}$ is also a zero of $z-\lambda_0+i a_+(z;l)$ in $U_\epsilon$,
and since there is exactly one such zero, it must be that $-\overline{z_{T_l+}}=z_{T_l+}$, and thus $z_{T_l+}\in i\Real$.  The same argument shows $z_{T_l-}\in i\Real$.
\qed

\section{Proof of Theorem \ref{thm:0isrigid}, the resonant uniqueness of $V\equiv 0$ when $d=1$}
Theorem \ref{thm:0isrigid}, a result on the
resonant rigidity of the $0$ potential on $\Real \times \Sphere^1$,
follows rather directly from Theorems 
\ref{thm:poleexist}, \ref{thm:0}, and  \ref{thm:leadcorrection}.

\vspace{2mm}
\noindent{\em Proof of Theorem \ref{thm:0isrigid}.}  Suppose 
$X=\Real\times \Sphere^1$ and 
$V$ is as in Theorem \ref{thm:0isrigid}.
Then by Theorems \ref{thm:poleexist} and \ref{thm:0},
the one-dimensional operator $-\frac{d^2}{dx^2}+V_0$ on $\Real$
must have a resonance at the origin and nowhere else, and 
this resonance
must have multiplicity $1$.   But since $V_0\in L^\infty_c(\Real)$, by well-known
results for one-dimensional Schr\"odinger operators (e.g. \cite{Zw87}),
$V_0\equiv 0$.   

The operator $\Rzz(\lambda)-\frac{i}{2\lambda}1\otimes 1$ is analytic at
the origin.  Using this in Theorem \ref{thm:leadcorrection} along with
the fact that $R_V$ has poles at a sequence of thresholds tending to 
infinity, we
find 
$$ \sum_{k\not = 0}\frac{1}{k^2} \int_\Real
  \left( k^2V_kV_{-k}+ V_k'V_{-k}' \right)(x)dx =0.$$
But since for a real-valued potential $V_{-k}(x)=\overline{V}_k(x)$,
this implies $V_k\equiv 0$ for all $k$, and hence $V\equiv 0$.
\qed

\section{The potential $V(x,\theta)=2\chf(x)\cos\theta$ on 
$\Real \times \Sphere^1$}\label{s:example}
In this section we investigate the resonances near the $l$th threshold of the Schr\"odinger operator with potential
 $V(x,\theta)=2 \chf(x)\cos\theta$ on $X=\Real \times \Sphere^1$.  Here $\chf(x)$ is the characteristic
function of the interval
$I_0=[-1,1]$, so $\chf(x)=1$ if $|x|\leq 1$ and $\chf(x)=0$
if $|x|>1$.  This potential has $V_0\equiv 0$ so that
$\Vd=V$.  Proposition \ref{p:exnearthreshold}
 shows that the resonances nearest the threshold, which correspond to 
perturbations of the pole at the origin for $\Rzz(\lambda)$, are,
for this potential,
localized in a different way than for smooth potentials; compare Theorem
\ref{thm:leadcorrection}.  By Proposition \ref{p:exlogl},
  there is a sense 
in which Theorem \ref{thm:V0iszero} is sharp.  We remark that
some of the computations of this section are reminiscent of computations of
\cite[Section 2]{dro}.

In all of this section, 
$$V(x,\theta)= 2\chf(x)\cos\theta\; \text{and}\; X=\Real \times \Sphere^1.$$

We will use this preliminary lemma.
\begin{lemma}\label{l:Rzzcompose}
For $\lambda, \; \lambda'\in \Complex$, $\lambda \not = \pm \lambda'$,
\begin{multline}
\chf \Rzz(\lambda)\chf\Rzz(\lambda')\chf
= \frac{1}{(\lambda')^2-\lambda^2}\chf\left( \Rzz(\lambda')-\Rzz(\lambda)
\right) \chf \\+ \frac{i}{4\lambda \lambda'(\lambda+\lambda')}e^{i(\lambda+\lambda')}
\left( \phi_\lambda \otimes \phi_{\lambda'}+ \phi_{-\lambda}\otimes 
\phi_{-\lambda'}\right)
\end{multline}
where $$\phi_{\pm \lambda}(x)=e^{\pm i\lambda x}\chf(x).$$
Moreover, if $\tau \in \Complex$, $\tau \not = \pm \lambda $, applying the operator $\chf \Rzz(\tau)$ to the function
$\chf (x)e^{i\lambda x}$ yields
\begin{equation}\label{eq:rzzonchf}
  \left(\chf \Rzz(\tau)\chf e^{i\lambda\bullet}\right)(x)=\chf(x)\left( 
  \frac{1}{\lambda^2-\tau^2}e^{i\lambda x} +\frac{1}{2\tau (\lambda-\tau)}e^{-i\lambda}e^{i\tau(1+x)}+\frac{1}{2\tau (\tau+\lambda)}e^{i \lambda }e^{i\tau(1-x)}\right).
\end{equation}
\end{lemma}
\begin{proof}
The first can be seen, for example, by using (\ref{eq:RzzExplicit}),
the explicit expression for the
Schwartz kernel of $\Rzz$, and evaluating
$\int_{-1}^1 e^{i\lambda|x-x''|+i\lambda'|x''-x'|}dx''$ for $|x|,\;|x'|\leq 1$.
Likewise, (\ref{eq:rzzonchf}) follows from an explicit computation using 
 (\ref{eq:RzzExplicit}).
\end{proof}


\subsection{Resonances near the threshold $\tau_l=0$ for $V(x,\theta)=2\chf(x)\cos \theta$}
Since in this section we concentrate on the resonance near the 
threshold, we work on $B_l(1)$.  
A preliminary step is 
\begin{lemma}\label{l:near0}
  Let $\Rz^{\reg}(\zeta)=\Rz^{\reg}(\zeta;0,l)$.
Then for $l$ sufficiently large, uniformly on $B_l(1)$,
$$\left\| \pr_l \Big((I+ V\Rz^{\reg}(\zeta)\chf)^{-1}V +
V \Rz^{\reg}(\zeta)V + (V\Rz^{\reg}(\zeta))^3V  \Big)\pr_l\right\|=
O(l^{-2}).$$
\end{lemma}
\begin{proof}
Using the Neumann series,
$$(I+ V\Rz^{\reg}(\zeta)\chf)^{-1}V
=\sum_{j=0}^\infty (-V\Rz^{\reg}(\zeta))^{j}V.$$
By Lemma \ref{l:Rreg},
$\| (-V\Rz^{\reg}(\zeta))^j\|=O(l^{-2})$ on $B_l(1)$ if $j\geq 4$ and 
$l$ is sufficiently large.  This ensures the Neumann series for
$(I+ V\Rz^{\reg}(\zeta)\chf)^{-1}$ converges,
and
$$\left\| (I+ V\Rz^{\reg}(\zeta)\chf)^{-1}V -
\sum_{j=0}^3(-V\Rz^{\reg}(\zeta))^{j}V \right\| =O(l^{-2})$$
on $B_l(1)$.

Now we note that our explicit expression for $V$ means that $ \pr_l V \pr_l=0$.
Likewise, it implies that 
$\pr_l (V\Rz^{\reg}(\zeta))^{2}V\pr_l=0$, completing the proof.
\end{proof}

\begin{prop}\label{p:exnearthreshold} For $l$ sufficiently large, the poles of 
$R_V(\zeta)$ in $B_l(1)$ satisfy 
$\tau_l(\zeta)= \frac{1}{4l\sqrt{2l}}(-1-i+e^{i2\sqrt{2l}})+
O(l^{-2})$.
\end{prop}
\begin{proof}
We give a proof similar to that of Theorem \ref{thm:leadcorrection}
using Proposition 
\ref{p:approx2}.

Let $\Rz^{\reg}$
be as in Lemma \ref{l:near0}, and restrict $\zeta$ to $\zeta\in B_l(1)$.  Note $\Rzz(\lambda)-(i/(2\lambda))1\otimes 1$ is regular at $\lambda =0$.
Set $z=\tau_l(\zeta)$, 
 $S_l(z)= (I+V\Rz^{\reg}(\zeta_l(z))\chf)^{-1}V\pr_l$, and   
 $h_{\pm l}(x,\theta)= 
\frac{1}{2\sqrt{\pi}}\chf(x) e^{\pm il\theta}$.  
We use  $\mcd_{S_l}$ 
is as defined by (\ref{eq:mcddef}) and $U_\epsilon$ is as in Proposition 
\ref{p:approx2}. 
Then just as in the proof of 
Theorem \ref{thm:leadcorrection}, the poles of $\RV$ in $B_l(1)$
are identified via $z=\tau_l(\zeta)$ 
with the zeros of $z^2\mcd_{S_l}(z)$ in $U_1$.
Set $z_0=0$ and $T_l= \pr_l\left(-V \Rz^{\reg}(\zeta)V - (V\Rz^{\reg}(\zeta))^3V\right)\pr_l$.  Then by Lemma  \ref{l:near0},
in our application of Proposition \ref{p:approx2}
we can take $s=0$ and $m_0=2$.
We
claim that uniformly for $z\in U_1$
\begin{equation}\label{eq:Tlclaim}
z^2 \mcd_{T_l}(z)= \left(z+\frac{1}{2(2l)^{3/2}}\left( 1-e^{2i\sqrt{2l}}+i\right)
+O(l^{-2})\right)^2.
\end{equation}
Assuming for the moment that (\ref{eq:Tlclaim}) holds, this shows that the two
zeros (when counted with multiplicity) of $z^2\mcd_{T_l}(z)$ in $U_1$ satisfy 
$z= \frac{1}{2(2l)^{3/2}}\left(-1-i+e^{2i\sqrt{2l}}\right)+O(l^{-2})$.  An application 
of Proposition \ref{p:approx2} and Lemma \ref{l:near0} then proves the 
proposition.

We now turn to showing (\ref{eq:Tlclaim}). 
We use
\begin{equation}
  \Rz^{\reg}(\zeta_l(z))V\pr_l=\sum_{\pm} \left(e^{\pm i \theta}\Rzz(\tau_{l+1})+
  e^{\mp i \theta} \Rzz(\tau_{l-1})\right) \chf \pr_{l\pm}
\end{equation}
where $\tau_{l\pm 1}=\tau_{l\pm 1}(\zeta_l(z))$, so that 
\begin{equation}\label{eq:prvpr}
\pr_l V\Rz^{\reg} (\zeta_l(z))V\pr_l= \chf\left( \Rzz(\tau_{l-1})+ \Rzz(\tau_{l+1})\right)\chf \pr_l.
\end{equation} Then using
 (\ref{eq:rzzonchf}) gives
\begin{equation} \label{eq:leadterm}\int_X h_{\mp l} V\Rz^{\reg}(\zeta_l(z)) V h_{\pm l}
= \frac{-i}{2(2l)^{3/2}}(1-e^{2i\sqrt{2l}}) +\frac{1}{2(2l)^{3/2}}+O(l^{-2})
\end{equation}
uniformly on $U_{1}$.
Now note
\begin{equation}\label{eq:redistrib}\int_X h_{\mp l} (V\Rz^{\reg} )^3 Vh_{\pm l} =
\int_X (V\Rz^{\reg}V h_{\mp l}) \left( \chf (\Rz^{\reg}V)^2h_{\pm l} \right).
\end{equation}
By  (\ref{eq:rzzonchf}), $\| (V\Rz^{\reg}V h_{\mp l})\|=O(l^{-1})$
and $\|\chf (\Rz^{\reg}V)^2h_{\pm l}\|=O(l^{-1})$.
Then using the
expression for $\mcd_{T_l}$
as in (\ref{eq:specialdetermin}) and equations (\ref{eq:prvpr}), (\ref{eq:leadterm}) and (\ref{eq:redistrib}) 
completes the proof of (\ref{eq:Tlclaim}). 
\end{proof}

\subsection{Existence of poles of $R_V$ within $\approx \log l$ of the $l$th threshold,
for $V(x,\theta)=2\chf(x)\cos \theta$}
As a point of comparison with Theorem \ref{thm:V0iszero},
for the special case $V(x,\theta)=2\chf(x)\cos \theta$ on $X=\Real\times \Sphere^1$
we consider
the existence of poles of $R_V(\zeta)$ in $D_l(\alpha \log l)$ with 
$|\tau_l(\zeta)|>1$.

Again, we use the coordinate $z=\tau_l(\zeta)$ on $B_l(\alpha \log l)$,
and the functions $\phi_\lambda$ are as defined in Lemma \ref{l:Rzzcompose}.

\begin{lemma}\label{l:fshere}
Let $\alpha>0$ be fixed, and set $z=\tau_l(\zeta)$.  Then for $l$ sufficiently large, uniformly on 
$B_l(\alpha \log l)\setminus B_l(1)$ 
\begin{multline}\left\|\pr_l(I+V\Rz(\zeta)\chf(I-\pr_l))^{-1}V\Rz(\zeta)\chf \pr_l +
  \left(f_+\otimes \phi_z + f_-\otimes \phi_{-z}\right)\pr_l \right. \\
  - \left. \frac{1}{2l^2} \chf \Rzz(z)\chf\pr_l\right\|=
O\left(\frac{1}{l^{5/2}}e^{2(\Im z)_-}\right) +O(l^{-3/2})
\end{multline}
where 
$$f_{\pm}(x)=f_\pm(x,z,l)= \frac{i e^{iz}}{4z}\chf(x)\left(
\frac{e^{i\tau_{l+1}}}{\tau_{l+1}(z+\tau_{l+1})}
\phi_{\pm \tau_{l+1}}
+\frac{e^{i\tau_{l-1}}}{\tau_{l-1}(z+\tau_{l-1})}
\phi_{\pm\tau_{l-1}} \right).$$
\end{lemma}
For notational simplicity, we have written $\tau_{l\pm 1}$ for $\tau_{l\pm1}(\zeta_l(z))$.
\begin{proof}
We use 
$(I+V\Rz(\zeta)\chf(I-\pr_l))^{-1}= \sum_{j=0}^\infty (-V\Rz(\zeta)\chf(I-\pr_l))^{j}$ since $\| V\Rz(\zeta)\chf(I-\pr_l)\|=O(l^{-1/2})$.  This
estimate, and others in this proof, are uniform for $\zeta \in B_l(\alpha \log l)\setminus B_l(1)$.
By Lemma \ref{l:Rzzcompose}, (\ref{eq:RzzExplicit}), and the 
explicit expression for $V$, we 
see that 
$$\| \chf \Rz(\zeta)(I-\pr_l)V \Rz(\zeta)\chf \pr_l\|=
O(e^{2(\Im z) _-}/(l|z|))\; \text{ for $\zeta\in B_l(\alpha \log l)$}$$
for $l$ 
sufficiently large.  Moreover, this same lemma implies
that if $|j-l|\leq 2$, then
$\|\chf (V\Rz(\zeta)(I-\pr_l))^2 \chf \pr_j\|=O(l^{-3/2})$
uniformly on $B_l(\alpha \log l)$.  
This ensures
that 
\begin{multline}
\left \| \left( (I+V\Rz(\zeta)\chf(I-\pr_l))^{-1} -\sum_{j=0}^2(-V\Rz(\zeta)\chf(I-\pr_l))^j\right) 
 V\Rz(\zeta)\chf \pr_l \right\|\\=O\left(\frac{1}{l^{5/2} |z |} e^{2(\Im z)_-}\right) .
\end{multline}
Since, as in the proof of Proposition \ref{p:exnearthreshold},
 $\pr_l V\pr_l=0$ and $\pr_l(V \Rz(I-\pr_l))^2 V\pr_l=0$, it suffices
to use $-\pr_lV\Rz(\zeta)\chf(I-\pr_l)V\Rz(\zeta)\pr_l$ to approximate $\pr_l(I+V\Rz(\zeta)\chf(I-\pr_l))^{-1}V\Rz(\zeta)\chf \pr_l$ with desired accuracy. 

Using Lemma \ref{l:Rzzcompose} and its notation 
\begin{align}\label{eq:firstpower}
& \pr_l V \Rz(\zeta_l(z))(I-\pr_l) V \Rz(\zeta_l(z))\chf \pr_l  \nonumber \\ & 
=\chf \left( \Rzz(\tau_{l+1})\chf \Rzz(z)+\Rzz(\tau_{l-1})\chf\Rzz(z)\right)\chf \pr_l
\nonumber \\&  = \frac{1}{\tau_{l+1}^2-z^2}\chf\left( \Rzz(\tau_{l+1})-\Rzz(z)\right)\chf \pr_l + \frac{ie^{i(z+\tau_{l+1})}}{4z \tau_{l+1}(z+\tau_{l+1})}
\left( \phi_{\tau_{l+1}} \otimes \phi_{z} + \phi_{-\tau_{l+1}}\otimes 
\phi_{-z}\right)\pr_l \nonumber \\ & 
+ \frac{1}{\tau_{l-1}^2-z^2}\chf\left( \Rzz(\tau_{l-1})-\Rzz(z)\right)\chf\pr_l
 +\frac{ie^{i(z+\tau_{l-1})}}{4z \tau_{l-1}(z+\tau_{l-1})}
\left( \phi_{\tau_{l-1}} \otimes \phi_{z}+ \phi_{-\tau_{l-1}}\otimes 
\phi_{-z}\right)\pr_l. \nonumber 
\end{align}
 Note that 
$\left\|\frac{1}{\tau_{l\pm 1}^2-z^2}\chf \Rzz(\tau_{l\pm1})\chf\right\|=O(l^{-3/2})$ and $$\left\| \left( \frac{1}{\tau_{l+1}^2-z^2}+ \frac{1}{\tau_{l-1}^2-z^2} \right) \chf\Rzz(z)\chf -\frac{1}{2l^2}\chf \Rzz(z)\chf \right\| = O(l^{-4}|z|^{-1}e^{2(\Im z)_-}).$$  This gives
\begin{multline}
\pr_l(V\Rz(\zeta_l(z))(I-\pr_l)V\Rzz(\zeta_l(z))\chf\pr_l
=  \frac{ie^{i(z+\tau_{l+1})}}{4z\tau_{l+1}(z+\tau_{l+1})}\left(
\phi_{\tau_{l+1}}\otimes \phi_z+\phi_{-\tau_{l+1}}\otimes \phi_{-z}\right)\pr_l \\
+\frac{ie^{i(z+\tau_{l-1})}}{4z\tau_{l-1}(z+\tau_{l-1})}\left(
\phi_{\tau_{l-1}}\otimes \phi_z+\phi_{-\tau_{l-1}}\otimes \phi_{-z}\right) \pr_l
-\frac{1}{2l^2} \Rzz(z)\pr_l\\
+O_{L^2\rightarrow L^2}\left(\frac{1}{l^{5/2}|z|}e^{2(\Im z)_-}\right) +O_{L^2\rightarrow L^2}(l^{-3/2}).
\end{multline}
\end{proof}
Note that the functions $f_\pm$, $\phi_\pm$ in Lemma \ref{l:fshere}
depend holomorphically on $z$ in the set $\{ z\in \Complex: 1\leq z\leq \alpha \log l\}$.

The function $g_l$ of the next lemma appears in the proof of
 Proposition \ref{p:exlogl},
as its zeros
 approximate the locations
of the poles of $R_V(\zeta)$ away from the threshold
in $B_l(\alpha \log l)$, if $\alpha<1$.
A discussion of the Lambert function can be found, for example, in 
\cite{CGHJK}.  This next lemma is very similar to \cite[Lemma 2.4]{dro}.
\begin{lemma}\label{l:g}
The zeros of 
$$g_l(z)\defeq\left(1-\frac{1}{z8l\sqrt{2l}}e^{2i(\sqrt{2l}+z)}\right)^{2}
- \left(\frac{1}{8lz\sqrt{2l}}\left( ie^{2iz}+
e^{2iz}\right)\right)^{2}$$
are given by 
$z_\nu^{\pm}=z_\nu^{\pm}(l)=\frac{i}{2}\mathcal{W}_\nu\left(\frac{1}{4l\sqrt{2l}}
\left(-ie^{2i\sqrt{2l}}\mp i \pm 1\right)\right)$,
 where $\mathcal{W}_\nu$ is the $\nu$-th branch of the Lambert function.  In particular, we have 
$z^+_1\sim  -(3i/4)\log l$.  Moreover, for $l$ sufficiently large there
is a $r_0>0$ independent of $l$ so that if $w\in \Complex $ and $|w|<r_0$,
then
\begin{equation}\label{eq:glest}
|g_l(z_1^+(l)+w)|\geq 2|w|/3.\end{equation}
\end{lemma}
\begin{proof}
The zeros of $g_l$ are solutions of 
$$1-\frac{1}{z8l\sqrt{2l}}e^{2i(\sqrt{2l}+z)} = \pm \frac{1}{8lz\sqrt{2l}}
\left(
 i e^{2iz} + e^{2iz}\right)$$
and so satisfy
$$ze^{-2iz}= \frac{1}{8l\sqrt{2l}}(e^{2i\sqrt{2l}}\pm 1 \pm i).$$
Solutions of this equation are given by 
$$z_\nu^\pm=\frac{i}{2}\mathcal{W}_\nu\left(\frac{1}{4l\sqrt{2l}}
( -ie^{2i\sqrt{2l}}\mp i\pm 1)\right).$$
From \cite[(4.20)]{CGHJK}, $z_1^+\sim \frac{-3i}{4} \log l$ as $l\rightarrow
\infty$.

To finish the proof, we set $\gamma=1/(8l\sqrt{2l})$ and write
$$g_l(z)= \left(1+\frac{\gamma}{z}e^{2iz}(-e^{2i\sqrt{2l}}-1-i)\right)
\left(1+\frac{\gamma}{z}e^{2iz}(-e^{2i\sqrt{2l}}+1+i)\right)
.$$
Now we evaluate at $z=z_1^++w$, with $w\in \Complex $, $|w|$ small, to find
\begin{align*}
g_l(z^+_1+w)& = 
 \left(1+\frac{z^+_1 e^{2iw} }{z^+_1+w}
\frac{\gamma}{z^+_1e^{-2iz^+_1}}(-e^{2i\sqrt{2l}}-1-i)\right)
\left(1+\frac{z^+_1  e^{2iw}}{z^+_1+w}\frac{\gamma}{z^+_1e^{-2iz^+_1}}
(-e^{2i\sqrt{2l}}+1+i)\right)
\\ & 
= \left(1-\frac{z^+_1 e^{2iw}\ }{z^+_1+w}\right)
\left(1+\frac{z^+_1 e^{2iw}}{z^+_1+w}\frac{-e^{2i\sqrt{2l}}+1+i}{e^{2i\sqrt{2l}}+1+i}\right)
\end{align*}
where for the second equality we have used 
$z_1^+e^{-2iz^+_1}=\gamma (e^{2i\sqrt{2l}}+1+i)$.  
This gives, then, recalling $|z_1^+|\rightarrow \infty$ as 
$l\rightarrow \infty$, 
$$g_l(z^+_1+w)= \left(-2iw+O(|w|/|z_1^+|)+O(|w|^2)\right)
\left( \frac{2(i+1)}{e^{2i\sqrt{2l}}+1+i}+O(|w|)\right)$$
for $|w|$ small.  Then there is a $r_0>0$ independent of $l$ so that for
$l$ sufficiently large and $|w|<r_0$, $|g_l(z^+_1+w)|>2|w|/3$.
\end{proof}

\begin{prop}\label{p:exlogl}
For $V(x,\theta)=2\chf(x)\cos \theta$ and $l$ sufficiently large,
$R_V(\zeta)$ has a pole at a point  $\zeta^+_l\in B_l(7/8 \log l)$ with
$\zeta_{l}^+$ satisfying $\tau_l(\zeta^+_l)= 
\frac{i}{2}\mathcal{W}_1\left(\frac{1}{4l\sqrt{2l}}
\left(ie^{2i\sqrt{2l}}-i + 1\right)\right)+O(l^{-1/2+\epsilon})$ for any
$\epsilon>0$.
\end{prop}
\begin{proof} We  continue to use $z=\tau_l(\zeta)$, and work in a region with $1<|z|<(7/8)\log l$.

  Using Lemma \ref{l:fshere},
  $$ \pr_l(I+V\Rz(\zeta)\chf(I-\pr_l))^{-1}V\Rz(\zeta)\chf \pr_l= F\pr_l+ \frac{1}{2l^2 }\chf \Rzz(z)\chf \pr_l +A$$
  where, with notation from Lemma \ref{l:fshere},
   $$F=F(z,l)=
  -f_+\otimes \phi_z-f_-\otimes \phi_{-z}$$
  and $\|A\|=O(l^{-5/2} e^{(2 \Im z)_-}) + O(l^{-3/2})$ on $B_l((7/8) \log l)\setminus B_1(l)$.  We recall that the poles of $\RV$ in 
  $B_l((7/8) \log l)\setminus B_1(l)$ are the zeros of
  $I + \pr_l(I+V\Rz(\zeta)\chf(I-\pr_l))^{-1}V\Rz(\zeta)\chf \pr_l$ in
  $B_l((7/8) \log l)\setminus B_1(l)$.  We write
  \begin{multline}
    I + \pr_l(I+V\Rz(\zeta)\chf(I-\pr_l))^{-1}V\Rz(\zeta)\chf \pr_l \\
  = \left( I+ \frac{1}{2l^2 }\chf \Rzz(z)\chf \pr_l \right)
  \left( I+ \left( I+ \frac{1}{2l^2 }\chf \Rzz(z)\chf \pr_l \right)^{-1}
  (F\pr_l+A)\right)
  \end{multline}
  since $I+ (1/2l^2)\chf \Rzz(z)\chf \pr_l$ is invertible here.
  For notational convenience, set
  $S=S_l= (I+ (1/2l^2)\chf \Rzz(z)\chf \pr_l)^{-1}$, and note that
  $S= I - (1/2l^2)\chf \Rzz(z)\chf \pr_l +O_{L^2\rightarrow L^2}(l^{-4} e^{4(\Im z)_-})$.
  
  We first  consider the poles
  of $I+S F\pr_l$.  
These poles are given by the zeros of the function
\begin{multline*}
\tilde{\mcd}_l(z)\defeq \det(I+ S F\pr_{l\pm})\\ = 
\left(1-\int_\Real (Sf_+)\phi_z\right)\left(1-\int_\Real (S f_-)\phi_{-z}\right)- 
\left(\int_\Real (S f_-)\phi_z\right)\left(\int_\Real( S f_+)\phi_{-z}\right)
\end{multline*}
with twice the multiplicity.
A computation and a use of the approximations 
$\tau_{l+1}=i\sqrt{2l}+O(l^{-1/2})$ and $\tau_{l-1}=\sqrt{2l}+O(l^{-1/2})$
 show that 
$\tilde{\mcd}_l(z)=g_l(z)+O(l^{-3/2})+O(l^{-2}\log l e^{2(\Im z)_-})$, where
 $g_l$ is the function of Lemma \ref{l:g}.  We note that both $g_l$ and
 $\tilde{\mcd}_l$ 
are analytic in $z$ if $1<|z|<(7/8)\log l$. 
We use
$z_1^+(l)$ is as in Lemma \ref{l:g}. Recalling that $\Im z^+_1\sim -(3/4)\log l$, the estimate (\ref{eq:glest}) combined with
Rouch\'e's Theorem shows that $\tilde{\mcd}_l(z)$ has a zero within 
$O(l^{-1/2+\epsilon})$, any $\epsilon>0$ of 
$z_1^+(l).$    This, in turn, means that 
$(I+SF\pr_l)^{-1}=(I+(I+(1/2l^2)\chf \Rzz(z)\chf)^{-1}F\pr_l)^{-1}$ has a single pole of multiplicity two at a point
satisfying $z=z_1^+(l)+O(l^{-1/2+\epsilon})$.  Moreover, we can find a $c_0=c_0(\epsilon)$ so that 
$\|(I+(I-(1/2l^2)\chf \Rzz(z)\chf)F\pr_l)^{-1}\|=O(l^{1+\epsilon})$  when the distance from $z$
 to the pole 
 is given by $c_0 l^{-1/2+\epsilon}$.

 Now using our estimate on $\|A\|$ we can apply the operator Rouch\'e Theorem to the pair
 $I+SF\pr_l$ and $I+SF\pr_l+SA$, to find that
 $I+SF\pr_l+SA$ has two poles
 (when counted with multiplicity)
 which are, using the $z$ coordinate, within $O(l^{-1/2+\epsilon})$ of $ z_1^+(l)$.
%
\end{proof}

{\bf Acknowledgments.}   The author gratefully acknowledges the 
partial support of a University of Missouri Research Leave and a  Simons Foundation Collaboration Grant for 
Mathematicians.  Thank you to Kiril Datchev, 
Alexis Drouot, Adam Helfer, and Alejandro Uribe
for helpful conversations.

\end{document}